\documentclass[11pt,reqno]{amsart}

\usepackage[a4paper,margin=31mm]{geometry}
\usepackage{amsmath,amssymb,amsthm,mathtools,mathrsfs}
\usepackage{booktabs,enumitem,microtype,xcolor}
\usepackage[colorlinks=true,allcolors=blue]{hyperref}
\usepackage[nameinlink,noabbrev]{cleveref}
\allowdisplaybreaks
\usepackage{tikz}
\numberwithin{equation}{section}

\usepackage{aliascnt}

\newtheorem{theorem}{Theorem}[section]

\newaliascnt{proposition}{theorem}
\newtheorem{proposition}[proposition]{Proposition}
\aliascntresetthe{proposition}

\newaliascnt{lemma}{theorem}
\newtheorem{lemma}[lemma]{Lemma}
\aliascntresetthe{lemma}

\newaliascnt{corollary}{theorem}
\newtheorem{corollary}[corollary]{Corollary}
\aliascntresetthe{corollary}

\theoremstyle{definition}

\newaliascnt{definition}{theorem}
\newtheorem{definition}[definition]{Definition}
\aliascntresetthe{definition}

\newaliascnt{assumption}{theorem}
\newtheorem{assumption}[assumption]{Assumption}
\aliascntresetthe{assumption}

\newaliascnt{example}{theorem}
\newtheorem{example}[example]{Example}
\aliascntresetthe{example}

\theoremstyle{remark}

\newaliascnt{remark}{theorem}

\aliascntresetthe{remark}

\crefname{theorem}{Theorem}{Theorems}
\Crefname{theorem}{Theorem}{Theorems}

\crefname{proposition}{Proposition}{Propositions}
\Crefname{proposition}{Proposition}{Propositions}

\crefname{lemma}{Lemma}{Lemmas}
\Crefname{lemma}{Lemma}{Lemmas}

\crefname{corollary}{Corollary}{Corollaries}
\Crefname{corollary}{Corollary}{Corollaries}

\crefname{definition}{Definition}{Definitions}
\Crefname{definition}{Definition}{Definitions}

\crefname{remark}{Remark}{Remarks}
\Crefname{remark}{Remark}{Remarks}

\newcommand{\N}{\mathbb N}
\newcommand{\Z}{\mathbb Z}
\newcommand{\R}{\mathbb R}
\newcommand{\E}{\mathbb E}
\newcommand{\Prob}{\mathbb P}
\newcommand{\1}{\mathbf 1}
\newcommand{\dd}{\,\mathrm d}
\newcommand{\cC}{\mathcal C}
\newcommand{\cK}{\mathcal K}
\newcommand{\cL}{\mathcal L}
\newcommand{\cM}{\mathcal M}
\newcommand{\cP}{\mathcal P}
\newcommand{\cQ}{\mathcal Q}
\newcommand{\cR}{\mathcal R}
\newcommand{\bK}{\boldsymbol K}
\newcommand{\bk}{\boldsymbol k}
\newcommand{\bDelta}{\boldsymbol\Delta}
\newcommand{\BL}{\mathrm {BL}}
\newcommand{\sur}{\operatorname{sur}}
\newcommand{\spr}{\operatorname{spr}}
\newcommand{\diag}{\operatorname{diag}}
\newcommand{\supp}{\operatorname{supp}}
\newcommand{\Mp}{\mathcal M_{\mathrm p}}
\newcommand{\CW}{\mathsf C}
\newcommand{\WW}{\mathsf W}
\newcommand{\Ald}{\mathrm A}
\newcommand{\TV}{\mathrm {TV}}
\newcommand{\QP}{\hyperref[ass:QP]{\textup{(QP)}}}

\crefname{assumption}{assumption}{assumptions}
\Crefname{assumption}{Assumption}{Assumptions}

\crefname{example}{example}{examples}
\Crefname{example}{Example}{Examples}

\title[Perron--Wright asymptotics]{Perron--Wright asymptotics and critical component processes
for regular inhomogeneous random graphs}
\author{Wentao Hu}
\address{Qiuzhen College, Tsinghua University, 30 Shuangqing Road, Haidian District, Beijing 100084, China}
\email{huwt24@mails.tsinghua.edu.cn}

\author{Wen Sun}
\address{School of Mathematical Sciences, University of Science and Technology of China, Jinzhai 96, 230026 Hefei}
\email{wensun.ustc@gmail.com}


\begin{document}

\begin{abstract}
We extend Wright's enumeration of connected labelled graphs
with fixed surplus to irreducible critical finite type kernels
and regular compact type kernels.
Let $C_t(k,r)$ denote the total weight of connected labelled graphs
with $k$ vertices and surplus $r$. Let $u>0$ satisfy
$Ku=u$ and $\int u\,d\mu=1$, and put
$\alpha=\int u^2\,d\mu$ and $\beta=\int u^3\,d\mu$.
We prove that the rooted-tree radius $\rho$ is positive and real
analytic near $1$, with
\[
  \rho(1)=1,
  \qquad
  \log\rho(t)
  =\frac{\alpha^2}{2\beta}(t-1)^2+O(|t-1|^3),
\]
and that, for fixed $r\in\mathbb Z_{\ge0}$ and $0\le L<\infty$,
\[
  C_t(k,r)
  =
  w_r\beta^{(r-1)/2}e^{-k}\rho(t)^{-k}
  k^{k+3r/2-2}
  \bigl\{1+O_{r,L}(k^{-1/2})\bigr\},
\]
uniformly for $|t-1|\le Lk^{-1/2}$, where $w_r$ is the
classical Wright constant.

At $t_n=1+\lambda n^{-1/3}$, these asymptotics yield local
size and surplus point-process limits for deterministic finite
type models and annealed compact type
Bollob\'as--Janson--Riordan models satisfying projective quadratic
regularity. After rescaling component sizes by
$\beta^{1/3}n^{-2/3}$, the limit is the Aldous process at parameter
$\alpha\beta^{-2/3}\lambda$, and the corresponding type profiles
concentrate at the Perron law $p(dx)=u(x)\mu(dx)$.
For every regular countable critical kernel, we construct graph
sequences with growing finite supports and the same local limits,
and identify the additive shift of the critical window caused by
spectral perturbations of order $n^{-1/3}$.
\end{abstract}

\maketitle
\tableofcontents

\section{Introduction}\label{sec:introduction}

Wright's constants play two exact roles in the critical
Erd\H{o}s--R\'enyi graph: they govern the asymptotic enumeration of
connected labelled graphs with fixed surplus~\cite{Wright1977,Wright1980},
and they appear in the factorial densities of Aldous's limiting size
and surplus process~\cite{Aldous1997,AldousLimic1998,JansonSpencer2007}.
The classical
evolution theory locates the critical window and describes the sizes and
complexity of its components~\cite{ErdosRenyi1960,Bollobas1984,
Luczak1990,LuczakPittelWierman1994,JansonLuczakRucinski2000}.  These
enumerative and probabilistic descriptions are joined by more than a
numerical coincidence: the
Wright enumerator, combined with vertex selection and component
isolation, produces the factorial densities of the Brownian excursion
process~\cite{Spencer1997,Janson2007,JansonSpencer2007,
vanDerHofstadKagerMuller2009}.

The earlier Erd\H{o}s--R\'enyi work~\cite{Sun2026ER} develops this
enumerative route directly for the homogeneous critical graph.  It
starts from exact formulas for ordered component tuples, proves finite
graph factorial moment bounds uniformly in the factorial order, and
uses them to obtain the local marked component process and ordered
\(\ell^2_\downarrow\) convergence.  It provides the homogeneous
prototype for the present work.  The question addressed here is what
replaces the Wright enumerator and its exact amplitude when a constant
edge weight is replaced by a Perron critical inhomogeneous kernel.

For a general kernel, the answer is not obtained by inserting an
effective variance into the Erd\H{o}s--R\'enyi formulas.  The critical
direction is selected by the positive Perron eigenfunction of the mean
operator, candidate components carry nonconstant type profiles, and the
finite population calculation depends on how those types are assigned.
The phase transition for inhomogeneous kernels was developed by
Bollob\'as, Janson and Riordan~\cite{BJR2007}; rank one critical models
were subsequently studied in
\cite{BhamidiHofstadLeeuwaarden2010,BhamidiHofstadLeeuwaarden2012,
vanDerHofstad2013,Turova2013,BhamidiSenWang2017}.
More recent universality theorems give global ordered or metric space
limits in several finite type and graphon settings
\cite{BhamidiBroutinSenWang2021,BaslingkerEtAl2025}.  These results do
not identify the connected graph partition function with kernel weights
at fixed surplus, its exact Perron deformation of Wright's constants, or
the fixed size factorial densities derived from that enumerator.  Those are
the objects developed in this paper.

Let \(K\) be the integral operator associated with a symmetric critical
kernel \(\kappa\), and normalize its positive Perron eigenfunction
\(u\) by \(\int u\,\dd\mu=1\).  Three quantities organize the
enumerative and probabilistic conclusions:
\begin{equation}\label{eq:intro-Perron-quantities}
 p(\dd x)=u(x)\mu(\dd x),\qquad
 \alpha=\int u^2\,\dd\mu,\qquad
 \beta=\int u^3\,\dd\mu.
\end{equation}
We call \(p\) the \emph{Perron law}; it is the limiting type law within
a large critical component.  The second Perron moment \(\alpha\) converts the
scalar critical perturbation, while the third moment \(\beta\) governs
the component size scale, the cubic decay in the critical density, and
the deformation of the Wright amplitude.

\subsection{Main results and contributions}
The paper has four principal contributions.
\begin{enumerate}[label=\textup{\arabic*.},leftmargin=*]

\item \textbf{The Perron--Wright enumerator.}
For fixed size \(k\) and surplus \(r\), the partition function
\(\CW_t(k,r)\) in \eqref{eq:CW-definition} sums the weights of
connected labelled candidate components after averaging their types.
Our first main result, \Cref{thm:Perron-Wright}, constructs a positive function \(\rho\), real analytic near \(1\), with
\begin{equation}\label{eq:intro-critical-radius}
 \rho(1)=1,
 \qquad
 \log\rho(t)=\frac{\alpha^2}{2\beta}(t-1)^2+O(|t-1|^3).
\end{equation}
For fixed \(r,L<\infty\), uniformly when
\(|t-1|\le Lk^{-1/2}\),
\begin{equation}\label{eq:intro-PW-formula}
 \CW_t(k,r)
 =w_r\beta^{(r-1)/2}
 e^{-k}\rho(t)^{-k}k^{\,k+3r/2-2}
 \{1+O_{r,L}(k^{-1/2})\}.
\end{equation}
The radius \(\rho(t)\) is the singular point of the rooted tree system
after the scalar edge perturbation has been written as an external
field.  Thus \eqref{eq:intro-PW-formula} is the natural analytic form of
the enumerator.  Its expansion in \eqref{eq:intro-critical-radius}
gives the more explicit consequence
\begin{equation}\label{eq:intro-PW-gaussian}
 \CW_t(k,r)
 =w_r\beta^{(r-1)/2}
 e^{-k}\exp\left\{-\frac{\alpha^2}{2\beta}k(t-1)^2\right\}
 k^{\,k+3r/2-2}\{1+O_{r,L}(k^{-1/2})\}.
\end{equation}
The theorem is an enumeration result in its own right: its statement
contains no ambient graph size, unselected vertices, or component
isolation probability.  In the single type model at \(t=1\), one has
\(u=\alpha=\beta=1\) and
\(\CW_1(k,r)=e^{-k}c(k,k+r-1)\), so
\eqref{eq:intro-PW-formula} reduces exactly to Wright's formula.

For \(r\ge2\), the new amplitude \(\beta^{(r-1)/2}\) comes from a
specific singular mechanism.  The rooted tree system has a square root
singularity in the Perron direction.  Paths whose internal vertices
carry attached rooted trees then produce a rank one Perron pole in the
path resolvent.  Contracting the three incident Perron factors at a
cubic multigraph kernel vertex gives \(\int u^3\,\dd\mu=\beta\).
Combining these vertex contractions with the singular path residues
leaves the power \((r-1)/2\).  The factor in
\eqref{eq:intro-PW-formula} therefore connects the operator geometry of
the critical kernel with Wright's combinatorial constants.

\item \textbf{An exact bridge to the critical component process.}
The enumerator enters the random graph through the ordered component
identity of \Cref{prop:component-count-formula}.  That identity
separates the connected component weights from vertex selection,
isolation, edge odds, and the finite population residual.  The residual
has different exact forms in the two principal models: the deterministic
finite type calculation samples prescribed type populations without
replacement, whereas the annealed compact type calculation averages
the isolation probability over the unselected random types.  Despite
this difference, both residuals have the same critical asymptotic.  In
the deterministic model, the following form is obtained after critical
recentering and replacement of the empirical Perron quantities by their
limits; the compact assertion uses \QP.  Writing
\(k_\Sigma:=\sum_{i=1}^q k_i\) for the total selected size, the common residual is
\begin{equation}\label{eq:intro-common-residual}
 \exp\left\{
  \frac{\alpha(t_n-1)k_\Sigma^2}{2n}
  -\frac{\beta k_\Sigma^3}{6n^2}
 \right\}\{1+o(1)\}.
\end{equation}
Combining \eqref{eq:intro-PW-gaussian},
\eqref{eq:intro-common-residual}, and Stirling's formula gives the
fixed size factorial asymptotics in \Cref{thm:local-main}.  The same
theorem then identifies the local size and surplus process:
\begin{equation}\label{eq:intro-process-limit}
 \sum_{C\in\cC(G_n)}
 \delta_{(\beta^{1/3}n^{-2/3}|C|,\sur(C))}
 \ \Longrightarrow\
 \Xi^{\Ald}_{\alpha\beta^{-2/3}\lambda}.
\end{equation}
Consequently, \(\beta^{1/3}\) gives the size normalization and
\(\alpha\beta^{-2/3}\) converts the graph parameter \(\lambda\) to
the standard Aldous parameter.  The result covers the deterministic
finite type model and, under \QP, the annealed compact type
BJR model.

\item \textbf{The Perron law and type profiles.}
The external field deformation \(\CW_{t,h}(k,r)\) turns ratios of
component coefficients into exponential moments of the empirical type
measure.  The coefficient bounds in
\Cref{prop:shrinking-field-variation} and their consequences in
\Cref{lem:profile} show that this empirical measure concentrates at
\(p=u\mu\).  They also supply the exponential integrability needed to
average the deterministic and annealed residuals.  As a result, the
marked assertion of \Cref{thm:local-main} attaches the deterministic
mark \(p\) to every atom of the limiting process.  For the compact
model, the annealed quadratic term additionally uses the projective
regularity condition \QP.  This condition is not needed for the
enumerative formula \Cref{thm:Perron-Wright} or for the deterministic
finite type theorem.

\item \textbf{Countable type realizations.}
Starting from a regular critical kernel on countably many types,
\Cref{thm:countable-realization} constructs a single graph sequence
with growing finite supports, integer type counts, and exact critical
normalization for which all the fixed size factorial and marked process
limits hold simultaneously.  The construction is substantive because
the constants in the finite type estimates depend on the truncation
level; the dimension cannot simply be allowed to grow inside those
estimates.  \Cref{thm:countable-unrenormalized} also identifies what
happens without exact spectral normalization: a displacement of the
empirical Perron root on the \(n^{-1/3}\) scale produces the
corresponding deterministic shift of the critical window parameter.
\end{enumerate}

\subsection{Relation to previous work}
The classical relation among sparse connected graph enumeration,
Brownian excursion area, and critical random graph components was
developed in
\cite{Wright1977,Wright1980,Spencer1997,Janson2007,Aldous1997,
JansonSpencer2007}.  The earlier Erd\H{o}s--R\'enyi paper
\cite{Sun2026ER} develops finite \(n\) factorial measure bounds that are
uniform in the factorial order.  The present work extends the fixed surplus
enumerative ingredient used there, together with its exact ordered
component calculation and factorial density argument at each fixed order,
to Perron critical kernels.  The single type enumerator is
an exact specialization at \(t=1\); the two papers use slightly
different, but asymptotically equivalent, finite graph parametrizations
of the critical window.

The globally ordered critical components and their measured metric
space limits in the classical model are described in
\cite{AddarioBerryBroutinGoldschmidt2010,
AddarioBerryBroutinGoldschmidt2012}.  The inhomogeneous literature
addresses a related set of global questions.  The framework of
\cite{BJR2007} establishes the phase transition for general kernels.
Rank one critical behavior is studied in
\cite{BhamidiHofstadLeeuwaarden2010,Turova2013}.  The metric and
universality results in
\cite{BhamidiSenWang2017,BhamidiBroutinSenWang2021}, together with the
\(L^3\) graphon theorem of \cite{BaslingkerEtAl2025}, yield stronger
global ordered or measured metric space limits in settings that overlap
ours.  In the strictly positive finite type setting of
\cite[(17)--(18) and Theorems~4.2--4.3]{BhamidiBroutinSenWang2021},
these constants agree after translating
the Perron normalizations.  If \(h\) is their right Perron vector, then
our normalization is \(u=h/(\boldsymbol\mu^{\mathsf T}h)\), where
\(\boldsymbol\mu=(\mu_a)_a\); symmetry
makes their left vector proportional to
\(\diag(\boldsymbol\mu)h\).  Their two scale constants therefore become
\(\sum_a\mu_au_a^2=\alpha\) and
\(\sum_a\mu_au_a^3=\beta\), while the scalar perturbation direction
gives drift \(\alpha\lambda\) and hence Aldous time
\(\alpha\beta^{-2/3}\lambda\).  Those global convergence theorems do
not, however, give the connected graph partition function with kernel weights \(\CW_t(k,r)\), the factor \(\beta^{(r-1)/2}\), the exact
residual formulas in \Cref{prop:component-count-formula}, or the fixed
size factorial densities of \Cref{thm:local-main}.  Conversely, local
fixed surplus factorial asymptotics do not provide the global tightness
needed for ordered or metric space convergence.  The two approaches
therefore describe different aspects of the same critical regime.
The principal theorem here is therefore not another global convergence
statement.  It is an exact asymptotic enumeration theorem whose
amplitude retains the Perron geometry of the kernel.  The process
limits are applications of this enumerator and of the exact component
identities developed below.

Analytic studies of weighted and multitype graph classes
\cite{dePanafieuRavelomanana2015,dePanafieu2019} provide precedents for
decomposing graphs into trees, paths, and finite kernels.  Andreis and
Veshaj~\cite{AndreisVeshaj2026} study multitype enumeration with
prescribed type and edge counts when the excess grows proportionally
to the vertex count.  Our fixed surplus regime instead retains the
Wright amplitude, its characteristic polynomial power, and the cubic
singular mechanism.  The Perron--Wright theorem isolates the resulting
component partition function so that it can be reused in other models
with the same product form component weights.

\subsection{Proof architecture}

The proof begins with one critically normalized field family.
\Cref{lem:critical-pair,lem:critical-pair-least-branch,lem:normalized-field-tree}
construct the nearby critical pair, select the least rooted tree branch,
and continue it jointly on a common dented domain.  On that branch,
\Cref{lem:uniform-path-determinant,lem:kernel-formula,lem:common-connected-continuation}
establish the path resolvent, the determinant factorization, the exact
finite multigraph kernel expressions, and the continuation of the
connected generating functions.  Their coefficients at the zero field
are extracted in \Cref{prop:base-PW-interface}.  The normalized field
derivative in \Cref{lem:one-mark-variation} yields the local and
shrinking field estimates of
\Cref{prop:shrinking-field-variation}.  Specialization to the scalar
field \(g_t\) in \eqref{eq:scalar-change-of-field} gives the exact radius
formula and its Gaussian expansion in \Cref{thm:Perron-Wright}.
Finally, \Cref{lem:finite-parameter-stability,cor:actual-sequence-PW}
make the critical radius, coefficient, and field estimates uniform
along nearby finite critical models whose parameters vary with \(n\).

The probabilistic part begins in \Cref{sec:profiles}, where the field
estimates are converted into concentration and quadratic integrability
bounds for empirical type profiles.  The estimates used later are
collected in \Cref{lem:profile}.
\Cref{prop:component-count-formula} then combines the connected
component weights with the exact residual factors in the deterministic
and annealed models.  The factorial moment calculation in
\Cref{sec:factorial-process} matches the limiting densities with those
in \Cref{prop:AW-density}.  The local factorial moment and mark
replacement criteria of
\Cref{lem:factorial-criterion,lem:mark-replacement} give the point
process and marked limits in \Cref{thm:local-main}.  Finally,
\Cref{lem:countable-truncation,lem:integer-approximation,lem:countablefixedstage}
control truncation, integer approximation, and each fixed finite type
stage; the diagonal construction then proves
\Cref{thm:countable-realization,thm:countable-unrenormalized}.

\subsection{Organization}
\Cref{sec:spaces,sec:models} introduce the spaces, models, and main
results.  \Cref{sec:enumeration} develops the Perron--Wright
enumeration and proves the coefficient asymptotics.
\Cref{sec:profiles} converts the resulting field estimates into
concentration and quadratic control of component type profiles.
\Cref{sec:component-formulae,sec:factorial-process} derive the exact
component count formulas and establish the local process limits, while
\Cref{sec:countable} constructs the countable type realizations.  The
remaining analytic, projective, and point process tools are collected
in the appendices.

\section{Spaces, operators, and local topologies}\label{sec:spaces}

This section fixes the type space notation and the modes of convergence
used in all three models.  The ground space of the unmarked process is
locally compact, so its point measures carry the usual vague topology.
For countably many types, however, the marked state space
\(E\times\cP(\N)\) is not locally compact.  The standard vague topology
on Radon point measures is therefore unavailable on that product
space, and we use a localized marked topology instead.

\subsection{Type spaces and operators}

This subsection defines the type spaces, profile metrics, kernel
operator, and Perron quantities used throughout the paper.

A \emph{type space} is a separable metric space
\((S,d_S)\), equipped with its Borel \(\sigma\)-field
\(\mathcal B(S)\) and a Borel probability measure \(\mu\).
We use only a finite discrete space, the countable discrete space
\(\N=\{1,2,\ldots\}\), or a compact metric space.  The set of Borel
probability measures on \(S\) is denoted by \(\cP(S)\).

For a positive integer \(m\), write \([m]=\{1,\ldots,m\}\), and put
\(\Z_{\ge0}=\{0,1,2,\ldots\}\).
For integers \(m,j\ge0\), define the falling factorial by
\begin{equation}\label{eq:fallingfactorialdefinition}
 (m)_j=m(m-1)\cdots(m-j+1),
 \qquad (m)_0=1.
\end{equation}
For a bounded function \(f:S\to\mathbb R\), write
\(\|f\|_\infty=\sup_{x\in S}|f(x)|\).
We write \(\1_A\) for the indicator of an event or set \(A\),
\(\lfloor x\rfloor\) for the integer part of \(x\), and use
\(O(\cdot)\), \(o(\cdot)\), and \(\sim\) in their standard asymptotic
sense; in particular, \(a_n\sim b_n\) means \(a_n/b_n\to1\).

If \(S\) is compact, \(\cP(S)\) has the topology of weak convergence.
For \(\nu,\eta\in\cP(S)\), it is metrized by
\[
 d_{\BL}(\nu,\eta)=
 \sup\left\{\left|\int_S f\,\dd(\nu-\eta)\right|:
 \|f\|_\infty\le1,\ \operatorname{Lip}(f)\le1\right\}.
\]
Here \(\operatorname{Lip}(f)\) is the Lipschitz constant of \(f\).
This metric will be used for weak convergence of compact type
profiles.
If \(S\) is finite or countable, then, for
\(\nu,\eta\in\cP(S)\), we use total variation in the equivalent
\(\ell^1\) normalization
\[
 d_{\TV}(\nu,\eta)=\sum_{x\in S}|\nu(\{x\})-\eta(\{x\})|.
\]
This metric will be used for finite and countable type profiles.

A \emph{kernel} is a symmetric Borel function
\(\kappa:S^2\to[0,\infty)\).  When
\(\kappa\in L^2(\mu\otimes\mu)\), its integral operator
\(K:L^2(\mu)\to L^2(\mu)\) is
\begin{equation}\label{eq:integral-operator}
 (Kf)(x)=\int_S\kappa(x,y)f(y)\,\mu(\dd y).
\end{equation}
The symbol \(\spr(A)\) denotes the spectral radius of a bounded operator
\(A\).  Whenever \(1\) is a simple isolated eigenvalue of \(K\) with a
strictly positive eigenfunction \(u\), we normalize that eigenfunction
by
\begin{equation}\label{eq:Perron-eigenfunction}
 Ku=u,\qquad \int_Su\,\dd\mu=1.
\end{equation}
The model assumptions in \Cref{sec:models} guarantee this Perron
structure: it follows from Perron--Frobenius in finite type, from
Perron--Frobenius--Jentzsch in compact type, and from the compact
positive operator argument in \Cref{lem:countable-truncation} for the
countable kernel; see \cite[Theorem~V.6.6]{Schaefer1974} for the positive operator
theory used here.

The probability measure obtained by tilting with the Perron
eigenfunction is called the \emph{Perron law}.  This law, the two scale
constants, and the degree function are, respectively,
\begin{equation}\label{eq:Perron-quantities}
 p(\dd x)=u(x)\mu(\dd x),\qquad
 \alpha=\int_Su^2\,\dd\mu,\qquad
 \beta=\int_Su^3\,\dd\mu,\qquad
 d(x)=\int_S\kappa(x,y)\,\mu(\dd y).
\end{equation}
We use the notation \(p,\alpha,\beta,d\) in
\eqref{eq:Perron-quantities} throughout.
We reserve \emph{type profile} for an empirical measure of the types
inside a component or a collection of components; \(p\) is always the
Perron law, while \(\mu\) (or its finite type analogue) is the ambient
type law.

In a finite type space \(S=\{1,\ldots,d_0\}\), where \(d_0\in\N\), write
\(\mu_a=\mu(\{a\})\), \(\kappa_{ab}=\kappa(a,b)\), and \(u_a=u(a)\).
Then \(K\) is represented by \(\kappa\diag(\boldsymbol\mu)\), where
\(\boldsymbol\mu=(\mu_a)_{a\in S}\), and
\(\diag(\boldsymbol\mu)\) denotes the diagonal matrix with diagonal
entries \((\mu_a)_{a\in S}\).  Moreover,
\[
 p_a=\mu_au_a,\qquad
 \alpha=\sum_a\mu_au_a^2,\qquad
 \beta=\sum_a\mu_au_a^3.
\]

\subsection{Point measures and local topologies}\label{subsec:topology}

This subsection defines the unmarked and marked point measure spaces,
together with the local topologies used in
\Cref{thm:local-main,thm:countable-realization,thm:countable-unrenormalized}.

Put
\[
 E=(0,\infty)\times\Z_{\ge0},
\]
with the product of the usual topology and the discrete topology.
This is a locally compact Polish space.  A \emph{Radon point measure} on
\(E\) is an integer valued Borel measure that is finite on compact sets.
Let \(\Mp(E)\) be the space of all such measures, equipped with the vague
topology:
\(\xi_n\to\xi\) vaguely precisely when
\(\int f\,\dd\xi_n\to\int f\,\dd\xi\) for every
\(f\in C_c(E)\), the continuous compactly supported functions on \(E\).
For a Borel set \(B\), write \(\xi|_B\) for the restriction of \(\xi\)
to \(B\).
Throughout, \(\Longrightarrow\) denotes convergence in distribution in
the topology specified for the displayed random elements.
Equivalently, it is enough to work on the increasing local windows
\[
 B_{a,b,R}=[a,b]\times\{0,\ldots,R\},
 \qquad 0<a<b<\infty,\quad R\in\Z_{\ge0}.
\]

For \(S=\N\), the Polish space \((\cP(\N),d_{\TV})\) is not locally
compact, and neither is \(E\times\cP(\N)\).  Let
\(\cM_{\rm lm}(E\times\cP(\N))\) be the integer valued Borel measures
\(\xi\) such that
\[
 \xi(B_{a,b,R}\times\cP(\N))<\infty
 \quad\hbox{for every \(0<a<b<\infty\) and \(R<\infty\)}.
\]
The \emph{localized marked topology} is generated by the maps
\[
 \xi\longmapsto\int f\,\dd\xi,
 \qquad
 f\in C_b(E\times\cP(\N)),
\]
where the projection of \(\supp(f)\) onto \(E\) is compact; here
\(\supp(f)\) is the closure of \(\{z:f(z)\ne0\}\).  The notation
\(C_b\) denotes the bounded continuous real valued functions.
To make the last statement explicit, choose functions
\(\chi_m\in C_c(E)\), \(0\le\chi_m\le1\), such that every compact
subset of \(E\) is contained, for all sufficiently large \(m\), in a
region on which \(\chi_m=1\).  The measure
\(\chi_m\xi\) is finite on the Polish product
\(E\times\cP(\N)\).  Equip finite measures on that product with any
metric \(d_m\) for weak convergence.  Then
\[
 d_{\rm lm}(\xi,\eta)
 :=\sum_{m\ge1}2^{-m}\{1\wedge d_m(\chi_m\xi,\chi_m\eta)\}
\]
generates the localized marked topology.  Indeed, weak convergence of
all \(\chi_m\xi_n\) tests every function in the displayed defining
class, and the converse follows by testing \(\chi_m f\) for bounded
continuous \(f\).  In particular, this topology is Hausdorff and
metrizable and has a countable convergence determining family.

Equivalently, convergence may be checked by weak convergence of the
restricted finite measures on local windows whose boundaries have zero
mass under the limiting measure.  Unlike test functions with compact
support in the full product, the present test class sees every mark
attached to a point in a fixed local window of \(E\), even when the marks
leave compact subsets of \(\cP(\N)\).  This is the topology used for the
countable type marked limit.  See
\cite{DaleyVereJones2003,Kallenberg2017} for background on weak and
vague convergence of point measures.

For an integer \(q\ge1\) and a point process \(\xi\) on a measurable
space, its \(q\)-th
\emph{factorial moment measure} \(\mathfrak a_\xi^{(q)}\) is defined by
\begin{equation}\label{eq:factorial-measure}
 \int f\,\dd\mathfrak a_\xi^{(q)}
 =\E\sum_{z_1,\ldots,z_q\in\xi}^{\ne}
 f(z_1,\ldots,z_q)
\end{equation}
for nonnegative measurable \(f\).  More precisely, if
\(\xi=\sum_{j\in J}\delta_{Z_j}\), with repeated locations retained as
separate atom occurrences, then the superscript \({}^{\ne}\) means that
the sum is over ordered tuples of pairwise distinct indices in \(J\).
Thus different atom occurrences at the same location are counted, as
they must be when two graph components have the same size and surplus.
The measure \(\mathfrak a_\xi^{(q)}\) therefore gives the expected count
of ordered tuples of length \(q\) of distinct atom occurrences.  In
\Cref{sec:factorial-process,sec:countable}, convergence of these
measures on local windows is combined with
\Cref{lem:factorial-criterion} to identify the component point process
limits.

\section{Models and main results}\label{sec:models}

This section defines the three classes of inhomogeneous random graphs
considered in the paper and states the main results.  The common graph
construction and component processes are given in
\Cref{subsec:common-graph};
\Cref{def:finite-model,def:BJR-model,def:countable-kernel} specify the
deterministic finite type model, the
annealed compact type BJR model, and the regular countable kernel.  The
compact process theorem additionally assumes \QP.  \Cref{subsec:base-enumeration}
introduces the fixed size,
fixed surplus enumerator, whose Perron--Wright asymptotic is the first
main result, \Cref{thm:Perron-Wright}.  After the classical process of
Aldous~\cite{Aldous1997} is recalled in \Cref{subsec:limiting-process},
\Cref{thm:local-main} gives the fixed size factorial asymptotics and the
local size and surplus limits, together with marks equal to the Perron
law, for the finite and compact
models.  \Cref{thm:countable-realization,thm:countable-unrenormalized}
give the corresponding countable type constructions with and without
critical normalization.

\subsection{Graph model and component processes}\label{subsec:common-graph}

This subsection defines the edge probability rule shared by all three
model classes and fixes the component point processes used in the main
limit theorems.

Fix \(\lambda\in\R\).  The finite and compact models below specify a
type space \(S\), vertex types \(X_1,\ldots,X_n\in S\), and a symmetric
nonnegative kernel \(\kappa_n:S^2\to[0,\infty)\) directly.  In the
countable setting, the finite supports, deterministic type assignments,
and kernels are constructed from \Cref{def:countable-kernel} in
\Cref{thm:countable-realization,thm:countable-unrenormalized}.  In every
case the vertex set is \(V_n=[n]\).  Conditional on the types, distinct
unordered vertex pairs are joined independently, with probabilities
\begin{equation}\label{eq:edge-probability}
 p_{ij}^{(n)}
 =1-\exp\left\{-\frac{t_n\kappa_n(X_i,X_j)}{n}\right\},
    \qquad t_n=1+\lambda n^{-1/3}.
\end{equation}
We consider \(n\) sufficiently large that \(t_n>0\), and denote the
resulting labelled simple undirected graph by \(G_n\); in particular,
self loops and multiple edges are absent.  The symbols \(\Prob\) and \(\E\) refer to
its specified law.  The type assignments are deterministic in the
finite and countable constructions.  In the compact type BJR model of
\Cref{def:BJR-model}, the law is annealed and averages over both the
i.i.d.\ types and the conditional edge variables.

Let \(\cC(G_n)\) be the set of connected components of \(G_n\).  For a
finite connected graph \(C\), let \(|C|\) and \(e(C)\) denote its
numbers of vertices and edges and define its \emph{surplus} by
\[
 \sur(C)=e(C)-|C|+1\in\Z_{\ge0}.
\]
The \emph{type profile} of \(C\), namely its empirical type measure, is
\[
 L_C=|C|^{-1}\sum_{i\in C}\delta_{X_i}\in\cP(S).
\]
Here and below, \(\delta_x\) denotes the unit Dirac measure at \(x\).
With the scaling constants in \eqref{eq:Perron-quantities}, define the
unmarked component point process by
\begin{equation}\label{eq:Xi}
 \Xi_n^\star=\sum_{C\in\cC(G_n)}
 \delta_{(\beta^{1/3}n^{-2/3}|C|,\sur(C))}\in\Mp(E).
\end{equation}
For finite or compact \(S\), the space \(\cP(S)\) is compact, so the
marked state space \(E\times\cP(S)\) is locally compact.  In these two
cases define
\begin{equation}\label{eq:Xi-marked}
 \widehat\Xi_n^\star=\sum_{C\in\cC(G_n)}
 \delta_{(\beta^{1/3}n^{-2/3}|C|,\sur(C),L_C)}
 \in\Mp(E\times\cP(S)),
\end{equation}
with the usual vague topology.  For \(S=\N\), the same displayed sum is
regarded instead as an element of
\(\cM_{\rm lm}(E\times\cP(\N))\), the space with localized marked
topology defined in \Cref{subsec:topology}.

\subsection{The three model classes}\label{subsec:model-classes}

We use the common construction in three settings, distinguished by the
type space and the assignment of vertex types.  The finite type model
has a fixed finite type set and prescribed deterministic type counts.

\begin{definition}[Finite type model]\label{def:finite-model}
Fix \(S=\{1,\ldots,d_0\}\), numbers \(\mu_a>0\) with
\(\sum_a\mu_a=1\), and a symmetric nonnegative irreducible matrix
\(\kappa=(\kappa_{ab})_{a,b\in S}\) such that
\begin{equation}\label{eq:finite-criticality}
 \spr\bigl(\kappa\diag(\boldsymbol\mu)\bigr)=1.
\end{equation}
For every \(n\), choose integers \(n_a\ge0\), summing to \(n\), such that
\(n_a=n\mu_a+O(1)\).  Assign exactly \(n_a\) vertices type \(a\), and use
\(\kappa_n=\kappa\) in \eqref{eq:edge-probability}.
\end{definition}

The compact type model replaces the finite type set by a compact metric
space.  Its vertex types are i.i.d.\ with law \(\mu\), so the random graph
law is annealed over both the types and the edges.

\begin{definition}[Regular compact type BJR model]\label{def:BJR-model}
Let \(S\) be a compact metric space, let \(\mu\) be a full support Borel
probability measure (every nonempty open subset of \(S\) has positive
\(\mu\)-measure), and let
\(\kappa:S^2\to[0,\infty)\) be continuous and symmetric.  For an
operator \(A\) on \(L^2(\mu)\), write
\(\|A\|_{2\to2}=\sup_{\|f\|_2\le1}\|Af\|_2\), where
\(\|f\|_2=(\int_S|f|^2\,\dd\mu)^{1/2}\).  Let \(K\) be the integral
operator associated with \(\kappa\) as in \eqref{eq:integral-operator}.
Assume:
\begin{enumerate}[label=\textup{(B\arabic*)}]
\item The compact self-adjoint operator \(K\) has norm
      \(\|K\|_{2\to2}=1\).
\item There is an integer \(m\ge1\) such that the continuous integral
      kernel of \(K^m\) is strictly positive on \(S^2\).
\end{enumerate}
For each \(n\), sample \(X_1,\ldots,X_n\) independently with law \(\mu\),
and use \(\kappa_n=\kappa\) in \eqref{eq:edge-probability}.  Every
probability and expectation in this model averages over both the types
and the edges; thus the conclusion is annealed.
\end{definition}

Condition \textup{(B2)} means that a fixed power of \(K\) has a strictly
positive continuous kernel.  Together with \textup{(B1)}, it yields the
simple positive Perron eigenfunction used in
\eqref{eq:Perron-quantities}; this eigenfunction is continuous and
bounded away from zero.  The irreducibility and criticality conditions
in \Cref{def:finite-model} give the analogous finite dimensional
conclusion.  Hence the quantities in \eqref{eq:Perron-quantities} are
finite and \(\beta>0\) in both models.

The compact component process theorem requires one additional
regularity condition on \(\kappa\).  It asks for an absolutely and
uniformly convergent representation by symmetric products of continuous
functions.  We first define the corresponding kernel class.

\begin{definition}[Symmetric projective kernel class]
\label{def:projectiveclass}
For \(f,g\in C(S)\), write
\((f\otimes g)(x,y)=f(x)g(y)\).  A continuous real symmetric kernel
\(h\) belongs to \(\mathscr P_{\mathrm s}(S)\) if it admits a
representation
\begin{equation}\label{eq:projectiveexpansion}
 h=\sum_{j\ge1}a_j f_j\otimes f_j,
 \qquad
 \sum_{j\ge1}|a_j|\|f_j\|_\infty^2<\infty,
\end{equation}
where \(a_j\in\R\) and \(f_j\in C(S)\).  The sequence may be finite.
Set
\begin{equation}\label{eq:projectivenorm}
 \|h\|_{\pi,\mathrm s}
 =\inf\left\{
   \sum_j|a_j|\|f_j\|_\infty^2:
   h=\sum_j a_jf_j\otimes f_j
  \right\},
\end{equation}
with value \(+\infty\) outside \(\mathscr P_{\mathrm s}(S)\).
This is the symmetric kernel subspace obtained from the projective
tensor construction on \(C(S)\); see \cite[Chapter~2]{Ryan2002} for
the projective tensor norm.  The summability in
\eqref{eq:projectiveexpansion} implies absolute and uniform
convergence on \(S^2\).  Hence the right side of
\eqref{eq:projectiveexpansion} defines a continuous kernel.
\end{definition}

\begin{assumption}[\textup{(QP)} Projective quadratic regularity]\label{ass:QP}
In the compact type BJR model, assume
\[
 \kappa\in\mathscr P_{\mathrm s}(S).
\]
\end{assumption}

\Cref{lem:QP-correction-kernel} shows that it is
equivalent to projective regularity of the canonical quadratic
correction kernel appearing in the compact component calculation.  The
resulting representation supplies the exponential estimate used in the
annealed residual.  Thus \QP{} is a sufficient regularity condition for
the present argument rather than an assumption of the enumerative
theorem.  It is basis independent, directly checkable, and includes
finite feature kernels and absolutely summable continuous feature
expansions.  Further criteria and examples are given in
\Cref{sec:QP-examples}.

Unlike the preceding two cases, the countable model begins with a
limiting kernel on infinitely many types rather than a graph sequence
that is already specified for every \(n\).  A finite graph must retain
only finitely many types, replace their limiting masses by integer vertex
counts, and control the critical point of the resulting finite operator.
The definition below therefore specifies the infinite kernel
\((\mu,\kappa)\); \Cref{thm:countable-realization} constructs the finite
supports, type counts, and critical normalization.

\begin{definition}[Countable kernel]\label{def:countable-kernel}
Let \(S=\N\), let \(\mu_a>0\) with \(\sum_{a\ge1}\mu_a=1\), and let
\(\kappa=(\kappa_{ab})_{a,b\ge1}\) be symmetric.  We call
\((\mu,\kappa)\) a \emph{regular countable critical kernel} if there are
constants \(\kappa_-,\kappa_+\) such that
\begin{equation}\label{eq:countable-bounds}
 0<\kappa_-\le\kappa_{ab}\le\kappa_+<\infty
 \quad(a,b\ge1),\qquad
 \spr(K)=1,
\end{equation}
where \(Kf(a)=\sum_{b\ge1}\kappa_{ab}f_b\mu_b\) on
\(L^2(\mu)\), with \(f_b=f(b)\).
\end{definition}

\subsection{The Perron--Wright component enumerator}
\label{subsec:base-enumeration}

This subsection defines the fixed size, fixed surplus component
enumerator in \eqref{eq:CW-definition} and states its critical window
asymptotic in \Cref{thm:Perron-Wright}.  It extends
Wright's enumeration of connected labelled graphs with fixed surplus
by summing over vertex types and kernel weights.

Let \(c(k,m)\) denote the number of connected simple graphs with vertex
set \([k]\) and \(m\) edges.  For \(r\in\Z_{\ge0}\), the
\emph{Wright constant} \(w_r>0\) is defined by
\begin{equation}\label{eq:Wright-definition}
 c(k,k+r-1)\sim w_r k^{\,k+3r/2-2}
\end{equation}
as \(k\to\infty\); see \cite{Wright1977,Wright1980}.  The name
Perron--Wright reflects that
\Cref{thm:Perron-Wright} retains \(w_r\)
and multiplies it by \(\beta^{(r-1)/2}\), with \(\beta\) as defined in
\eqref{eq:Perron-quantities}.

Fix \(k\ge1\), and let \(H\) be a connected simple graph on the labelled
vertex set \([k]\).  Write \(E(H)\) for its edge set and
\(e(H)=|E(H)|\) for its number of edges.  The next weight is the
contribution of this particular labelled graph together with a fixed
choice of vertex types.  For \(t>0\) and
\(\boldsymbol x=(x_1,\ldots,x_k)\in S^k\), define
\[
 \omega_t(H,\boldsymbol x)=
 \prod_{v=1}^k e^{-t d(x_v)}
 \prod_{\{v,w\}\in E(H)}t\kappa(x_v,x_w).
\]
The first product is the vertex tilt representing the leading isolation
contribution, while the second product gives the kernel weight of the
edges of \(H\).  Integrating over the type law gives the weighted
contribution of the labelled graph \(H\) after averaging over types:
\begin{equation}\label{eq:tilted-graph-weight}
 \WW_t(H)=\int_{S^k}
 \omega_t(H,\boldsymbol x)\prod_{v=1}^k\mu(\dd x_v).
\end{equation}
The integral is a finite sum in a discrete type space.  For
\(r\in\Z_{\ge0}\), define
\begin{equation}\label{eq:CW-definition}
 \CW_t(k,r)=
 \sum_{\substack{H\ {\rm connected\ on}\ [k]\\
                  e(H)=k+r-1}}\WW_t(H).
\end{equation}
\(\CW_t(k,r)\) is the total weight of all connected labelled
candidate components with \(k\) vertices and surplus
\(e(H)-k+1=r\).  Equivalently, it is the partition function for this
fixed size, fixed surplus component ensemble, with \(t\) acting as an
edge intensity parameter.
The exact finite graph identities in \Cref{prop:component-count-formula}
separate this enumerator from the remaining vertex selection and
isolation factors.
When \(\CW_t(k,r)>0\), define the probability law
\(\cQ_{k,r,t}\) on pairs \((H,\boldsymbol x)\), where
\(H\) is connected and simple on \([k]\),
\(e(H)=k+r-1\), and
\(\boldsymbol x=(x_1,\ldots,x_k)\in S^k\), by requiring that every
nonnegative measurable function \(F\) satisfy
\begin{align}
 &\E_{\cQ_{k,r,t}}F(H,X_1,\ldots,X_k)\notag\\
 &\quad=\frac1{\CW_t(k,r)}
 \sum_{\substack{H\ {\rm connected\ on}\ [k]\\e(H)=k+r-1}}
 \int_{S^k}F(H,x_1,\ldots,x_k)
 \omega_t(H,\boldsymbol x)
 \prod_{v=1}^k\mu(\dd x_v).
 \label{eq:Q-definition}
\end{align}
Thus \((X_1,\ldots,X_k)\) is the type vector under this law.  Its type
profile is
\[
 L_k=k^{-1}\sum_{i=1}^k\delta_{X_i}.
\]
All statements involving \(\cQ_{k,r,t}\) are understood for
\(\CW_t(k,r)>0\).  Under the irreducibility assumption in
\Cref{def:finite-model}, or under the compact primitivity assumption
\textup{(B2)} in \Cref{def:BJR-model}, this positivity condition is
satisfied for all sufficiently large \(k\), for each fixed surplus
\(r\).  To see this directly, choose types \(x,y\) with
\(\kappa(x,y)>0\).  If necessary take neighborhoods of these types on
which the kernel remains positive; full support gives both
neighborhoods positive \(\mu\)-mass.  The complete bipartite graph
\(K_{2,r+1}\), with its two parts assigned the two types or
neighborhoods, is connected and has surplus \(r\).  Adding leaves
along a positive kernel link increases the number of vertices without
changing the surplus.  This gives positive weight for every sufficiently
large \(k\).  The same construction applies when \(x=y\), after choosing
the participating types in a neighborhood on which \(\kappa>0\).

The next result is the paper's primary enumerative theorem.
\begin{theorem}[Perron--Wright formula]
\label{thm:Perron-Wright}
For the type law and kernel in the finite type model of
\Cref{def:finite-model}, or in the regular compact type BJR model
satisfying \textup{(B1)--(B2)} of \Cref{def:BJR-model}, there is a positive function \(t\mapsto\rho(t)\), real analytic near \(1\).  It is the critical radius of
\[
 T=z\exp\{g_t-d+KT\},
 \qquad
 g_t=(\log t)\1-(t-1)d.
\]
It satisfies
\begin{equation}\label{eq:Perron-Wright-radius}
 \rho(1)=1,
 \qquad
 \log\rho(t)
 =\frac{\alpha^2}{2\beta}(t-1)^2+O(|t-1|^3).
\end{equation}

For every fixed \(R\in\mathbb Z_{\ge0}\) and every sequence
\(\eta_k\downarrow0\),
\begin{equation}\label{eq:Perron-Wright-main}
 \CW_t(k,r)
 =
 w_r\beta^{(r-1)/2}e^{-k}\rho(t)^{-k}
 k^{\,k+3r/2-2}
 \left\{
  1+O_R\!\left(\eta_k+k^{-1/2}\right)
 \right\},
\end{equation}
uniformly for
\[
 0\le r\le R,
 \qquad
 |t-1|\le\eta_k,
\]
as \(k\to\infty\).  The implied constant is independent of
\(k\), \(r\), \(t\), and the particular sequence \((\eta_k)\), once
\(k\) is sufficiently large.

In particular, for every fixed \(0\le L<\infty\), uniformly for \(|t-1|\le Lk^{-1/2},\)
\begin{equation}\label{eq:Perron-Wright-gaussian}
 \CW_t(k,r)
 =
 w_r\beta^{(r-1)/2}e^{-k}
 \exp\!\left\{
  -\frac{\alpha^2}{2\beta}k(t-1)^2
 \right\}
 k^{\,k+3r/2-2}
 \left\{
  1+O_{R,L}\!\left(k^{-1/2}\right)
 \right\}.
\end{equation}

\end{theorem}

\subsection{The Aldous size and surplus process}
\label{subsec:limiting-process}

The limit in \Cref{thm:local-main} is the classical size and surplus
process constructed by Aldous~\cite{Aldous1997}.  We recall it here to
fix the normalization of the size coordinate, critical window parameter,
and surplus marks.  Let
\(B=(B(s))_{s\ge0}\) be standard Brownian motion.  For
\(\Lambda\in\R\), define
\[
 W_\Lambda(s)=B(s)+\Lambda s-\frac{s^2}{2},\qquad
 \overline W_\Lambda(s)=W_\Lambda(s)-
 \inf_{0\le v\le s}W_\Lambda(v).
\]
Let \((I_j)_{j\ge1}\) be the excursion intervals of
\(\overline W_\Lambda\) away from zero, in any measurable ordering, and
let \(\gamma_j=|I_j|\).  Conditional on \(\overline W_\Lambda\), place a
unit intensity Poisson point process on the region under its graph,
\[
 \{(s,y):s\ge0,\ 0<y<\overline W_\Lambda(s)\},
\]
and let \(N_j\) be the number of its points whose first coordinate is in
\(I_j\).  The \emph{Aldous size and surplus process}
of~\cite{Aldous1997} is
\begin{equation}\label{eq:Aldous-process}
 \Xi_\Lambda^{\Ald}=\sum_{j\ge1}\delta_{(\gamma_j,N_j)}\in\Mp(E).
\end{equation}
It is independent of the ordering of the excursions and is locally
finite almost surely.

The following identity gives the fixed size factorial densities of the
Aldous process in terms of the same Wright constants.  It is the
normalization used later to identify the limiting point process; the
enumerative connection with Brownian excursion area is discussed in
\cite{Spencer1997,Janson2007}.

\begin{proposition}[Aldous size and surplus factorial densities]
\label{prop:AW-density}
For \(q\ge1\), \(x_1,\ldots,x_q>0\), and
\(r_1,\ldots,r_q\in\Z_{\ge0}\), put
\(S_q=\sum_{i=1}^q x_i\).
The \(q\)-th factorial moment measure of
\(\Xi_\Lambda^{\Ald}\), with counting measure in the surplus coordinates
and Lebesgue measure in the size coordinates, has density
\begin{equation}\label{eq:AW-density}
 \exp\left\{-\frac{S_q^3}{6}+\frac{\Lambda S_q^2}{2}
 -\frac{\Lambda^2S_q}{2}\right\}
 \prod_{i=1}^q\frac{w_{r_i}}{\sqrt{2\pi}}\,
 x_i^{3r_i/2-5/2}.
\end{equation}
On every window \(B_{a,b,R}\), these factorial moment measures determine
the law of the restricted point process.
\end{proposition}

The proof, including the moment determinacy assertion, is given in
\Cref{app:point-process-tools}.

\subsection{Critical component point process limit}
\label{subsec:factorial-main}

This subsection states the paper's main probabilistic theorem,
\Cref{thm:local-main}, for the deterministic finite type and annealed
compact type BJR models.  It gives fixed size factorial asymptotics, the
local size and surplus process limit, and the corresponding marked limit
with the Perron law as deterministic mark.  We first define the limiting
factorial density and its exponential rate function.

For \(q\ge1\), \(x_1,\ldots,x_q>0\), and
\(r_1,\ldots,r_q\in\Z_{\ge0}\), put
\[
 S_q=\sum_{i=1}^q x_i
\]
and define
\begin{align}
 F_{\alpha,\beta}(S_q,\lambda)
 &=\frac{\beta S_q^3}{6}
   -\frac{\alpha\lambda S_q^2}{2}
   +\frac{\alpha^2\lambda^2S_q}{2\beta},
 \label{eq:F-definition}\\
 m_{\alpha,\beta,\lambda}^{(q)}
 ((x_1,r_1),\ldots,(x_q,r_q))
 &=e^{-F_{\alpha,\beta}(S_q,\lambda)}
 \prod_{i=1}^q
 \frac{w_{r_i}\beta^{(r_i-1)/2}}{\sqrt{2\pi}}\,
 x_i^{3r_i/2-5/2}.
 \label{eq:m-density}
\end{align}
Thus \(m_{\alpha,\beta,\lambda}^{(q)}\) is a nonnegative function on
\(((0,\infty)\times\Z_{\ge0})^q\); it will be the limiting factorial
moment density with respect to Lebesgue measure in each size coordinate
and counting measure in each surplus coordinate.

\begin{theorem}[Local factorial asymptotics and component process limit]
\label{thm:local-main}
Consider either the deterministic finite type model of
\Cref{def:finite-model} or the annealed compact type BJR model of
\Cref{def:BJR-model}; in the compact case assume additionally
\Cref{ass:QP}.
\begin{enumerate}[label=\textup{(\roman*)}]
\item \emph{Fixed size factorial asymptotics.}
Fix \(q\ge1\), \(r_1,\ldots,r_q\in\Z_{\ge0}\), and
\(0<a<b<\infty\).  For \(x_i\in[a,b]\), let
\(k_i=\lfloor x_i n^{2/3}\rfloor\).  The superscript \({}^{\ne}\)
below means that the components are ordered and pairwise distinct.  Then
\begin{align}
 &n^{2q/3}\E
 \sum_{C_1,\ldots,C_q\in\cC(G_n)}^{\ne}
 \prod_{i=1}^q
 \1_{\{|C_i|=k_i,\ \sur(C_i)=r_i\}}
 \notag\\
 &\hspace{35mm}\longrightarrow
 m_{\alpha,\beta,\lambda}^{(q)}
 ((x_1,r_1),\ldots,(x_q,r_q))
 \label{eq:factorial-asymptotic}
\end{align}
uniformly on \([a,b]^q\).
\item \emph{Component processes.}
In both models, the unmarked processes satisfy
\begin{equation}\label{eq:process-main}
 \Xi_n^\star\Longrightarrow
 \Xi_{\Lambda}^{\Ald}
 \quad\hbox{in \(\Mp(E)\)},\qquad
 \Lambda=\alpha\beta^{-2/3}\lambda.
\end{equation}
The marked processes satisfy
\begin{equation}\label{eq:marked-process-main}
 \widehat\Xi_n^\star\Longrightarrow
 \sum_{(x,r)\in\Xi_\Lambda^{\Ald}}\delta_{(x,r,p)}.
\end{equation}
Convergence in \eqref{eq:marked-process-main} is vague on
\(E\times\cP(S)\).
\end{enumerate}
\end{theorem}

\subsection{Countable type realizations and local limits}
\label{subsec:countable-main}

The countable type results start from an infinite regular critical
kernel and construct finite graphs with growing type supports and
integer type counts.  Each finite truncation has its own Perron root and
therefore shifts the critical point.  \Cref{thm:countable-realization}
uses exact critical normalization at every stage, while
\Cref{thm:countable-unrenormalized} uses the original kernel scaling when
the remaining spectral drift is controlled on the critical window
scale.

\begin{theorem}[Countable type critical realization]
\label{thm:countable-realization}
\emph{Finite support realization.}
For every regular countable critical kernel
\((\mu,\kappa)\) satisfying the assumptions in
\Cref{def:countable-kernel}, there are increasing finite sets
\(S_n\subset\N\), deterministic positive type counts
\((n_a)_{a\in S_n}\) summing to \(n\), and the empirical probability
measure \(\nu^{(n)}\in\cP(\N)\) defined by
\(\nu_a^{(n)}=\nu^{(n)}(\{a\})=n_a/n\) for \(a\in S_n\) and
\(\nu_a^{(n)}=0\) for \(a\notin S_n\).  These objects have the following
properties:
\begin{equation}\label{eq:countable-realization-convergence}
 \bigcup_nS_n=\N,\qquad
 d_{\TV}(\nu^{(n)},\mu)\to0,\qquad
 n^{1/3}\min_{a\in S_n}\nu_a^{(n)}\to\infty.
\end{equation}
Write \(\boldsymbol\nu^{(n)}=(\nu_a^{(n)})_{a\in S_n}\), and let
\(\kappa_{S_n}=(\kappa_{ab})_{a,b\in S_n}\) denote the restriction of
\(\kappa\) to \(S_n^2\).
Put
\begin{equation}\label{eq:countable-rho-n}
 \sigma_n=\spr\{\kappa_{S_n}\diag(\boldsymbol\nu^{(n)})\}.
\end{equation}
Assign exactly \(n_a\) vertices type \(a\), and join types \(a,b\) with
probability
\begin{equation}\label{eq:countable-edge}
 1-\exp\left\{
 -\frac{(1+\lambda n^{-1/3})\kappa_{ab}}{n\sigma_n}
 \right\}.
\end{equation}
Then the following conclusions hold with the Perron quantities
\((p,\alpha,\beta)\) of the original countable kernel.
\begin{enumerate}[label=\textup{(\roman*)}]
\item \emph{Fixed size factorial asymptotics.}
Fix \(q\ge1\), \(r_1,\ldots,r_q\in\Z_{\ge0}\), and
      \(0<a<b<\infty\).  For \(x_i\in[a,b]\), put
      \(k_i=\lfloor x_i n^{2/3}\rfloor\).  Then
\begin{align}
 &n^{2q/3}\E
 \sum_{C_1,\ldots,C_q\in\cC(G_n)}^{\ne}
 \prod_{i=1}^q
 \1_{\{|C_i|=k_i,\ \sur(C_i)=r_i\}}
 \notag\\
 &\hspace{35mm}\longrightarrow
 m_{\alpha,\beta,\lambda}^{(q)}
 ((x_1,r_1),\ldots,(x_q,r_q))
 \label{eq:countable-factorial-asymptotic}
\end{align}
uniformly on \([a,b]^q\).
\item \emph{Component processes.}
The unmarked and marked component processes satisfy
\begin{align}
 \Xi_n^\star&\Longrightarrow
 \Xi_{\Lambda}^{\Ald}
 \quad\hbox{in \(\Mp(E)\)},
 &\Lambda&=\alpha\beta^{-2/3}\lambda,
 \label{eq:countable-process-main}\\
 \widehat\Xi_n^\star&\Longrightarrow
 \sum_{(x,r)\in\Xi_\Lambda^{\Ald}}\delta_{(x,r,p)}.
 \label{eq:countable-marked-process-main}
\end{align}
The marked convergence is in the localized marked topology of
\Cref{subsec:topology}, with total variation as the metric on
\(\cP(\N)\).
\end{enumerate}
\end{theorem}

Thus \Cref{thm:countable-realization} constructs, from every kernel in
\Cref{def:countable-kernel}, one finite support graph sequence on which
all the stated local limits hold simultaneously.

\begin{theorem}[Countable type limits without spectral renormalization]
\label{thm:countable-unrenormalized}
Use the finite supports \(S_n\), deterministic type counts
\((n_a)_{a\in S_n}\), empirical laws \(\nu^{(n)}\), and spectral radii
\(\sigma_n\) constructed in \Cref{thm:countable-realization}.  Let
\(\widetilde G_n\) be the graph on the same typed vertex sets in which
types \(a,b\) are joined with probability
\begin{equation}\label{eq:countable-edge-unrenormalized}
 1-\exp\left\{
 -\frac{(1+\lambda n^{-1/3})\kappa_{ab}}{n}
 \right\}.
\end{equation}
If, for some \(\vartheta\in\R\),
\begin{equation}\label{eq:countablespectraldrift}
 n^{1/3}(\sigma_n-1)\longrightarrow\vartheta,
\end{equation}
then the fixed size factorial asymptotics in
\eqref{eq:countable-factorial-asymptotic} and the process convergences
\eqref{eq:countable-process-main}--\eqref{eq:countable-marked-process-main}
hold for \(\widetilde G_n\) with \(\lambda\) replaced everywhere by
\(\lambda+\vartheta\); in the cited equations, the point processes are
formed from \(\widetilde G_n\).  In particular, the Aldous parameter
becomes \(\alpha\beta^{-2/3}(\lambda+\vartheta)\).  Under the stronger
condition \(\sigma_n=1+o(n^{-1/3})\), these conclusions hold for
\(\widetilde G_n\) with the original parameter \(\lambda\).
\end{theorem}

\section{Perron--Wright enumeration and uniform perturbations}
\label{sec:enumeration}

In this section, we establish the fixed surplus Perron--Wright
coefficient estimate and its uniform perturbative extensions.  Specifically, we prove \Cref{thm:Perron-Wright}, give a uniform estimate
for the coefficients with an external field in
\Cref{prop:shrinking-field-variation}, and establish stability along
convergent sequences of finite type critical models in
\Cref{lem:finite-parameter-stability}.

Here we work with either the finite type pair in \Cref{def:finite-model}
or a compact type pair satisfying \textup{(B1)--(B2)} of \Cref{def:BJR-model}.
Denote \(\mathscr X=\mathbb R^{d_0}\) in the finite type case and
\(\mathscr X=C(S)\), with the uniform norm, in the compact type case.
Products, exponentials, and order relations in \(\mathscr X\) are
understood pointwise.  Let \(K\) denote the complex linear
integral operator
\[
 (Kf)(x)=\int_S\kappa(x,y)f(y)\,\dd\mu(y),
 \qquad
 d=K\1,
\]
both on \(\mathscr X_{\mathbb C}\) and on
\(L^2(\mu;\mathbb C)\).  For \(v\in\mathscr X_{\mathbb C}\), let
\[
 (M_vf)(x):=v(x)f(x)
\]
denote the multiplication operator on either space.

We use \(\mathscr X_{\mathbb C}\) for the tree equation and the
holomorphic implicit function arguments, and \(L^2(\mu;\mathbb C)\)
for spectral projections, traces, Hilbert--Schmidt norms, and
regularized determinants.  Thus inverses
arising from the tree linearization are taken in
\(\mathcal L(\mathscr X_{\mathbb C})\), whereas all spectral and
determinantal notation refers to \(L^2(\mu;\mathbb C)\).
The compatibility of the two realizations is supplied by
\Cref{app:continuous-kernel-HS}; in finite type the distinction
disappears.

Under the standing criticality and irreducibility assumptions, the
Perron eigenvalue of \(K\) is one and is algebraically simple.
Normalize its strictly positive eigenfunction by
\[
 Ku=u,
 \qquad
 \int_Su\,\dd\mu=1,
\]
and put
\[
 \alpha=\int_Su^2\,\dd\mu,
 \qquad
 \beta=\int_Su^3\,\dd\mu,
 \qquad
 \ell(v)=\frac1\alpha\int_Suv\,\dd\mu.
\]

For \(g\in\mathscr X_{\mathbb C}\), define the connected weight with
field \(g\) by
\begin{equation}\label{eq:fielded-connected-weight}
 \CW_g(k,r)
 :=
 \sum_{\substack{H\ {\rm connected\ on}\ [k]\\
                  e(H)=k+r-1}}
 \int_{S^k}
 \prod_{v=1}^k e^{g(x_v)-d(x_v)}
 \prod_{\{v,w\}\in E(H)}\kappa(x_v,x_w)
 \prod_{v=1}^k\mu(\dd x_v).
\end{equation}
For every fixed \(k,r\), the map
\(g\mapsto\CW_g(k,r)\) is entire on
\(\mathscr X_{\mathbb C}\).  We write
\[
 \CW(k,r):=\CW_0(k,r),
\]
which is the connected weight \(\CW_t(k,r)\) from
\eqref{eq:CW-definition} at \(t=1\).

Let \(T_g(z;x)\) denote the exponential generating function of rooted
weighted trees whose root has type \(x\).  Decomposing a tree at its
root, the root contributes
\(z e^{g(x)-d(x)}\), while its incident rooted subtrees form a labelled
set with total activity
\[
 \int_S\kappa(x,y)T_g(z;y)\,\dd\mu(y)=(KT_g)(x).
\]
The labelled exponential formula therefore gives, coefficientwise,
\begin{equation}\label{eq:field-tree-equation}
 T_g=z\exp\{g-d+KT_g\}.
\end{equation}

\subsection{Critical normalization and uniform tree continuation}
\label{subsec:critical-field-transport}
This subsection constructs the critically normalized rooted tree
generating function.  It then gives a jointly holomorphic
continuation to one common dented domain.

For a field \(g\) near zero, consider the complexification of
\eqref{eq:field-tree-equation}.
Since \(1\) is an isolated algebraically simple eigenvalue of \(K\),
fix a positively oriented contour \(\Gamma\) enclosing \(1\) and no
other point of \(\operatorname{spec}(K)\).  The map
\[
 \tau\longmapsto KM_\tau
\]
is holomorphic.  Hence, on some complex neighborhood
\(\mathcal V\subset\mathscr X_{\mathbb C}\) of \(\mathbf 1\), the contour
\(\Gamma\) remains in the resolvent set of \(KM_\tau\), and the Riesz
projection
\[
 \mathsf Q(\tau):=
 \frac{1}{2\pi i}\int_\Gamma(\zeta I-KM_\tau)^{-1}\,\dd\zeta
\]
has rank one.  Denote by \(\chi(\tau)\) the unique eigenvalue of
\(KM_\tau\) enclosed by \(\Gamma\).  Analytic Riesz perturbation makes
\(\chi:\mathcal V\to\mathbb C\) holomorphic, with
\(\chi(\mathbf 1)=1\).

Define the critical pair by
\begin{equation}
\label{eq:critical-tree-field} 
 \tau_g=\rho_g\exp\{g-d+K\tau_g\}, \quad
 \chi(\tau_g)=1.
\end{equation}

\begin{lemma}[Analytic critical pair]\label{lem:critical-pair}
There are \(\varepsilon_{\rm an}>0, C>0\) and a locally unique pair of
holomorphic maps
\[
 g\longmapsto\rho_g\in\mathbb C,
 \qquad g\longmapsto\tau_g\in\mathscr X_{\mathbb C},
 \qquad \|g\|_\infty<\varepsilon_{\rm an},
\]
satisfying 
\eqref{eq:critical-tree-field}, with
\(\rho_0=1\) and \(\tau_0=\1\). When
\(g\) is real,
\(\rho_g\) and \(\tau_g\) are strictly positive.

Moreover, write
\[
 a(g):=\log\rho_g,
 \qquad
 p(f):=\int_Su f\,\dd\mu.
\]
Then
\begin{equation}\label{eq:source-radius-derivative}
 D_ga(0)[f]=-p(f).
\end{equation}
\begin{equation}\label{eq:source-radius-Taylor}
 |a(g+q)-a(g)+p(q)|
 \le C\bigl(\|q\|_\infty^2
             +\|g\|_\infty\|q\|_\infty\bigr)
\end{equation}
whenever \(g,q,g+q\in
B_{\mathscr X_{\mathbb C}}(0,\varepsilon_{\rm an})\).
\end{lemma}

\begin{proof}
Define
\[
 \mathcal F(a,\tau,g)
 :=\bigl(\log\tau-a\1-g+d-K\tau,\,\chi(\tau)-1\bigr).
\]
The normalized left Perron functional satisfies
\(\ell(u)=1\) and \(\ell K=\ell\).  Since
\[
 D_\tau(KM_\tau)\big|_{\tau=\1}[v]=KM_v,
\]
the perturbation formula for a simple eigenvalue gives
\[
 D_\tau \chi(\1)[v]
 =\ell(KM_vu)
 =\ell(M_vu)
 =\frac1\alpha\int_Su^2v\,\dd\mu.
\]
Consequently, at \((a,\tau,g)=(0,\1,0)\), the derivative in
\((a,\tau)\), applied to
\((b,v)\in\mathbb C\times\mathscr X_{\mathbb C}\), is
\begin{equation}\label{eq:critical-derivative}
 (b,v)\longmapsto
 \left((I-K)v-b\1,
       \alpha^{-1}\int_Su^2v\,\dd\mu\right).
\end{equation}
Pairing the first coordinate with \(u\) determines \(b\).  On the
complement of codimension one to the Perron direction, \(I-K\) is a
bounded isomorphism; the second coordinate then determines the
remaining multiple of \(u\), because
\[
 \alpha^{-1}\int_Su^2(cu)\,\dd\mu=c\beta/\alpha.
\]
Thus \eqref{eq:critical-derivative} is invertible, and the analytic
implicit function theorem gives a locally unique holomorphic pair
\((\rho_g,\tau_g)\) on a ball
\(B_{\mathscr X_{\mathbb C}}(0,\varepsilon_{\rm an})\) for some
\(\varepsilon_{\rm an}>0\).  After decreasing
\(\varepsilon_{\rm an}\) if necessary, conjugation and local uniqueness
give reality on the real slice, and \(\rho_g\) and \(\tau_g\) are
strictly positive there.

Let \(v_f=D_g\tau_0[f]\).  Differentiating the logarithmic form of
\eqref{eq:critical-tree-field} at \(g=0\) gives
\[
 (I-K)v_f=D_ga(0)[f]\1+f.
\]
Pairing with \(u\) yields
\[
 D_ga(0)[f]=-\int_Suf\,\dd\mu=-p(f).
\]
After shrinking \(\varepsilon_{\rm an}\) once more,
\Cref{lem:nested-product-Cauchy} applied to the holomorphic map \(a\) gives
\[
 \sup_{\|g\|_\infty<\varepsilon_{\rm an}}
 \|D_g^2a(g)\|_{\mathcal L^2(\mathscr X_{\mathbb C};\mathbb C)}
 <\infty .
\]
 The identity 
\[
 \begin{aligned}
 a(g+q)-a(g)+p(q)
  =(D_ga(g)-D_ga(0))[q] 
  +\int_0^1(1-s)
      D_g^2a(g+sq)[q,q]\,\dd s
 \end{aligned}
\]
 then proves \eqref{eq:source-radius-Taylor}.
\end{proof}

\begin{lemma}[The rooted tree branch as the least solution]
\label{lem:critical-pair-least-branch}
Fix
\(g\in B_{\mathscr X}(0,\varepsilon_{\rm an})\), and let
\((\rho,\tau)=(\rho_g,\tau_g)\) be the critical pair supplied by
\eqref{eq:critical-tree-field}.  Then, for every
\(0\le z\le\rho\), \eqref{eq:field-tree-equation} has a least
nonnegative solution \(T_{\min}(\cdot,z)\).  Put
\[
 \widehat T_{\min}(w):=T_{\min}(\cdot,\rho w),
 \qquad
 V_{\min}(w):=\tau^{-1}\widehat T_{\min}(w).
\]
Then, for every \(0<w\le1\),
\begin{equation}
\label{eq:critical-least-branch-bound}
 0<V_{\min}(w)\le w\1\le\1.
\end{equation}
The normalized rooted tree series is analytic for \(|w|<1\) and
extends continuously to the closed real interval \([0,1]\), where it
agrees with \(V_{\min}\).  In particular,
\begin{equation}
 V_{\min}(1)=\1.
\end{equation}
Its radius of convergence is exactly one.

Equivalently, the rooted tree series in the original variable \(z\)
has radius exactly \(\rho\), agrees with \(T_{\min}(\cdot,z)\) for
\(0\le z\le\rho\), and the map
\[
 z\longmapsto T_{\min}(\cdot,z)
\]
is continuous from \([0,\rho]\) into \(\mathscr X\), with
\[
 T_{\min}(\cdot,\rho)=\tau.
\]
\end{lemma}

\begin{proof}
Dividing \eqref{eq:field-tree-equation}, at \(z=\rho w\), by the
critical equation in \eqref{eq:critical-tree-field} shows that
\(V=\tau^{-1}T(\cdot,\rho w)\) solves
\begin{equation}
\label{eq:normalized-critical-tree-equation}
 V=w\exp\{-A\1+AV\}
   =w\exp\{A(V-\1)\},
 \qquad A=KM_\tau.
\end{equation}

For fixed \(0\le w\le1\), set
\[
 \Phi_w(V):=w\exp\{A(V-\1)\},
 \qquad
 V^{[0]}=0,
 \qquad
 V^{[m+1]}=\Phi_w(V^{[m]}).
\]

The map \(\Phi_w\) is order preserving.  Moreover, whenever
\(0\le V\le\1\), positivity of \(A\) gives
\[
 0<\Phi_w(V)\le w\1\le\1
 \qquad(w>0).
\]
Thus \(V^{[m]}\) increases pointwise and is bounded by \(w\mathbf 1\). Let \(V_\infty\) denote its pointwise limit. The bounded-kernel estimate and dominated convergence give \(AV^{[m]}\to AV_\infty\) uniformly. Passing to the limit in the iteration equation, uniformly on \(S\), shows that \(V_\infty\) solves \eqref{eq:normalized-critical-tree-equation} and that \(V^{[m]}\to V_\infty\) in \(\mathscr X\). Every other nonnegative solution dominates every iterate, so \(V_\infty=V_{\min}(w)\), and the preceding bound proves \eqref{eq:critical-least-branch-bound}.

Combinatorially, \(T^{[m]}:=\tau V^{[m]}\) is exactly the
rooted tree generating function resolved by type at activity
\(z=\rho w\), restricted to trees of height \(<m\).
Consequently, as \(m\to\infty\), the iterates converge
coefficientwise to the full rooted tree series
\[
 V_{\min}(w)=\sum_{n\ge1}V_nw^n.
\]
Here all coefficients are nonnegative, and each coefficient stabilizes
after finitely many iterations. 

At \(w=1\), put
\(q:=\1-V_{\min}(1)\).  The first iterate is strictly positive, so
\(0\le q<\1\). Let \(\ell_A\) be a strictly positive left Perron functional of \(A\),
normalized by \(\ell_A(\1)=1\);  \eqref{eq:normalized-critical-tree-equation} then
gives
\[
 \ell_A\!\left(-\log(\1-q)-q\right)=0.
\]
Hence \(q=0\), since the function inside the functional is
nonnegative and vanishes only at \(q=0\).  Thus
\(V_{\min}(1)=\1\). 

Consequently, the coefficient partial sums increase pointwise to
\(\1\). Dini's theorem and positivity then
give
\[
 \sum_{n\ge1}V_n=\1
 \qquad\text{in }\mathscr X,
\]
and \(w\mapsto V_{\min}(w)\) is continuous on \([0,1]\), with radius of
convergence at least one.

Finally, the point \(w=1\) cannot lie inside the disk of convergence.
Otherwise differentiation of
\eqref{eq:normalized-critical-tree-equation} at
\((w,V)=(1,\1)\) would give
\[
 (I-A)V'(1)=\1.
\]
Applying \(\ell_A\) makes the left side zero and the right hand
side \(\ell_A(\1)>0\), a contradiction.  The normalized radius is
therefore exactly one.  Since
\(T(z)=\tau V(z/\rho)\), the radius in \(z\) is exactly \(\rho\), and
the final endpoint assertion follows.
\end{proof}


We now use this identification of the least branch to normalize the rooted tree
germ and continue it uniformly in the field parameter.

For complex \(g\) sufficiently close to zero, let \(T_g(z)\) denote
the unique \(\mathscr X_{\mathbb C}\)-valued power-series germ at
\(z=0\) satisfying \eqref{eq:field-tree-equation} and \(T_g(0)=0\),
and define
\[
 \widehat T_g(w):=T_g(\rho_gw),
 \qquad
 T(w):=\widehat T_0(w).
\]
By \Cref{lem:critical-pair-least-branch}, on the real slice this germ
is the normalized least branch; the next lemma constructs its joint
\(\Delta\)-continuation and local fold chart.

\begin{lemma}[Uniform continuation of the normalized rooted tree generating function]
\label{lem:normalized-field-tree}
There exist constants
\[
 \eta>0,\quad
 0<\varphi<\frac{\pi}{2},\quad
 0<r_T<1,\quad
  0<\delta<
\frac{r_T^2}{2},\quad
 0<\varepsilon_T\le\varepsilon_{\rm an}.
\]
Set
\begin{equation}\label{eq:base-dented-domain}
 \Delta
 :=
 \left\{
  w:\ |w|<1+\eta,\quad w\ne1,\quad
  |\arg(1-w)|<\pi-\varphi
 \right\},
\end{equation}
and let \(X(w)=\sqrt{1-w}\) be the branch positive on \(0<w<1\).
Then the following assertions hold.

\smallskip
\noindent\textup{\rm(A) The global normalized branch.}
The germ \((w,g)\mapsto\widehat T_g(w)\) admits a jointly
holomorphic continuation to
\[
 \Delta\times
 B_{\mathscr X_{\mathbb C}}(0,\varepsilon_T),
\]
still denoted by \(\widehat T_g(w)\).  For real \(g\), its restriction
to \(0<w<1\) is the normalized least branch selected in
\Cref{lem:critical-pair-least-branch}.

Moreover, the maps
\[
 (w,g)\longmapsto\widehat T_g(w),
 \qquad
 (w,g)\longmapsto
 (I-M_{\widehat T_g(w)}K)^{-1}
\]
are jointly holomorphic and uniformly bounded
on
\[
 \bigl(\Delta\cap\{|1-w|>\delta\}\bigr)
 \times B_{\mathscr X_{\mathbb C}}(0,\varepsilon_T).
\]

\smallskip
\noindent\textup{\rm(B) The uniform fold chart.}
For real \(g\), define initially
\begin{equation}\label{eq:real-X-tree-branch}
 \widehat T(0,g):=\tau_g,
 \qquad
 \widehat T(X,g):=\widehat T_g(1-X^2),
 \quad 0<X<r_T.
\end{equation}
These values extend to a jointly holomorphic map
\[
 \widehat T:
 \{|X|<r_T\}\times
 B_{\mathscr X_{\mathbb C}}(0,\varepsilon_T)
 \longrightarrow\mathscr X_{\mathbb C},
\]
and
\begin{equation}\label{eq:tree-chart-identification}
 \widehat T_g(w)=\widehat T(X(w),g),
 \qquad
 w\in\Delta,\qquad
 0<|1-w|<r_T^2.
\end{equation}
Moreover, there is a jointly holomorphic map
\[
 R_T:
 \{|X|<r_T\}\times
 B_{\mathscr X_{\mathbb C}}(0,\varepsilon_T)
 \longrightarrow\mathscr X_{\mathbb C}
\]
such that
\begin{equation}\label{eq:field-tree-fold-expansion}
 \widehat T(X,g)
 =\tau_g
  +X\,\partial_X\widehat T(0,g)
  +X^2R_T(X,g).
\end{equation}
In particular,
\begin{equation}\label{eq:base-square-root}
 \widehat T(X,0)
 =\1-\sqrt{\frac2\beta}\,uX+X^2R_T(X,0).
\end{equation}
\end{lemma}

\begin{proof}
We first construct the local fold chart near \(w=1\) and the
base continuation away from \(1\), and then decompose the resulting
dented domain into the regions shown in
\Cref{fig:tree-continuation-gluing} and glue the two constructions on
their overlap.

\smallskip
\noindent\emph{The local fold chart.}
Put
\[
 \mathcal F(w,U,g)
 :=
 U-w\rho_g\exp\{g-d+KU\}.
\]
In the local coordinates \(w=1-s\) and \(U=\tau_g+V\), we have
\[
 \mathcal F(1,\tau_g,g)=0,
 \qquad
 D_U\mathcal F(1,\tau_g,g)=I-M_{\tau_g}K.
\]
Choose \(0<\varepsilon_T\le\varepsilon_{\rm an}\) sufficiently small
that
\[
 \|\tau_g-\1\|_\infty<\frac12
 \qquad\text{whenever}\qquad
 \|g\|_\infty<\varepsilon_T.
\]
Throughout the proof, \(\varepsilon_T\) may be decreased finitely many
times without changing notation. The preceding bound makes \(M_{\tau_g}\) invertible, and
\[
 M_{\tau_g}^{-1}(M_{\tau_g}K)M_{\tau_g}=KM_{\tau_g}.
\]
Hence \(M_{\tau_g}K\) and \(KM_{\tau_g}\) are similar.   Since
\(\chi(\tau_g)=1\), zero is an algebraically simple eigenvalue of
\(I-M_{\tau_g}K\).  Analytic Riesz perturbation
\cite[Chapter~VII, Section~1]{Kato1995} therefore gives holomorphic maps
\[
 g\longmapsto u_g\in\mathscr X_{\mathbb C},
 \qquad
 g\longmapsto\ell_g\in\mathscr X_{\mathbb C}^*,
\]
normalized by
\[
 u_0=u,\qquad
 \ell_0=\ell,\qquad
 \ell_g(u_g)=1,
\]
and satisfying
\[
 \ker(I-M_{\tau_g}K)=\mathbb C u_g,
 \qquad
 \operatorname{ran}(I-M_{\tau_g}K)=\ker\ell_g.
\]
The range identity follows from the Fredholm index zero and the
simplicity of the nullspace.  On the real slice, \(u_g\) may be chosen
strictly positive.

At \(g=0\), the two fold coefficients are
\[
 \ell\bigl(-D_w\mathcal F(1,\1,0)\bigr)=\frac1\alpha,
 \qquad
 \ell\bigl(D_U^2\mathcal F(1,\1,0)[u,u]\bigr)
 =-\frac\beta\alpha.
\]
By the joint holomorphy of
\(\tau_g\), \(u_g\), \(\ell_g\), and \(\mathcal F\), the corresponding
two fold coefficients depend holomorphically on \(g\). Since their values at \(g=0\) displayed above are nonzero, after decreasing
\(\varepsilon_T\) once more, both coefficients are nonzero for every
\(
 g\in B_{\mathscr X_{\mathbb C}}(0,\varepsilon_T).
\)
Applying \Cref{lem:holomorphic-fold} to
\[
 (s,g,V)\longmapsto
 \mathcal F(1-s,\tau_g+V,g)
\]
on this parameter neighborhood gives \(r_T>0\), a holomorphic map
\[
 \gamma:
 B_{\mathscr X_{\mathbb C}}(0,\varepsilon_T)
 \longrightarrow\mathbb C,
 \qquad
 \gamma_0=\sqrt{\frac2\beta},
\]
and jointly holomorphic maps
\[
 R_\pm:
 \{|X|<r_T\}\times
 B_{\mathscr X_{\mathbb C}}(0,\varepsilon_T)
 \longrightarrow\mathscr X_{\mathbb C}
\]
such that
\[
 V_\pm(X,g)
 =
 \pm\gamma_gu_gX+X^2R_\pm(X,g).
\]
These branches exhaust all sufficiently small solutions of
\[
 \mathcal F(1-X^2,\tau_g+V,g)=0.
\]
These branches are real on the real slice.  After decreasing
\(\varepsilon_T\) once more, \(u_g\) and \(\gamma_g\) are strictly
positive for every
\(
 g\in B_{\mathscr X}(0,\varepsilon_T).
\)

After decreasing \(r_T\), if necessary, the least-branch bound selects
the negative branch for every real \(g\in
B_{\mathscr X}(0,\varepsilon_T)\) and \(0<X<r_T\). Thus,
\[
 \widehat T(X,g):=\tau_g+V_-(X,g),
 \qquad |X|<r_T,
\]
is the required joint extension of
\eqref{eq:real-X-tree-branch}.  Taking \(R_T=R_-\) proves
\eqref{eq:field-tree-fold-expansion} and its specialization
\eqref{eq:base-square-root}.

After decreasing \(r_T\) and \(\varepsilon_T\) once more if necessary,
the last assertion of \Cref{lem:holomorphic-fold}, applied to the
negative branch, shows that
\[
 D_U\mathcal F
 \bigl(1-X^2,\widehat T(X,g),g\bigr)
 =
 I-M_{\widehat T(X,g)}K
\]
is invertible whenever
\[
 0<|X|<r_T,
 \qquad
 \|g\|_\infty<\varepsilon_T.
\]

\smallskip
\noindent\emph{The base dented domain.}
At \(g=0\), write
\[
 T(w):=\widehat T_0(w)=\sum_{n\ge1}T_nw^n.
\]
By \Cref{lem:critical-pair-least-branch},
\[
 T_n\ge0,
 \qquad
 \sum_{n\ge1}T_n=\1
 \quad\text{in }\mathscr X.
\]
Hence the series converges uniformly on the closed unit disk.  If
\(|w|=1\) and \(w\ne1\), then
\[
 |T(w;x)|<\sum_{n\ge1}T_n(x)=1,
 \qquad x\in S.
\]
Indeed,
\[
 T_1=e^{-d}>0,
 \qquad
 T_2=e^{-d}K(e^{-d})>0,
\]
and their contributions have different phases.  Strict positivity of
\(T_2\) follows from irreducibility in finite type and from
\textup{(B2)} in compact type.  Compactness gives
\(\|T(w)\|_\infty<1\), and
\[
 |(M_{T(w)}K)^mf|
 \le \|T(w)\|_\infty^mK^m|f|.
\]
Since \(\operatorname{spr}(K)=1\),
\[
 \operatorname{spr}(M_{T(w)}K)<1.
\]
For \(|w|<1\), the same conclusion follows from
\[
 |T(w)|\le T(|w|)\le |w|\1.
\]
Thus \(I-M_{T(w)}K\) is invertible throughout the closed unit disk
away from \(w=1\).

The analytic implicit function theorem continues \(T\) through every
boundary point other than \(1\), while the fold chart supplies a
continuation in a punctured dented neighborhood of \(1\).  A finite
covering of the complementary compact arc then gives the domain
\(\Delta\) in \eqref{eq:base-dented-domain}, chosen so that every
closed subset of \(\Delta\) bounded away from \(1\) has an open
neighborhood on which
\[
 T(w)
 \qquad\text{and}\qquad
 (I-M_{T(w)}K)^{-1}
\]
are holomorphic and bounded.

\smallskip
\noindent\emph{The geometric cover.}
Choose
\(
 0<\delta<\frac{r_T^2}{2}
\)
and define
\[
 \Omega_{\rm out}
 :=\Delta\cap\{|1-w|>\delta\},
 \quad
 \Omega_{\rm loc}
 :=\Delta\cap\{0<|1-w|<2\delta\},
 \quad
 \Omega_{\rm ov}
 :=\Omega_{\rm out}\cap\Omega_{\rm loc}.
\]
Thus
\(
 \Delta=\Omega_{\rm out}\cup\Omega_{\rm loc},
\)
and \(2\delta<r_T^2\) ensures that \(\Omega_{\rm loc}\) lies inside
the fold chart.

Decreasing \(\eta\) and increasing \(\varphi\), if necessary, and
retaining the notation \(\Delta\), we may assume that
\(\overline{\Omega_{\rm out}}\) is compactly contained in the open
base-continuation region constructed above.

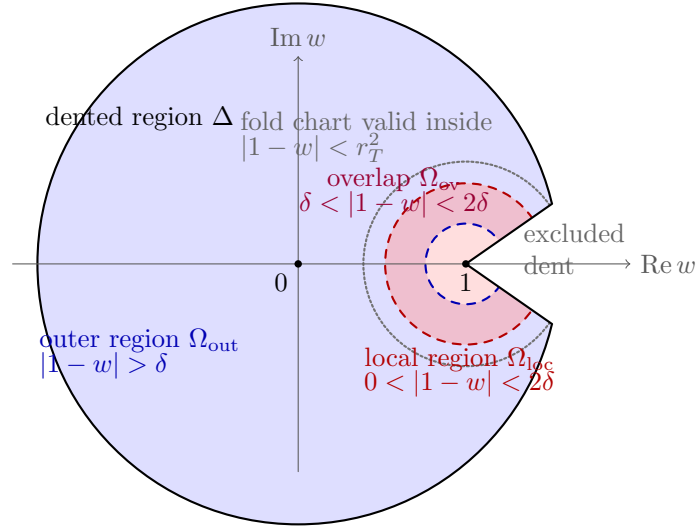
\begin{figure}[htbp]
\centering
\begin{tikzpicture}[
 scale=.82,
 line cap=round,
 line join=round,
 font=\small
]
 \coordinate (O) at (0,0);
 \coordinate (P) at (2.7,0);
 \def\R{4.2}
 \def\rd{0.65}
 \def\rl{1.30}
 \def\rf{1.65}

\begin{scope}
  \clip (O) circle (\R);

  \fill[black!5] (-4.3,-4.3) rectangle (5.2,4.3);

  \fill[blue!13,even odd rule]
   (-4.3,-4.3) rectangle (5.2,4.3)
   (P) circle (\rd);

  \fill[red!13] (P) circle (\rl);

  \fill[purple!25,even odd rule]
   (P) circle (\rl)
   (P) circle (\rd);

  \draw[blue!70!black,dashed,thick] (P) circle (\rd);
  \draw[red!70!black,dashed,thick] (P) circle (\rl);
  \draw[black!55,densely dotted,thick] (P) circle (\rf);
 \end{scope}

 \draw[thick] (O) circle (\R);
 \fill[white]
  (P) -- (5.2,1.75) -- (5.2,-1.75) -- cycle;
 \draw[thick] (P) -- (4.09,.97);
 \draw[thick] (P) -- (4.09,-.97);

 \draw[->,black!65] (-4.6,0) -- (5.35,0)
  node[right] {$\operatorname{Re}w$};
 \draw[->,black!65] (0,-3.35) -- (0,3.35)
  node[above] {$\operatorname{Im}w$};

 \fill (O) circle (1.7pt) node[below left] {$0$};
 \fill (P) circle (1.7pt) node[below] {$1$};

 \node at (-2.55,2.35) {dented region \(\Delta\)};

 \node[blue!70!black,align=left] at (-2.55,-1.45)
  {outer region \(\Omega_{\rm out}\)\\[-1mm]
   \(|1-w|>\delta\)};

 \node[red!70!black,align=center] at (2.6,-1.75)
  {local region \(\Omega_{\rm loc}\)\\[-1mm]
   \(0<|1-w|<2\delta\)};

 \node[purple!80!black,align=center] at (1.55,1.15)
  {overlap \(\Omega_{\rm ov}\)\\[-1mm]
   \(\delta<|1-w|<2\delta\)};

 \node[black!60,align=left] at (4.45,.25)
  {excluded\\dent};

 \node[black!60,align=left] at (1.10,2.05)
  {fold chart valid inside\\[-1mm]
   \(|1-w|<r_T^2\)};
\end{tikzpicture}
\caption{The continuation regions used in the proof
(schematic, not to scale).}
\label{fig:tree-continuation-gluing}
\end{figure}

The blue and red regions cover \(\Delta\); the purple annular sector is
their connected overlap.  We construct the fielded continuation on
the outer and local regions separately and identify the two branches
on this overlap.

\smallskip
\noindent\emph{The outer region.}
After decreasing \(\varepsilon_T\), \Cref{lem:uniform-holomorphic-implicit}, applied to \(\mathcal F(w,U,g)\) and the compact base graph
\[
 \{(w,T(w),0):w\in\overline{\Omega_{\rm out}}\}
\]
gives \(r_{\rm out}>0\) and a jointly holomorphic solution
\[
 U_{\rm out}:
 \Omega_{\rm out}\times
 B_{\mathscr X_{\mathbb C}}(0,\varepsilon_T)
 \longrightarrow\mathscr X_{\mathbb C}.
\]
Here  \(U_{\rm out}(w,g)\) is the unique solution of
\(\mathcal F(w,U,g)=0\) in \( B_{\mathscr X_{\mathbb C}}\bigl(T(w),r_{\rm out}\bigr)\), uniformly for
\[
 w\in\Omega_{\rm out},
 \qquad
 \|g\|_\infty<\varepsilon_T.
\]
The accompanying bound
\eqref{eq:uniform-implicit-bounds} gives
\[
 \sup_{\substack{w\in\Omega_{\rm out}\\
                  \|g\|_\infty<\varepsilon_T}}
 \left\{
  \|U_{\rm out}(w,g)\|_\infty
  +
  \bigl\|(I-M_{U_{\rm out}(w,g)}K)^{-1}\bigr\|
 \right\}<\infty.
\]

\smallskip
\noindent\emph{The local region and the overlap.}
On \(\Omega_{\rm loc}\), define
\[
 U_{\rm loc}(w,g):=\widehat T(X(w),g).
\]
This is jointly holomorphic because
\[
 |X(w)|=|1-w|^{1/2}<\sqrt{2\delta}<r_T.
\]
At \(g=0\), the branches \(U_{\rm loc}\) and \(U_{\rm out}\) agree with
\(T(w)\) on the real interval
\[
 1-2\delta<w<1-\delta,
\]
which is contained in \(\Omega_{\rm ov}\).  Since
\(\Omega_{\rm ov}\) is connected, the identity theorem gives
\[
 U_{\rm loc}(w,0)=U_{\rm out}(w,0)=T(w),
 \qquad w\in\Omega_{\rm ov}.
\]

Since \(|X(w)|\le\sqrt{2\delta}<r_T\) on
\(\Omega_{\rm ov}\), \Cref{lem:nested-product-Cauchy}, applied to
\((X,g)\mapsto\widehat T(X,g)\) on a slightly larger product
neighborhood, gives a uniform bound for \(D_g\widehat T\) after
decreasing \(\varepsilon_T\).  Hence
\[
 U_{\rm loc}(w,g)-T(w)
 =
 \int_0^1
 D_g\widehat T(X(w),sg)[g]\,\dd s,
\]
and therefore
\[
 \sup_{w\in\Omega_{\rm ov}}
 \|U_{\rm loc}(w,g)-T(w)\|_\infty
 \le C\|g\|_\infty.
\]
Decrease \(\varepsilon_T\) once more so that
\[
 C\varepsilon_T<r_{\rm out}.
\]
Then \(U_{\rm loc}(w,g)\) lies in
\(B_{\mathscr X_{\mathbb C}}(T(w),r_{\rm out})\) throughout the
overlap.  Since it solves
\[
 \mathcal F(w,U_{\rm loc}(w,g),g)=0,
\]
the uniqueness property of \(U_{\rm out}\) gives
\[
 U_{\rm loc}=U_{\rm out}
 \qquad\text{on}\qquad
 \Omega_{\rm ov}\times
 B_{\mathscr X_{\mathbb C}}(0,\varepsilon_T).
\]
The two branches therefore glue, according to
\[
 \widehat T_g(w):=
 \begin{cases}
  U_{\rm loc}(w,g),&w\in\Omega_{\rm loc},\\
  U_{\rm out}(w,g),&w\in\Omega_{\rm out},
 \end{cases}
\]
to a jointly holomorphic solution on
\[
 \Delta\times B_{\mathscr X_{\mathbb C}}(0,\varepsilon_T).
\]
Ordinary implicit uniqueness near \(w=0\) shows that this solution
extends the original power-series germ.  Since both sides of
\eqref{eq:tree-chart-identification} are holomorphic on the connected
set
\[
 \Delta\cap\{0<|1-w|<r_T^2\}
\]
and agree on \(\Omega_{\rm loc}\), the identity theorem proves
\eqref{eq:tree-chart-identification} throughout this set.
\end{proof}

\subsection{Core and kernel reduction with uniform path and determinant estimates}

This subsection starts from the normalized rooted tree generating
function constructed above and derives, uniformly for \(g\) near zero,
the generating functions for connected components of each fixed
surplus \(r\).  Its combinatorial ingredient is the classical core and kernel
reduction and the associated labelled path substitution construction;
see \cite{Wright1977,Wright1980,JKLP1993} and
\cite[Sections~II.5 and~VII.5]{FlajoletSedgewick2009}.

We first isolate the local structure obtained by expanding a single
kernel edge.  Let \(x,y\in S\) be its endpoint types, not necessarily
distinct.  The edge may be replaced by
\[
 x=x_0-x_1-\cdots-x_{j-1}-x_j=y,
 \qquad j\ge1,
\]
whose internal vertices carry arbitrary rooted tree attachments; the
resulting object is called a \emph{decorated path}.  Rooted trees
attached to the kernel vertices themselves are accounted for separately
by the kernel vertex factors.  This process is illustrated in
\Cref{fig:decorated-path}.

\begin{figure}[t]
\centering
\begin{tikzpicture}[
  scale=0.90,
  treev/.style={circle,fill=gray!70,inner sep=1.25pt},
  endv/.style={circle,draw=black,fill=white,minimum size=4.8mm,
               inner sep=0pt,font=\scriptsize},
  intv/.style={circle,draw=black,fill=gray!20,minimum size=4.8mm,
               inner sep=0pt,font=\scriptsize},
  every node/.style={transform shape}
]
\node[endv] (ex) at (0,0) {\(x\)};
\node[endv] (ey) at (1.9,0) {\(y\)};
\draw[thick] (ex)--(ey);
\node[font=\scriptsize] at (0.95,-0.65) {one kernel edge};

\draw[->,>=stealth] (2.45,0)--(3.55,0);
\node[font=\scriptsize,align=center] at (3.0,0.55)
  {expand the\\suppressed path};

\node[endv] (x) at (4.1,0) {\(x\)};
\node[intv] (x1) at (5.75,0) {\(x_1\)};
\node[intv] (x2) at (7.55,0) {\(x_2\)};
\node[endv] (y) at (9.2,0) {\(y\)};
\draw[thick] (x)--(x1)--(x2)--(y);
\node[treev] (q1) at (5.35,0.75) {};
\node[treev] (q2) at (4.95,1.35) {};
\node[treev] (q3) at (5.75,1.45) {};
\node[treev] (q4) at (7.95,0.75) {};
\node[treev] (q5) at (7.65,1.45) {};
\node[treev] (q6) at (8.35,1.45) {};
\draw[gray!80] (x1)--(q1)--(q2);
\draw[gray!80] (q1)--(q3);
\draw[gray!80] (x2)--(q4)--(q5);
\draw[gray!80] (q4)--(q6);
\node[font=\scriptsize] at (6.65,-0.65)
  {a path with rooted trees attached at its internal vertices};
\end{tikzpicture}
\caption{Expansion of one kernel edge into a decorated path.  The gray
branches are the rooted trees attached at the internal path vertices.}
\label{fig:decorated-path}
\end{figure}
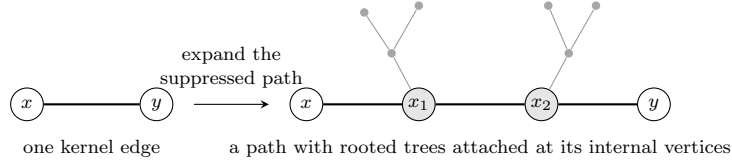

A unicyclic \(2\)-core is a cycle, and summing its closed decorated
paths produces the regularized determinant below.  The multigraph assembly at positive excess is given after the analytic path calculus.

In the normalized coordinate, put
\begin{equation}\label{eq:field-path-operator}
 \mathbf L(w,g):=KM_{\widehat T_g(w)}, \quad\mathsf L(X,g):=KM_{\widehat T(X,g)}.
\end{equation}
The integral kernel of \(\mathbf L(w,g)\) is
\(\kappa(x,y)\widehat T_g(w;y)\): the first factor gives the next path
edge, and the second sums the rooted trees attached at the newly
reached internal vertex.

Since \(\mathbf L(0,0)=0\) and
 \((w,g)\mapsto\mathbf L(w,g)\) is jointly holomorphic near
\((0,0)\), fix a sufficiently small product neighborhood
\(\mathcal U\) of \((0,0)\) and constants \(q<1\), \(C<\infty\) such
that
\begin{equation}\label{eq:germ-operator-bounds}
 \|\mathbf L(w,g)\|_{\mathcal L(L^2(\mu;\mathbb C))}\le q,
 \qquad
 \|\mathbf L(w,g)\|_{\mathfrak S_2(L^2(\mu))}\le C,
 \qquad (w,g)\in\mathcal U.
\end{equation}

For \(j\ge1\), let \(\mathcal P_j(w,g)\in C(S^2)\) be the continuous
kernel of \(\mathbf L(w,g)^{j-1}K\).  It is the type weight of a
decorated path with \(j\) edges and fixed endpoint types; in particular,
\(\mathcal P_1=\kappa\).  On \(\mathcal U\),
\begin{equation}\label{Neumann_sum}
     (I-\mathbf L(w,g))^{-1}K
 =\sum_{j\ge1}\mathbf L(w,g)^{j-1}K
\end{equation}
sums all decorated paths.  The next lemma continues this resolvent to
the common domain and isolates its rank one \(X^{-1}\) singularity in
the Perron direction.

\begin{lemma}[Uniform path and determinant estimates]
\label{lem:uniform-path-determinant}
After shrinking, if necessary, the domain and parameters supplied by
\Cref{lem:normalized-field-tree}, we retain the same notation and put
\[
 \Omega:=\Delta\cap\{|1-w|>\delta\}.
\]

\smallskip
\noindent\textup{\rm(A) Global path and determinant continuation.}
On the global product
\[
 \Delta\times
 B_{\mathscr X_{\mathbb C}}(0,\varepsilon_T),
\]
the path operator
\[
 (I-\mathbf L(w,g))^{-1}K
\]
has a continuous kernel
\[
 \mathcal P(w,g)\in C(S^2)
\]
that depends jointly holomorphically on \((w,g)\).  Moreover,
\begin{equation}\label{def_of_D}
 \mathcal D(w,g):=\det{}_2(I-\mathbf L(w,g))
\end{equation}
is jointly holomorphic and nonzero throughout this product. 

\smallskip
\noindent\textup{\rm(B) Local singular structure.}
On the local product
\[
 \{|X|<r_T\}\times
 B_{\mathscr X_{\mathbb C}}(0,\varepsilon_T),
\]
there exist jointly holomorphic maps
\[
 \Pi:g\longmapsto\Pi_g\in C(S^2),\qquad
 (X,g)\longmapsto E_P(X,g)\in C(S^2),\qquad
 (X,g)\longmapsto\mathcal R_D(X,g)\in\mathbb C.
\]
For \(w\in\Delta\) with \(0<|1-w|<r_T^2\), set
\(X=X(w)\), with the branch fixed in
\Cref{lem:normalized-field-tree}.  Then
\begin{align}
 \mathcal P(w,g)
 &=
 \frac{\Pi_g}{X}+E_P(X,g),
 \label{eq:field-path-pole-short}\\
 \mathcal D(w,g)
 &=
 X\mathcal R_D(X,g).
 \label{eq:field-det2-factorization}
\end{align}
Moreover, \(\mathcal R_D\) is nonzero throughout the local product and
\(\mathcal R_D(0,0)>0\).  At the zero field,
\[
 \Pi_0(x,y)=\frac{u(x)u(y)}{\sqrt{2\beta}}.
\]

\smallskip
\noindent\textup{\rm(C) Logarithms and uniform bounds.}
The functions \(\mathcal D\) and \(\mathcal R_D\) admit unique jointly
holomorphic logarithms on their respective domains, normalized by
\[
 \log\mathcal D(0,g)=0,
 \qquad
 \log\mathcal R_D(0,0)\in\mathbb R,
\]
respectively.

Moreover, the following bounds hold:
\begin{itemize}
    \item On
\(
 \Omega\times
 B_{\mathscr X_{\mathbb C}}(0,\varepsilon_T)
\),
the maps
\[
 \mathcal P,\qquad
 \log\mathcal D,\qquad
 \operatorname{Tr}\mathbf L^2
\]
and their first Fr\'echet derivatives in \(g\) are uniformly bounded
in their natural target norms.

\item On
\(B_{\mathscr X_{\mathbb C}}(0,\varepsilon_T)\), the maps \(\Pi\) and \(D_g\Pi\) are uniformly bounded.  

\item On
\(
 \{|X|<r_T\}\times
 B_{\mathscr X_{\mathbb C}}(0,\varepsilon_T)
\),
 the maps
\[
 \partial_X^jF(X,g),
 \qquad
 \partial_X^jD_gF(X,g)
\]
are uniformly bounded, in their natural target norms, 
for
\[
 F\in
 \left\{
  \widehat T,\,
  E_P,\,
  \log\mathcal R_D,\,
  \operatorname{Tr}\mathsf L^2
 \right\},\qquad j=0,1,2,3.
\]
\end{itemize}
\end{lemma}

\begin{proof}
Throughout the proof, the domains and parameter neighborhoods may be
shrunk finitely many times without changing notation.

\smallskip
\noindent\emph{Step 1: the local path and determinant factorizations.}
By \Cref{lem:continuous-kernel-rank-one-calculus}, applied to
\(\mathsf L(X,g)\), gives a jointly holomorphic simple
eigenvalue \(\lambda(X,g)\) and its Riesz projection
\(\mathsf Q(X,g)\), with
\[
 \mathsf Q(0,g)=\mathsf Q(\tau_g).
\]
Put
\begin{equation}
 \mathsf C(X,g)
 :=
 \mathsf L(X,g)-\lambda(X,g)\mathsf Q(X,g).
\end{equation}
The critical normalization gives
\[
 \lambda(0,g)=1.
\]
Consequently,
\[
 \sigma(X,g):=
 \begin{cases}
  \dfrac{1-\lambda(X,g)}{X},&X\ne0,\\[1.2ex]
  -\partial_X\lambda(0,g),&X=0,
 \end{cases}
\]
is jointly holomorphic.

At the base point, \eqref{eq:base-square-root} gives
\[
 \partial_X\mathsf L(0,0)
 =
 -\sqrt{\frac2\beta}\,KM_u.
\]
The first variation formula for a simple eigenvalue therefore yields
\[
 \sigma(0,0)
 =
 -\ell\!\left(\partial_X\mathsf L(0,0)u\right)
 =
 \sqrt{\frac2\beta}\,\ell(Ku^2)
 =
 \frac{\sqrt{2\beta}}{\alpha}>0.
\]
After decreasing \(r_T\) and \(\varepsilon_T\), \(\sigma\) is
nonzero and \(I-\mathsf C(X,g)\) is invertible throughout the local
parameter neighborhood.

For \(X\ne0\), the Riesz splitting gives
\[
 (I-\mathsf L(X,g))^{-1}K
 =
 \frac{\mathsf Q(X,g)K}{X\sigma(X,g)}
 +(I-\mathsf C(X,g))^{-1}
   (I-\mathsf Q(X,g))K.
\]
Let \(q(X,g)\in C(S^2)\) be the continuous kernel of
\(\mathsf Q(X,g)K\), and let \(e^\perp(X,g)\in C(S^2)\) be the kernel
of the complementary term. Both are jointly holomorphic.
Define
\[
 \Pi_g:=\frac{q(0,g)}{\sigma(0,g)}
\]
and
\[
 E_P(X,g)
 :=
 \frac1X
 \left\{
  \frac{q(X,g)}{\sigma(X,g)}
  -\frac{q(0,g)}{\sigma(0,g)}
 \right\}
 +e^\perp(X,g).
\]
The expression in braces vanishes at \(X=0\), so \(E_P\) extends
across \(X=0\); the extension is jointly holomorphic. Hence the
continuous kernel of the local path resolvent is
\[
 \frac{\Pi_g}{X}+E_P(X,g).
\]

At \(g=0\), the kernel of \(\mathsf Q(0,0)K\) is
\[
 q(0,0;x,y)=\frac{u(x)u(y)}{\alpha}.
\]
Together with the value of \(\sigma(0,0)\), this gives
\[
 \Pi_0(x,y)=\frac{u(x)u(y)}{\sqrt{2\beta}}.
\]

Because \(\mathsf Q\mathsf C=\mathsf C\mathsf Q=0\), the product
formula for \(\det_2\) and the rank-one identity give
\begin{equation}\label{eq:local-det2-factorization}
 \det{}_2(I-\mathsf L(X,g))
 =
 (1-\lambda(X,g))e^{\lambda(X,g)}
 \det{}_2(I-\mathsf C(X,g))
 =
 X\mathcal R_D(X,g),
\end{equation}
where
\[
 \mathcal R_D(X,g)
 :=
 \sigma(X,g)e^{\lambda(X,g)}
 \det{}_2(I-\mathsf C(X,g)).
\]
Thus \(\mathcal R_D\) is jointly holomorphic and nonzero.

At the base point,
\(\mathsf C(0,0)=K-\mathsf Q(0,0)\) is self-adjoint with spectrum in
\([-1,1)\). The real function
\(t\mapsto\det{}_2(I-t\mathsf C(0,0))\) is therefore nonzero on
\([0,1]\) and equals one at \(t=0\), so it is positive there.
Together with \(\sigma(0,0)>0\) and \(\lambda(0,0)=1\), this yields
\(\mathcal R_D(0,0)>0\).

\smallskip
\noindent\emph{Step 2: continuation to the common domain.}
The bounds on closed outer parts in
\Cref{lem:normalized-field-tree}, with threshold \(\delta\), give
uniform invertibility of
\(I-M_{\widehat T_g(w)}K\). Put
\[
 \mathsf A(w,g)
 :=
 (I-M_{\widehat T_g(w)}K)^{-1}.
\]
Both \(\widehat T_g(w)\) and \(\mathsf A(w,g)\) are jointly
holomorphic and uniformly bounded on a slightly larger field
ball. After decreasing \(\varepsilon_T\), Cauchy's estimate in
\cref{lem:nested-product-Cauchy} gives
\begin{equation}\label{eq:outer-tree-bounds}
 \sup_{\substack{w\in\Omega\\
                  \|g\|_\infty<\varepsilon_T}}
 \left\{
  \|\widehat T_g(w)\|_\infty
  +\|D_g\widehat T_g(w)\|
  +\|\mathsf A(w,g)\|
  +\|D_g\mathsf A(w,g)\|
 \right\}<\infty.
\end{equation}
In particular,
\[
 \|\widehat T_g(w)-T(w)\|_\infty
 \le C\|g\|_\infty,
 \qquad w\in\Omega.
\]

The operator identity gives
\begin{equation}\label{eq:outer_path}
 (I-\mathbf L(w,g))^{-1}K
 =
 K+K\mathsf A(w,g)M_{\widehat T_g(w)}K.
\end{equation}
On \(\Omega\), its continuous kernel is
\[
 \begin{aligned}
 \kappa(x,y)
 +\int_S\kappa(x,u)
 \Bigl[
  \mathsf A(w,g)
  \{\widehat T_g(w;\cdot)\kappa(\cdot,y)\}
 \Bigr](u)\,\mu(\dd u).
 \end{aligned}
\]
On
\[
 \Delta\cap\{\delta<|1-w|<r_T^2\},
\]
this kernel and the local kernel
\[
 \frac{\Pi_g}{X(w)}+E_P(X(w),g)
\]
represent the same path operator. Since \(\mu\) has full support,
their continuous kernel representatives coincide. The two formulae therefore glue to the asserted jointly holomorphic kernel
\[
 \mathcal P:
 \Delta\times
 B_{\mathscr X_{\mathbb C}}(0,\varepsilon_T)
 \longrightarrow C(S^2).
\]
This proves \eqref{eq:field-path-pole-short}.

Since \(\mathbf L\) is jointly holomorphic with values in
\(\mathfrak S_2\), \(\mathcal D\) is jointly holomorphic. It remains
to prove global nonvanishing. On the local region
\(0<|1-w|<r_T^2\), \eqref{eq:tree-chart-identification} and \eqref{eq:local-det2-factorization} give
\[
 \mathcal D(w,g)
 =
 X(w)\mathcal R_D(X(w),g),
\]
which is nonzero.

By the choice of the smaller dented domain \(\Delta\), the closed outer
set \(\overline{\Omega}\) lies compactly in the base continuation
region, where \(I-M_{T(w)}K\) is invertible. Thus \(I-KM_{T(w)}\) is invertible as
well.  Hence
\[
 d_\Omega
 :=
 \inf_{w\in\overline{\Omega}}
 \left|\det{}_2(I-KM_{T(w)})\right|>0.
\]
Moreover,
\[
 \|\mathbf L(w,g)-\mathbf L(w,0)\|_{\mathfrak S_2}
 \le
 \|\kappa\|_{L^2(\mu^{\otimes2})}
 \|\widehat T_g(w)-T(w)\|_\infty
 \le C\|g\|_\infty.
\]
The operators in the last display lie in one bounded
Hilbert--Schmidt ball, on which, for some \(C<\infty\),
\[
 |\det{}_2(I-A)-\det{}_2(I-B)|
 \le C\|A-B\|_{\mathfrak S_2}.
\]
After a final decrease of \(\varepsilon_T\), it follows that
\[
 |\mathcal D(w,g)|
 \ge\frac{d_\Omega}{2},
 \qquad
 w\in\Omega,\quad
 \|g\|_\infty<\varepsilon_T.
\]
Since \(\delta<r_T^2\), the local and outer regions cover \(\Delta\).
Thus \(\mathcal D\) is nonzero on the full product.

The global and local product domains are simply connected. Since
\(\widehat T_g(0)=0\), one has \(\mathcal D(0,g)=1\); hence
\(\mathcal D\) has a unique jointly holomorphic logarithm normalized by
\[
 \log\mathcal D(0,g)=0.
\]
Likewise, \(\mathcal R_D(0,0)>0\) selects the unique jointly
holomorphic logarithm satisfying
\[
 \log\mathcal R_D(0,0)\in\mathbb R.
\]

\smallskip
\noindent\emph{Step 3: uniform bounds.}
The bounds in \eqref{eq:outer-tree-bounds} and
\eqref{eq:outer_path} give the asserted bounds for \(\mathcal P\).
The Hilbert--Schmidt estimates
\[
 |\det{}_2(I-B)|
 \le \exp\{C\|B\|_{\mathfrak S_2}^2\},
 \qquad
 |\operatorname{Tr}B^2|
 \le \|B\|_{\mathfrak S_2}^2
\]
give the corresponding bounds for
\(\mathcal D\) and \(\operatorname{Tr}\mathbf L^2\).
Cauchy's estimate in the field variable in
\Cref{lem:nested-product-Cauchy} then yields uniform bounds for
\[
 D_g\mathcal P,\qquad
 D_g\mathcal D,\qquad
 D_g\operatorname{Tr}\mathbf L^2.
\]

The determinant separation gives
\[
 D_g\log\mathcal D
 =
 \frac{D_g\mathcal D}{\mathcal D}.
\]
Together with the bounded base logarithm on
\(\overline{\Omega}\), integration along the segment \(sg\),
\(0\le s\le1\), bounds the chosen logarithm itself and its first field
derivative.

Finally, before the last shrinking, \(\Pi\) and the local factors
\[
 \widehat T,\qquad E_P,\qquad
 \log\mathcal R_D,\qquad \operatorname{Tr}\mathsf L^2
\]
are jointly holomorphic and uniformly bounded on slightly larger
product neighborhoods. Applying
\Cref{lem:nested-product-Cauchy} and then retaining the notation
\(r_T,\varepsilon_T\) for the smaller radii gives the asserted bounds
for \(\Pi,D_g\Pi\) and, for \(0\le j\le3\), for both
\(\partial_X^jF\) and \(\partial_X^jD_gF\).
\end{proof}

We now pass from the rooted tree and path calculus to connected
components.  

For the weights with a field in \eqref{eq:fielded-connected-weight}, put
\[
 G_r(z,g):=\sum_{k\ge1}\CW_g(k,r)\frac{z^k}{k!},
 \qquad
 \widehat G_r(w,g):=G_r(\rho_gw,g).
\]
At \(g=0\), write
\[
 G_r(z):=G_r(z,0);
\]
since \(\rho_0=1\), one has
\(\widehat G_r(w,0)=G_r(w)\).
Rooting an unrooted tree at one of its \(k\) vertices multiplies its
weight by \(k\).  Hence
\begin{equation}\label{eq:field-tree-pointing}
 w\partial_w\widehat G_0(w,g)
 =\int_S\widehat T_g(w;x)\,\mu(\dd x).
\end{equation}
A tree on \(k\) vertices has \(k-1\) edges, so subtracting the
series rooted at an edge from the series rooted at a vertex counts every unrooted
tree once.  Cutting the distinguished edge into two rooted trees gives
\begin{equation}\label{eq:base-unrooted-tree}
 \widehat G_0(w,g)
 =\int_S\widehat T_g(w;x)\,\mu(\dd x)
  -\frac12\iint_{S^2}\kappa(x,y)
       \widehat T_g(w;x)\widehat T_g(w;y)
       \,\mu(\dd x)\mu(\dd y),
\end{equation}
where \(1/2\) forgets the ordering of the two sides.

For positive excess, prune all pendant trees from a connected component
to obtain its \(2\)-core and suppress every maximal path whose internal
vertices have degree two.  The result is a connected multigraph of
minimum degree at least three, called its \emph{multigraph kernel}.
Loops and parallel edges may occur although the original graph is
simple.  This graph object is distinct from the connection kernel
\(\kappa\), which supplies the type weights of its edges.
\Cref{fig:multigraph-kernel-loop} shows both the reduction and the
expansion of a loop of the multigraph kernel.  The latter becomes a closed decorated
path; simplicity of the original graph forces its length to be at least
three.

\begin{figure}[!hb]
\centering
\begin{tikzpicture}[
  scale=0.84,
  corev/.style={circle,fill=black,inner sep=1.3pt},
  treev/.style={circle,fill=gray!70,inner sep=1.25pt},
  endv/.style={circle,draw=black,fill=white,minimum size=4.8mm,
               inner sep=0pt,font=\scriptsize},
  intv/.style={circle,draw=black,fill=gray!20,minimum size=4.8mm,
               inner sep=0pt,font=\scriptsize},
  every node/.style={transform shape}
]
\node[corev] (a) at (0,2.3) {};
\node[corev] (b) at (1.1,3.0) {};
\node[corev] (c) at (2.2,2.3) {};
\node[corev] (d) at (1.1,1.6) {};
\node[corev] (m) at (1.1,2.3) {};
\draw (a)--(b)--(c);
\draw (a)--(m)--(c);
\draw (a)--(d)--(c);
\node[treev] (p1) at (-0.55,2.9) {};
\node[treev] (p2) at (-1.05,3.35) {};
\node[treev] (p3) at (2.75,2.85) {};
\node[treev] (p4) at (3.2,3.3) {};
\node[treev] (p5) at (1.1,0.9) {};
\draw[gray!80] (a)--(p1)--(p2);
\draw[gray!80] (c)--(p3)--(p4);
\draw[gray!80] (d)--(p5);
\node[font=\scriptsize] at (1.1,0.42) {connected graph};

\draw[->,>=stealth] (3.45,2.3)--(4.25,2.3);
\node[font=\scriptsize,align=center] at (3.85,2.75) {prune\\trees};

\node[corev] (a2) at (4.75,2.3) {};
\node[corev] (b2) at (5.75,3.0) {};
\node[corev] (c2) at (6.75,2.3) {};
\node[corev] (d2) at (5.75,1.6) {};
\node[corev] (m2) at (5.75,2.3) {};
\draw (a2)--(b2)--(c2);
\draw (a2)--(m2)--(c2);
\draw (a2)--(d2)--(c2);
\node[font=\scriptsize] at (5.75,0.42) {\(2\)-core};

\draw[->,>=stealth] (7.15,2.3)--(8.05,2.3);
\node[font=\scriptsize,align=center] at (7.60,2.78)
  {suppress\\degree two paths};

\node[endv] (kx) at (8.65,2.3) {};
\node[endv] (ky) at (10.55,2.3) {};
\draw (kx) to[bend left=28] (ky);
\draw (kx)--(ky);
\draw (kx) to[bend right=28] (ky);
\node[font=\scriptsize] at (9.60,0.42) {multigraph kernel};
\node[font=\small] at (-0.65,3.85) {(a)};

\node[endv] (lx) at (0.8,-2.1) {\(x\)};
\draw[thick] (lx) to[out=145,in=35,looseness=5.2] (lx);
\node[font=\scriptsize] at (0.8,-3.0) {one loop of the multigraph kernel};

\draw[->,>=stealth] (2.1,-2.1)--(3.35,-2.1);
\node[font=\scriptsize,align=center] at (2.72,-1.58)
  {expand the\\closed path};

\node[endv] (cx) at (4.15,-2.1) {\(x\)};
\node[intv] (cx1) at (6.0,-1.15) {\(x_1\)};
\node[intv] (cx2) at (7.85,-2.1) {\(x_2\)};
\draw[thick] (cx)--(cx1)--(cx2)--(cx);
\node[treev] (r1) at (5.65,-0.35) {};
\node[treev] (r2) at (5.25,0.22) {};
\node[treev] (r3) at (6.18,0.28) {};
\node[treev] (r4) at (8.2,-1.25) {};
\node[treev] (r5) at (8.65,-0.68) {};
\draw[gray!80] (cx1)--(r1)--(r2);
\draw[gray!80] (r1)--(r3);
\draw[gray!80] (cx2)--(r4)--(r5);
\node[font=\scriptsize,align=center] at (6.0,-3.0)
  {a closed decorated path\\of length at least three};
\node[font=\small] at (-0.65,-0.85) {(b)};
\end{tikzpicture}
\caption{The core and kernel reduction at positive excess and the treatment
of a loop of the multigraph kernel.  Panel~\textup{(a)} produces a multigraph kernel,
whose parallel edges are retained.  Panel~\textup{(b)} expands a
loop into a closed decorated path; reversal accounts for the
factor \(1/2\) in \eqref{eq:base-loop-weight}.}
\label{fig:multigraph-kernel-loop}
\end{figure}
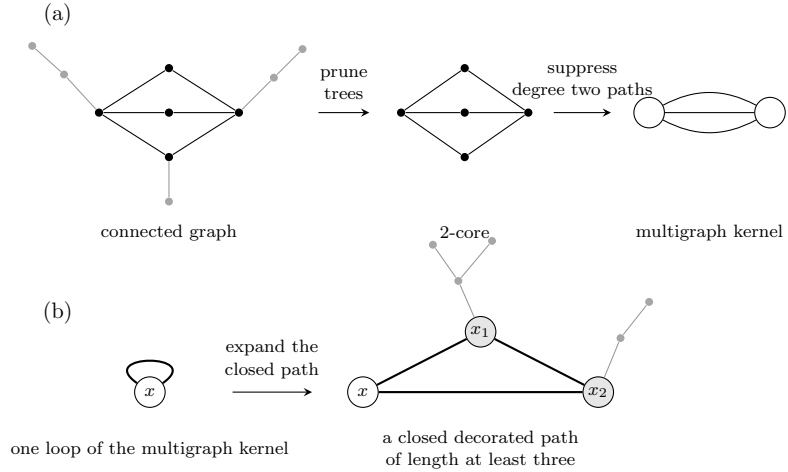

On the common domain furnished by
\Cref{lem:uniform-path-determinant}, recall that
\(\mathcal P_1=\kappa\), and put
\[
 \mathcal P_{\ge2}(w,g)
 :=
 \mathcal P(w,g)-\mathcal P_1.
\]
For an edge class joining distinct kernel vertices of types \(x,y\),
define
\begin{equation}\label{eq:base-parallel-weight}
 \mathfrak P_m(w,g;x,y)
 :=
 \frac{\mathcal P_{\ge2}(w,g;x,y)^m}{m!}
 +\frac{\mathcal P_1(x,y)
        \mathcal P_{\ge2}(w,g;x,y)^{m-1}}{(m-1)!},
 \qquad m\ge1.
\end{equation}
For a loop based at \(x\), define
\begin{equation}\label{eq:base-loop-weight}
 \mathfrak L(w,g;x)
 :=\frac12\{\mathcal P(w,g;x,x)-\mathcal P_1(x,x)
                          -\mathcal P_2(w,g;x,x)\}.
\end{equation}

\begin{lemma}[Exact weighted kernel formula]
\label{lem:kernel-formula}
Let \(H\) be an unlabelled connected multigraph of minimum degree at
least three.  Denote
\(\ell(H):=e(H)-v(H)\) and \(r(H):=\ell(H)+1\).  Write \(m_{ii}\) for its number of loops at vertex \(i\),
\(m_{ij}\) for the number of edges between distinct vertices \(i,j\),
and \(\operatorname{Aut}_v(H)\) for the vertex permutations preserving
these multiplicities.  On \(\mathcal U\), its exact
contribution to \(\widehat G_{r(H)}(w,g)\) is

\begin{align}
 \widehat{\mathcal G}_H(w,g)
 &=\frac1{|\operatorname{Aut}_v(H)|}
   \int_{S^{v(H)}}
   \prod_{i=1}^{v(H)}\widehat T_g(w;x_i)
   \prod_i\frac{\mathfrak L(w,g;x_i)^{m_{ii}}}{m_{ii}!}
   \notag\\*
 &\quad{}\times
   \prod_{\substack{i<j\\m_{ij}\ge1}}
       \mathfrak P_{m_{ij}}(w,g;x_i,x_j)
   \prod_{i=1}^{v(H)}\mu(\dd x_i).
\label{eq:base-kernel-formula}
\end{align}

If \(\ell=e(H)-v(H)>0\) is fixed, only finitely many such kernels
occur, and \(\widehat G_{\ell+1}\) is their sum.
\end{lemma}

\begin{proof} In fact, this is the
weighted version of the classical kernel reduction in
\cite{Wright1977,JKLP1993,FlajoletSedgewick2009}.

Deleting all pendant trees and suppressing every maximal degree-two
path produces a unique multigraph kernel.  Conversely, attach a rooted tree to
each kernel vertex and replace every kernel edge by a nonempty path whose
internal vertices carry rooted tree decorations.  The kernel
\(\mathcal P_j\) accounts for a path of length \(j\), while the factors
\(\widehat T_g(w;x_i)\) account for the kernel vertices.  Labelled
products keep all sets of vertex labels disjoint.

  For the \(m_{ij}\) indistinguishable edges between
two distinct kernel vertices, division by their permutation group gives
the factorials in \eqref{eq:base-parallel-weight}.  Its first term
assigns paths of length at least two to every edge, while its second term
selects the unique path of length one and divides the remaining paths by
their permutation group.

Since the underlying labelled graphs are simple, a kernel loop must
expand into a closed path of length at least three.
The path integral counts each such closed path in its two orientations,
so forgetting the orientation contributes the factor \(1/2\) in
\eqref{eq:base-loop-weight}.  Permuting the \(m_{ii}\) indistinguishable
kernel loops does not change the resulting decoration; division by this
permutation group contributes the factor \(1/m_{ii}!\).

Two decorations of a fixed representative produce the
same labelled graph precisely when they are related by an element of
\(\operatorname{Aut}_v(H)\).  Since every kernel vertex carries a
distinct vertex label, this action is free, and division by
\(|\operatorname{Aut}_v(H)|\) gives the remaining compensation factor
in \eqref{eq:base-kernel-formula}.

The combinatorial factors entering
\eqref{eq:base-kernel-formula} are summarized in
\Cref{tab:kernel-factor-dictionary}. It shows that every connected simple labelled graph
of kernel type \(H\) is counted exactly once, proving
\eqref{eq:base-kernel-formula}.

Finally, minimum degree three implies
\[
 3v(H)\le2e(H),
 \qquad v(H)\le2\ell,
 \qquad e(H)\le3\ell,
\]
so the kernel set is finite for fixed \(\ell\).  Together with
\eqref{Neumann_sum}, this completes the proof.
\end{proof}

\begin{table}[t]
\centering
\caption{Combinatorial origin of the factors in
\eqref{eq:base-kernel-formula}.}
\label{tab:kernel-factor-dictionary}
\small
\renewcommand{\arraystretch}{1.18}
\begin{tabular}{@{}p{0.34\textwidth}p{0.59\textwidth}@{}}
\hline
\textbf{Factor} & \textbf{Combinatorial origin}\\
\hline
\(\widehat T_g(w;x_i)\)
&
The rooted tree attached to kernel vertex \(i\).
\\

\(\mathcal P_j(w,g;x_i,x_j)\)
&
The decorated path of length \(j\) replacing a kernel edge.
\\

\(\mathcal P_{\ge2}^{\,m_{ij}}/m_{ij}!\)
&
None of the \(m_{ij}\) nonloop kernel edges is represented by a direct
edge; all receive indistinguishable paths of length at least two.
\\

\(\mathcal P_1\mathcal P_{\ge2}^{\,m_{ij}-1}/(m_{ij}-1)!\)
&
Exactly one kernel edge is represented by a direct edge; the remaining
\(m_{ij}-1\) paths are indistinguishable.  These are the only two cases
allowed by simplicity.
\\

\(\mathfrak L
 =\tfrac12(\mathcal P-\mathcal P_1-\mathcal P_2)\)
&
A kernel loop expands into a closed path of length at least three;
the factor \(1/2\) forgets its two orientations.
\\

\(\mathfrak L^{m_{ii}}/m_{ii}!\)
&
The \(m_{ii}\) loop edges based at kernel vertex \(i\) are
indistinguishable.
\\

\(1/|\operatorname{Aut}_v(H)|\)
&
Compensation for the vertex automorphisms of the chosen unlabelled
kernel representative.
\\

\(\prod_i\mu(\dd x_i)\)
&
Integration over the type assigned to each kernel vertex.
\\
\hline
\end{tabular}
\end{table}

Finally, we continue all connected generating functions to the common domain.
\begin{lemma}[Connected functions on the common domain]
\label{lem:common-connected-continuation}
Retain the domains $\Delta$ and $\Omega$, and the constant
$\varepsilon_T$, fixed in \Cref{lem:uniform-path-determinant}.  For every
\(R\in\mathbb Z_{\ge0}\), there exists
\(0<\varepsilon_R\le\varepsilon_T\) such that
\(\widehat G_r(w,g)\), \(0\le r\le R\), admit jointly
holomorphic continuations to
\[
 \Delta\times
 B_{\mathscr X_{\mathbb C}}(0,\varepsilon_R).
\]
On this product, \(\widehat G_0\) is given by
\eqref{eq:base-unrooted-tree}.

Recall the definitions in \eqref{eq:field-path-operator} and
\eqref{def_of_D}, and define
\begin{equation}\label{eq:field-unicyclic-det2}
 \widehat G_1(w,g)
 =
 -\frac12\log\mathcal D(w,g)
 -\frac14\operatorname{Tr}\mathbf L(w,g)^2,
\end{equation}
and, for \(2\le r\le R\), \(\widehat G_r\) is the finite sum of the
kernel contributions in \eqref{eq:base-kernel-formula}.

Moreover,
\[
 \sup_{\substack{0\le r\le R,\ w\in\Omega\\
                  \|g\|_\infty<\varepsilon_R}}
 \left\{
  |\widehat G_r(w,g)|
  +\|D_g\widehat G_r(w,g)\|_{\mathscr X_{\mathbb C}^*}
 \right\}<\infty.
\]
\end{lemma}

\begin{proof}
By \eqref{eq:germ-operator-bounds} and the Schatten norm inequality
\[
 \bigl|\operatorname{Tr}\mathbf L(w,g)^m\bigr|
 \le
 \|\mathbf L(w,g)\|_{\mathfrak S_2}^{\,2}
 \|\mathbf L(w,g)\|_{\mathcal L(L^2(\mu;\mathbb C))}^{m-2},
\]
the cycle trace series converges normally and absolutely on
\(\mathcal U\).
 The trace and logarithm identity for \(\det_2\) gives
\[
 \widehat G_1(w,g)
 =
 \frac12\sum_{m\ge3}
 \frac{\operatorname{Tr}\mathbf L(w,g)^m}{m}
 =
 -\frac12\log\mathcal D(w,g)
 -\frac14\operatorname{Tr}\mathbf L(w,g)^2.
\]
The normalization \(\log\mathcal D(0,g)=0\) selects the logarithmic
branch represented by the convergent trace series.

The tree identity \eqref{eq:base-unrooted-tree}, the preceding
unicyclic identity, and the finite sum over multigraph kernels
\eqref{eq:base-kernel-formula} therefore agree with
\(\widehat G_r\) on \(\mathcal U\).  By
\Cref{lem:normalized-field-tree,lem:uniform-path-determinant}, their
right-hand sides are jointly holomorphic on
\[
 \Delta\times
 B_{\mathscr X_{\mathbb C}}(0,\varepsilon_R).
\]
Uniqueness of analytic continuation proves the asserted
continuations.

Finally, on \(\Omega\), the tree formula uses only
\(\widehat T_g\), the unicyclic formula uses
\(\log\mathcal D\) and \(\operatorname{Tr}\mathbf L^2\), and every
formula at positive surplus is a finite sum of products of
\(\widehat T_g\) and \(\mathcal P\).  The outer bounds from
\Cref{lem:normalized-field-tree,lem:uniform-path-determinant}, together with termwise differentiation of these finite products, one
factor at a time, therefore give the displayed uniform estimate.

\end{proof}

\subsection{Wright asymptotics at the zero field}

We now specialize the preceding uniform analytic structure to the zero
field.  For the remainder of this section, write
\[
 X:=X(w),\qquad W:=1-w=X^2,
\]
where \(X(w)\) is the square-root branch fixed in
\Cref{lem:normalized-field-tree}.

\begin{proposition}[Perron--Wright coefficients at the zero field]
\label{prop:base-PW-interface}
For every fixed \(R\in\mathbb Z_{\ge0}\), uniformly for
\(0\le r\le R\),
\begin{equation}\label{eq:base-PW-interface}
 \CW(k,r)
 =
 w_r\beta^{(r-1)/2}e^{-k}
 k^{k+3r/2-2}
 \left\{1+O_R(k^{-1/2})\right\}.
\end{equation}
Here \(w_r\) is exactly the classical Wright constant from the
enumeration of connected labelled graphs with fixed excess
\cite{Wright1977,Wright1980}; with the traditional excess convention
\(\ell=e-v\), it is the constant indexed by \(\ell=r-1\).  In
particular,
\[
 w_0=1,
 \qquad
 w_1=\sqrt{\pi/8}.
\]
\end{proposition}

\begin{proof}
 We first derive the singular forms needed
for coefficient extraction.

For trees, specialize \eqref{eq:field-tree-pointing} to \(g=0\).
By \eqref{eq:base-square-root} and
\(\int_Su\,\dd\mu=1\),
\[
 \int_ST(w;x)\,\mu(\dd x)
 =
 1-\sqrt{\frac2\beta}\,W^{1/2}+O(W).
\]
The fold chart is analytic in \(X\).  Since
\[
 w\partial_w=-(1-W)\partial_W,
\]
termwise integration gives, for some \(b_0\in\mathbb C\),
\[
 G_0(w)
 =
 G_0(1)-W
 +\frac23\sqrt{\frac2\beta}\,W^{3/2}
 +b_0W^2+O(W^{5/2}).
\]

For unicyclic components, specialize
\eqref{eq:field-unicyclic-det2} and
\eqref{eq:field-det2-factorization} to \(g=0\):
\[
 G_1(w)
 =
 -\frac12\log\mathcal D(w,0)
 -\frac14\operatorname{Tr}\mathbf L(w,0)^2,
\]
and
\[
 \mathcal D(w,0)=X\mathcal R_D(X,0).
\]
By \Cref{lem:uniform-path-determinant},
\(\mathcal R_D(X,0)\) is analytic and nonzero near \(X=0\), while
\(\operatorname{Tr}\mathsf L(X,0)^2\) is analytic there.  Thus
\(\log\mathcal D(1-X^2,0)\) is equal to \(\log X\) plus a function
analytic at \(X=0\).  Since \(W=1-w=X^2\), it follows that, for some
\(H_1\in\mathbb C\),
\[
 G_1(w)
 =
 \frac14\log W^{-1}+H_1+O(W^{1/2}).
\]

Now let \(r\ge2\) and put \(\ell=r-1\).  Inserting
\eqref{eq:field-path-pole-short} into
\eqref{eq:base-kernel-formula} shows
that a kernel \(H\) with \(e(H)\) edges contributes
\[
 X^{-e(H)}a_H+O(X^{-e(H)+1}).
\]
For excess \(\ell\), one has \(e(H)\le3\ell\), with equality precisely
for cubic kernels.  Thus every noncubic kernel and every use of a
regular path remainder gains at least one power of \(X\).

A cubic kernel of excess \(\ell\) has \(2\ell\) vertices and
\(3\ell\) edges, and its type contraction is
\[
 \left(\frac1{\sqrt{2\beta}}\right)^{3\ell}
 \left(\int_Su^3\,\dd\mu\right)^{2\ell}
 =
 2^{-3\ell/2}\beta^{\ell/2}.
\]
All remaining loop, multiplicity, orientation, and automorphism factors
are independent of the type space.  Summing over the finitely many
cubic kernels therefore gives a constant \(A_r>0\), depending only on
\(r\), such that
\[
 G_r(w)
 =
 A_r\beta^{(r-1)/2}W^{-3(r-1)/2}
 +O\!\left(W^{-3(r-1)/2+1/2}\right),
 \qquad r\ge2.
\]

For \(r=0\), apply
\eqref{eq:exact-power-coefficient} to the term \(W^{3/2}\);
the terms \(G_0(1)-W+b_0W^2\) have zero \(k\)-th coefficient for all
sufficiently large \(k\).  For \(r=1\),
\eqref{eq:exact-log-coefficient} gives the coefficient of
\(\log W^{-1}\), while the constant \(H_1\) contributes nothing for
\(k\ge1\).  For \(r\ge2\), apply
\eqref{eq:exact-power-coefficient} to
\(W^{-3(r-1)/2}\).

After these terms are subtracted, the corresponding remainders are
\[
 O(W^{5/2}),\qquad
 O(W^{1/2}),\qquad
 O\!\left(W^{-3(r-1)/2+1/2}\right),
\]
respectively.  The common outer bounds from
\Cref{lem:common-connected-continuation} allow
\Cref{lem:uniform-Banach-Delta-transfer} to be applied to these
remainders.  Since \(0\le r\le R\) contains only finitely many values,
the resulting estimates are uniform in \(r\).  Consequently, there are
constants \(C_r>0\) such that, uniformly for \(0\le r\le R\),
\[
 [w^k]G_r(w)
 =
 C_r\beta^{(r-1)/2}k^{3r/2-5/2}
 \left\{1+O_R(k^{-1/2})\right\}.
\]
In the model with one type
\(S=\{\ast\}\), \(\mu=\delta_\ast\), and
\(\kappa(\ast,\ast)=1\), one has
\(\beta=1\) and
\[
 \CW(k,r)=e^{-k}c(k,k+r-1).
\]
Since \(\CW(k,r)=k![w^k]G_r(w)\), comparison, using Stirling's
formula, with the classical Wright asymptotic
\[
 c(k,k+r-1)\sim w_r k^{k+3r/2-2}
 \qquad\text{\cite{Wright1977,Wright1980}}
\]
gives \(C_r=w_r/\sqrt{2\pi}\).  Returning to the general model and
applying Stirling's formula proves \eqref{eq:base-PW-interface}.
\end{proof}
\subsection{Uniform field variation and the Perron--Wright formula}

This subsection transfers the Perron--Wright asymptotic at the zero
field to small external field perturbations.  We first show that the
field derivative of the normalized coefficient
\[
 \rho_g^k\CW_g(k,r)
\]
is bounded on the same Perron--Wright scale as its value at \(g=0\),
uniformly over every fixed finite surplus range.  Integration along
field segments then gives a uniform relative estimate for
\(\|g\|_\infty=O(k^{-1/2})\).  Finally, we specialize this estimate to a
scalar field and prove \Cref{thm:Perron-Wright}.

\begin{lemma}[Uniform normalized field derivative]
\label{lem:one-mark-variation}
For every fixed \(R\in\mathbb Z_{\ge0}\), there are
\(0<\varepsilon_R\le\varepsilon_T\) , \(C_R<\infty\), and
\(k_R\in\mathbb Z_{\ge0}\) such that,
for every \(0\le r\le R\), 
\(f\in\mathscr X\), \(k\ge k_R\), and \(g\in\mathscr X\) with
\(\|g\|_\infty<\varepsilon_R\),
\begin{equation}\label{eq:one-mark-coefficient-bound}
 \left|
  D_g\{\rho_g^k\CW_g(k,r)\}[f]
 \right|
 \le
 C_Re^{-k}k^{k+3r/2-2}\|f\|_\infty.
\end{equation}

\end{lemma}

\begin{proof}
Fix \(R\), choose \(\varepsilon_R\) as in
\Cref{lem:common-connected-continuation}, and decrease it once if
necessary.  All analytic estimates below are taken on the complex
field ball
\(B_{\mathscr X_{\mathbb C}}(0,\varepsilon_R)\).

Coefficientwise differentiation on \(\mathcal U\) gives
\begin{equation}\label{eq:marked-coefficient-identity}
 [w^k]D_g\widehat G_r(w,g)[f]
 =
 \frac1{k!}
 D_g\{\rho_g^k\CW_g(k,r)\}[f].
\end{equation}

\smallskip
\noindent\emph{Step 1: uniform local expansions.}
By \Cref{lem:uniform-path-determinant}\textup{(C)}, the regular local
factors and their first field derivatives admit uniform Taylor
expansions in \(X\).  It remains only to identify the singular orders.

Differentiating \eqref{eq:field-tree-pointing} gives
\[
 w\partial_wD_g\widehat G_0(w,g)[f]
 =
 \int_SD_g\widehat T(X,g)[f](x)\,\mu(\dd x).
\]
Taylor expansion of the right side through order \(X^2\), followed
by termwise integration using
\[
 w\partial_w=-(1-W)\partial_W,
\]
gives
\begin{equation}\label{eq:marked-tree-singularity}
 \begin{aligned}
 D_g\widehat G_0(w,g)[f]
 ={}&B_0(g)[f]+B_2(g)[f]W+b_0(g)[f]W^{3/2}\\
   &+B_4(g)[f]W^2+O(W^{5/2})\|f\|_\infty .
 \end{aligned}
\end{equation}

For unicyclic components,
\eqref{eq:field-det2-factorization} and
\eqref{eq:field-unicyclic-det2} give
\[
 \widehat G_1(1-X^2,g)
 =
 -\frac12\log X+\mathcal H_1(X,g),
\]
where \(\mathcal H_1\) is jointly holomorphic on a common local
product.  Since the logarithmic term is independent of \(g\),
\begin{equation}\label{eq:marked-unicyclic-singularity}
 D_g\widehat G_1(w,g)[f]
 =
 B_1(g)[f]+O(W^{1/2})\|f\|_\infty.
\end{equation}

Finally, let \(r\ge2\) and put \(\ell=r-1\).
By \eqref{eq:field-path-pole-short} and
\Cref{lem:uniform-path-determinant}, both a path factor and its first
field derivative have at most a simple \(X^{-1}\) pole.  Counting the
edges of \(H\) as in the proof of
\Cref{prop:base-PW-interface} therefore gives
\begin{equation}\label{eq:marked-positive-singularity}
 D_g\widehat G_r(w,g)[f]
 =
 b_r(g)[f]W^{-3(r-1)/2}
 +O\!\left(W^{-3(r-1)/2+1/2}\right)\|f\|_\infty,
 \qquad 2\le r\le R.
\end{equation}
All coefficient maps in \cref{eq:marked-tree-singularity,eq:marked-unicyclic-singularity,eq:marked-positive-singularity} are
holomorphic maps into
\(\mathscr X_{\mathbb C}^*\).  Indeed, they are obtained by evaluating
finitely many \(X\)-derivatives of the regular local factors controlled
by \Cref{lem:uniform-path-determinant}.  Since only finitely many
kernels occur for
\(2\le r\le R\), after decreasing \(\varepsilon_R\) once more,
\[
 \sup_{\|g\|_\infty<\varepsilon_R}
 \left\{
  \max_{j\in\{0,1,2,4\}}\|B_j(g)\|
  +\|b_0(g)\|
  +\max_{2\le r\le R}\|b_r(g)\|
 \right\}
 <\infty,
\]
  where the last maximum is omitted when \(R<2\). Similarly, the remainder bounds in
the preceding expansions are uniform on the same field ball \(B_{\mathscr X_{\mathbb C}}(0,\varepsilon_R)\).

\smallskip
\noindent\emph{Step 2: coefficient extraction.}
Set
\[
 F_{r,g}(w):=D_g\widehat G_r(w,g)
 \in\mathscr X_{\mathbb C}^*,
\]
and define \(H_{r,g}\) and \(\nu_r\) by
\[
\begin{array}{c|c|c}
 r&H_{r,g}(w)&\nu_r\\ \hline
 0&
 F_{0,g}(w)-B_0(g)-B_2(g)W&
 \dfrac32\\[1mm]
 1&
 F_{1,g}(w)-B_1(g)&
 \dfrac12\\[1mm]
 2\le r\le R&
 F_{r,g}(w)&
 -\dfrac{3(r-1)}2 
\end{array}
\]

By
\eqref{eq:marked-tree-singularity},
\eqref{eq:marked-unicyclic-singularity}, and
\eqref{eq:marked-positive-singularity}, there is a constant \(C<\infty\)
such that
\[
 \|H_{r,g}(w)\|_{\mathscr X_{\mathbb C}^*}
 \le C|1-w|^{\nu_r},
\]
uniformly for \(0\le r\le R\),
\(\|g\|_\infty<\varepsilon_R\), and \(w\) sufficiently close to \(1\)
in \(\Delta\).

On \(\Omega\), the required uniform bound follows from
\Cref{lem:common-connected-continuation}; the finitely many polynomial
terms subtracted for \(r=0,1\) are uniformly bounded there as well.
After shrinking the dented domain and adjusting the cutoff if necessary,
these local and outer bounds verify the hypotheses of
\Cref{lem:uniform-Banach-Delta-transfer}, uniformly in \(r\) and \(g\).

Applying that lemma to \(H_{r,g}\), uniformly over the finite surplus
range, gives
\begin{equation}\label{eq:marked-coefficient-orders}
 \|[w^k]F_{r,g}\|_{\mathscr X_{\mathbb C}^*}
 =
 \begin{cases}
  O(k^{-5/2}),&r=0,\\
  O(k^{-3/2}),&r=1,\\
  O\!\left(k^{3(r-1)/2-1}\right),&2\le r\le R.
 \end{cases}
\end{equation}
The polynomial terms subtracted for \(r=0,1\) have zero \(k\)-th
coefficient for all sufficiently large \(k\).

Combining \eqref{eq:marked-coefficient-identity},
\eqref{eq:marked-coefficient-orders}, and Stirling's formula gives
\[
 \left|
  D_g\{\rho_g^k\CW_g(k,r)\}[f]
 \right|
 \le
 \begin{cases}
  O_R(e^{-k}k^{k-2}),&r=0,\\
  O_R(e^{-k}k^{k-1}),&r=1,\\
  O_R(e^{-k}k^{k+3r/2-2}),&2\le r\le R,
 \end{cases}
 \|f\|_\infty.
\]
The unicyclic estimate is stronger than required, and all three cases
imply \eqref{eq:one-mark-coefficient-bound}.
\end{proof}

\begin{corollary}[Local and shrinking field variation]
\label{prop:shrinking-field-variation}
For every fixed \(R\in\mathbb Z_{\ge0}\), there are
\(\delta_R,c_R,C_R>0\) such that, for all sufficiently large
\(k\), \(0\le r\le R\), and every real
\(g\in\mathscr X\) with \(\|g\|_\infty\le\delta_R\),
\begin{equation}\label{eq:local-field-variation}
 \left|
  \frac{\rho_g^k\CW_g(k,r)}{\CW(k,r)}-1
 \right|
 \le C_R\|g\|_\infty,
 \qquad
 c_R\le\frac{\rho_g^k\CW_g(k,r)}{\CW(k,r)}\le C_R.
\end{equation}
Consequently, for every fixed \(U<\infty\),
\begin{equation}\label{eq:shrinking-field-variation}
 \sup_{\substack{0\le r\le R\\
                  \|g\|_\infty\le Uk^{-1/2}}}
 \left|
  \frac{\rho_g^k\CW_g(k,r)}{\CW(k,r)}-1
 \right|=O_{R,U}(k^{-1/2}).
\end{equation}
\end{corollary}

\begin{proof}
Integrate \eqref{eq:one-mark-coefficient-bound} along the segment
\(sg\), \(0\le s\le1\).  This gives
\[
 |\rho_g^k\CW_g(k,r)-\CW(k,r)|
 \le C_Re^{-k}k^{k+3r/2-2}\|g\|_\infty
\]
for all sufficiently large \(k\), uniformly for \(0\le r\le R\) and
real \(g\) with \(\|g\|_\infty<\varepsilon_R\).

Since the leading constants in \eqref{eq:base-PW-interface} are
positive and \(0\le r\le R\) is a finite range, that estimate gives a
uniform positive lower bound for \(\CW(k,r)\) on the same
Perron--Wright scale.  After increasing \(C_R\) and the lower threshold
for \(k\), if necessary, division therefore gives
\[
 \left|
  \frac{\rho_g^k\CW_g(k,r)}{\CW(k,r)}-1
 \right|
 \le C_R\|g\|_\infty.
\]
We may assume that \(C_R\ge3/2\), and choose
\[
 0<\delta_R<
 \min\left\{\varepsilon_R,\frac1{2C_R}\right\}.
\]
Then, whenever \(\|g\|_\infty\le\delta_R\),
\[
 \frac12
 \le
 \frac{\rho_g^k\CW_g(k,r)}{\CW(k,r)}
 \le
 \frac32.
\]
Thus \eqref{eq:local-field-variation} holds with \(c_R=1/2\).

Finally, for every fixed \(U<\infty\),
\(Uk^{-1/2}\le\delta_R\) for all sufficiently large \(k\).  Applying
the preceding estimate uniformly over
\(\|g\|_\infty\le Uk^{-1/2}\) gives
\eqref{eq:shrinking-field-variation}.
\end{proof}

We can now combine the base Perron--Wright estimate with the preceding
field variation. For \(t>0\), put
\begin{equation}
\label{eq:scalar-change-of-field}
 g_t:=(\log t)\1-(t-1)d.
\end{equation}
Choose \(0<\delta_\rho<1/2\) so that
\[
 (2+\|d\|_\infty)\delta_\rho<\varepsilon_{\rm an},
\]
and define
\[
 \rho(t):=\rho_{g_t},
 \qquad |t-1|<\delta_\rho.
\]
\begin{proof}[Proof of \Cref{thm:Perron-Wright}]
By \Cref{lem:critical-pair}, after restricting \(t\) to a sufficiently
small neighborhood of \(1\), the map
\(t\mapsto g_t\) is real analytic, and \(\rho(t)\) is positive and real
analytic for real \(t\) near \(1\), with \(\rho(1)=1\).
Moreover, \Cref{lem:critical-pair-least-branch}, applied with
\(g=g_t\), identifies \(\rho(t)\) as the radius of convergence of the
least rooted-tree branch solving
\[
 T=z\exp\{g_t-d+KT\}.
\]

We first compute the expansion of \(\rho(t)\) at \(t=1\).
Recall that
\[
 a(g):=\log\rho_g,
 \qquad
 p(f):=\int_Su f\,\dd\mu,
\]
and put
\[
 b:=\1-d,
 \qquad
 v_b:=D_g\tau_0[b].
\]
Symmetry of \(K\), together with \(Ku=u\), gives
\[
 p(d)
 =
 \int_SuK\1\,\dd\mu
 =
 \int_S(Ku)\,\dd\mu
 =
 1.
\]
Hence \(p(b)=0\), and \eqref{eq:source-radius-derivative} yields
\[
 D_ga(0)[b]=0.
\]

The simple-eigenvalue derivative computed in the proof of
\Cref{lem:critical-pair} is
\[
 D_\tau\chi(\1)[v]
 =
 \frac1\alpha\int_Su^2v\,\dd\mu.
\]
Differentiating the logarithmic form of
\eqref{eq:critical-tree-field} in the direction \(b\), and
differentiating the constraint \(\chi(\tau_g)=1\), therefore gives
\[
 (I-K)v_b=b,
 \qquad
 \int_Su^2v_b\,\dd\mu=0.
\]
Since
\[
 (I-K)\1=b,
 \qquad
 \ker(I-K)=\mathbb Cu,
\]
we obtain
\[
 v_b=\1-\frac{\alpha}{\beta}u.
\]

Write
\[
 w_b:=D_g^2\tau_0[b,b].
\]
A second differentiation of the logarithmic critical equation gives
\[
 (I-K)w_b
 =
 D_g^2a(0)[b,b]\1+v_b^2.
\]
Pairing with \(u\), and using \(Ku=u\), yields
\[
 D_g^2a(0)[b,b]
 =
 -p(v_b^2)
 =
 \frac{\alpha^2}{\beta}-1.
\]
On the other hand,
\[
 g_t
 =
 (t-1)b-\frac12(t-1)^2\1
 +O\!\left(|t-1|^3\right),
\]
while \eqref{eq:source-radius-derivative} gives
\[
 D_ga(0)[\1]=-1.
\]
Taylor's formula now implies
\[
 \log\rho(t)
 =
 a(g_t)
 =
 \frac{\alpha^2}{2\beta}(t-1)^2
 +O\!\left(|t-1|^3\right).
\]
This proves \eqref{eq:Perron-Wright-radius}.

We next prove the coefficient estimate.  Fix
\(R\in\mathbb Z_{\ge0}\) and a sequence \(\eta_k\to0\).
Uniformly for \(|t-1|\le\eta_k\),
\[
 \|g_t\|_\infty=O(\eta_k).
\]
Consequently, \eqref{eq:local-field-variation} gives, uniformly for
\(0\le r\le R\),
\begin{equation}\label{eq:scalar-local-field-response}
 \frac{\rho(t)^k\CW_{g_t}(k,r)}{\CW(k,r)}
 =
 1+O_R(\eta_k).
\end{equation}

Every graph counted by \(\CW_t(k,r)\) has \(k+r-1\) edges, and
\[
 e^{g_t(x)-d(x)}
 =
 t e^{-td(x)}.
\]
Comparison of the two weights configurationwise therefore gives the
exact scalar-gauge identity
\begin{equation}\label{eq:scalar-gauge-identity}
 \CW_t(k,r)=t^{r-1}\CW_{g_t}(k,r).
\end{equation}
Since \(0\le r\le R\) and \(|t-1|\le\eta_k\),
\[
 t^{r-1}=1+O_R(\eta_k).
\]
Combining this estimate with
\eqref{eq:scalar-local-field-response} gives
\[
 \CW_t(k,r)
 =
 \rho(t)^{-k}\CW(k,r)
 \left\{1+O_R(\eta_k)\right\}.
\]
Finally, the zero-field estimate \eqref{eq:base-PW-interface} gives,
uniformly for \(0\le r\le R\),
\[
 \CW(k,r)
 =
 w_r\beta^{(r-1)/2}e^{-k}
 k^{\,k+3r/2-2}
 \left\{1+O_R(k^{-1/2})\right\}.
\]
Multiplying the last two estimates proves
\[
 \CW_t(k,r)
 =
 w_r\beta^{(r-1)/2}e^{-k}\rho(t)^{-k}
 k^{\,k+3r/2-2}
 \left\{
  1+O_R\!\left(\eta_k+k^{-1/2}\right)
 \right\},
\]
uniformly for \(0\le r\le R\) and \(|t-1|\le\eta_k\).
This is \eqref{eq:Perron-Wright-main}.

Taking \(\eta_k=Lk^{-1/2}\) in
\eqref{eq:Perron-Wright-main} and using the expansion \eqref{eq:Perron-Wright-radius} proves
\eqref{eq:Perron-Wright-gaussian}.
\end{proof}

\subsection{Perron--Wright stability along finite type approximations}

The preceding estimates concern a fixed critical model.  In the
finite type approximations used later, however, both the type law and
the critical kernel vary with the ambient size.  We therefore state a
sequential stability result in the following lemma.

\begin{lemma}[Finite parameter stability along a sequence]
\label{lem:finite-parameter-stability}
Fix a finite type set and a critical pair
\(\vartheta=(\mu,\kappa)\) with positive type masses and irreducible
support.  Let \(\vartheta_n=(\mu_n,\kappa_n)\) be critical pairs with
the same positive support and
\[
 \delta_n
 :=
 \|\mu_n-\mu\|_1+\|\kappa_n-\kappa\|_{\max}
 \longrightarrow0.
\]
For each \(n\), let \(\mathcal K_n\subset\mathbb N\) be finite and
nonempty, put
\[
 \underline k_n:=\min\mathcal K_n,
 \qquad
 \overline k_n:=\max\mathcal K_n,
\]
and assume
\[
 \underline k_n\longrightarrow\infty,
 \qquad
 \overline k_n\delta_n\longrightarrow0.
\]
Let \(p_n,\alpha_n,\beta_n\) denote the Perron quantities of
\(\vartheta_n\), and let \(\rho_{n,g}\) denote its field radius.  Then
\[
 p_n\longrightarrow p,
 \qquad
 \alpha_n\longrightarrow\alpha,
 \qquad
 \beta_n\longrightarrow\beta,
\]
and the following assertions hold.

\begin{enumerate}[label=\textup{(\arabic*)}]

\item\label{item:actual-radius}
Put
\[
 a_n(g):=\log\rho_{n,g},
\]
where the logarithm is the branch normalized by \(a_n(0)=0\).
There are \(\varepsilon,C>0\), independent of all sufficiently large
\(n\), such that
\begin{equation}\label{eq:actual-source-radius-Taylor}
 \left|
  a_n(g+h)-a_n(g)+p_n(h)
 \right|
 \le
 C\bigl(
   \|h\|_\infty^2
   +\|g\|_\infty\|h\|_\infty
 \bigr)
\end{equation}
whenever \(g,h,g+h\) belong to
\(B_{\mathscr X_{\mathbb C}}(0,\varepsilon)\).

\item
For every fixed \(R<\infty\),
\begin{equation}\label{eq:actual-sequence-base-PW}
 \sup_{\substack{k\in\mathcal K_n\\0\le r\le R}}
 \left|
  \frac{\CW^{\vartheta_n}(k,r)}
  {w_r\beta_n^{(r-1)/2}
   e^{-k}k^{k+3r/2-2}}-1
 \right|
 \longrightarrow0.
\end{equation}

\item\label{item:actual-variation}
For every fixed \(R<\infty\), there are
\(\delta_R,c_R,C_R>0\), independent of all sufficiently large \(n\),
such that
\[
 c_R
 \le
 \frac{\rho_{n,g}^k\CW_g^{\vartheta_n}(k,r)}
      {\CW^{\vartheta_n}(k,r)}
 \le
 C_R
\]
uniformly for
\[
 k\in\mathcal K_n,\qquad
 0\le r\le R,\qquad
 g\in\mathscr X,\qquad
 \|g\|_\infty\le\delta_R.
\]
Furthermore, for every fixed \(R,U<\infty\),
\begin{equation}\label{eq:actual-sequence-shrinking-variation}
 \sup_{\substack{k\in\mathcal K_n,\ 0\le r\le R\\
                  \|g\|_\infty\le Uk^{-1/2}}}
 \left|
  \frac{\rho_{n,g}^k\CW_g^{\vartheta_n}(k,r)}
       {\CW^{\vartheta_n}(k,r)}-1
 \right|
 \longrightarrow0.
\end{equation}

\end{enumerate}
\end{lemma}

\begin{proof}
\smallskip
\noindent\emph{Step 1: uniform quantities for the field and radius.}
Since the type set is finite and the positive support is fixed, let
\(\vartheta'=(\mu',\kappa')\) vary near \(\vartheta=(\mu,\kappa)\).
The type masses \(\mu'(a)\) and the independent kernel values
\(\kappa'(a,b)\) on the fixed support form a finite-dimensional
parameter vector, with the mass normalization and kernel symmetry
understood.  After complexifying these coordinates, the
implicit-function argument in the proof of
\Cref{lem:critical-pair} applies jointly to
\((\vartheta',g)\) on a product neighborhood of
\((\vartheta,0)\). After shrinking this product
neighborhood, there exist \(\varepsilon,C>0\), independent of all
sufficiently large \(n\), such that \(a_n\) is holomorphic on
\(
 B_{\mathscr X_{\mathbb C}}(0,\varepsilon)
\)
and
\[
 \sup_{\|g\|_\infty<\varepsilon}
 \left\{
  \|D_g^2a_n(g)\|
  +\|D_g^3a_n(g)\|
 \right\}
 \le C.
\]

Ordinary finite-dimensional Perron perturbation gives
\[
 p_n\longrightarrow p,
 \qquad
 \alpha_n\longrightarrow\alpha,
 \qquad
 \beta_n\longrightarrow\beta.
\]
Since \(D_ga_n(0)=-p_n\), Taylor's formula gives
\[
 \begin{aligned}
  a_n(g+h)-a_n(g)+p_n(h)
  &=
  \bigl(D_ga_n(g)-D_ga_n(0)\bigr)[h]\\
  &\quad+
  \int_0^1(1-s)
   D_g^2a_n(g+sh)[h,h]\,\dd s.
 \end{aligned}
\]
Applying the same uniform second-derivative bound to
\(D_ga_n(g)-D_ga_n(0)\), and increasing \(C\) if necessary, proves
\eqref{eq:actual-source-radius-Taylor} uniformly for all sufficiently
large \(n\).

\smallskip
\noindent\emph{Step 2: positive comparison.}
All order comparisons in Steps~2--4 are made on the
real field space \(\mathscr X\).

Write
\[
 d:=\kappa\mu,
 \qquad
 d_n:=\kappa_n\mu_n,
\]
and define the effective field
\[
 h_n(a):=
 \log\frac{\mu_n(a)}{\mu(a)}+d(a)-d_n(a).
\]
Put \(m_\mu:=\min_a\mu(a)>0\).  For all sufficiently large \(n\),
\(\mu_n(a)\ge m_\mu/2\) for every type \(a\), and hence
\[
 \left\|
  \log\frac{\mu_n}{\mu}
 \right\|_\infty
 \le
 \frac{2}{m_\mu}\|\mu_n-\mu\|_1.
\]
Moreover,
\[
 \|d_n-d\|_\infty
 \le
 \|\kappa_n-\kappa\|_{\max}
 +\|\kappa\|_{\max}\|\mu_n-\mu\|_1.
\]
Consequently, for a constant \(C_h<\infty\) depending only on the
limiting model,
\begin{equation}\label{eq:effective-field-bound}
 \|h_n\|_\infty\le C_h\delta_n.
\end{equation}

On the common positive edge support, put
\[
 \epsilon_n:=
 \max_{\kappa(a,b)>0}
 \left|
  \frac{\kappa_n(a,b)}{\kappa(a,b)}-1
 \right|.
\]
The positive entries of \(\kappa\) are bounded away from zero, so
\begin{equation}
    \label{eq:actual-sequence-edge-perturbation}
 \epsilon_n=O(\delta_n),
 \qquad
 (1-\epsilon_n)\kappa
 \le\kappa_n
 \le(1+\epsilon_n)\kappa
\end{equation}
on that support.

Now fix \(R<\infty\).  Decrease the radius \(\varepsilon_R\) from
\Cref{lem:one-mark-variation}, if necessary, so that
\eqref{eq:local-field-variation} also holds on the real field ball
\(B_{\mathscr X}(0,\varepsilon_R)\).  Let
\(\varepsilon>0\) be the common complex field radius furnished
by Step~1, and set
\[
 \delta_R:=
 \frac14\min\{\varepsilon,\varepsilon_R\}.
\]
Since \(\|h_n\|_\infty\to0\) and \(\epsilon_n\to0\), there exists
\(n_R<\infty\) such that, for every \(n\ge n_R\),
\[
 \|h_n\|_\infty\le\delta_R,
 \qquad
 \epsilon_n\le\frac12.
\]
Thus, for 
\(\|g\|_\infty\le\delta_R\),
\begin{equation}\label{eq:actual-sequence-field-segment}
 \|h_n+sg\|_\infty
 \le2\delta_R
 \le\frac12\min\{\varepsilon,\varepsilon_R\},
 \qquad 0\le s\le1.
\end{equation}

Let \(W_{\vartheta,g}(H,\boldsymbol x)\) denote the integrand in
\eqref{eq:fielded-connected-weight} for the model \(\vartheta\).
The definition of \(h_n\) gives the exact identity
\[
 W_{\vartheta_n,g}(H,\boldsymbol x)
 =
 W_{\vartheta,g+h_n}(H,\boldsymbol x)
 \prod_{\{i,j\}\in E(H)}
 \frac{\kappa_n(x_i,x_j)}{\kappa(x_i,x_j)}.
\]
Configurations containing an edge with zero weight occur on neither
side.  Since every graph counted by \(\CW_g(k,r)\) has \(k+r-1\)
edges, positivity and summation give
\begin{equation}\label{eq:actual-sequence-weight-sandwich}
 (1-\epsilon_n)^{k+r-1}
 \CW_{g+h_n}^{\vartheta}(k,r)
 \le
 \CW_g^{\vartheta_n}(k,r)
 \le
 (1+\epsilon_n)^{k+r-1}
 \CW_{g+h_n}^{\vartheta}(k,r).
\end{equation}

The same comparison applies to rooted trees, with the root type kept
fixed.  If \(T_{n,k}(g;a)\) and \(T_k(g+h_n;a)\) denote the respective
\(z^k\)-coefficients at root type \(a\), then
\[
 \begin{aligned}
 &(1-\epsilon_n)^{k-1}
   \mu(a)T_k(g+h_n;a)\\
 &\qquad\le
   \mu_n(a)T_{n,k}(g;a)\\
 &\qquad\le
   (1+\epsilon_n)^{k-1}
   \mu(a)T_k(g+h_n;a).
 \end{aligned}
\]
After summing over \(a\), the scalar series have the same radii as
their corresponding finite systems of rooted-tree series.
The Cauchy--Hadamard formula therefore gives
\begin{equation}\label{eq:actual-sequence-tree-radius-sandwich}
 \frac{\rho_{g+h_n}}{1+\epsilon_n}
 \le
 \rho_{n,g}
 \le
 \frac{\rho_{g+h_n}}{1-\epsilon_n}.
\end{equation}
Combining
\eqref{eq:actual-sequence-weight-sandwich} and
\eqref{eq:actual-sequence-tree-radius-sandwich} gives, uniformly for
\(0\le r\le R\) and \(g\in\mathscr X\) with \(\|g\|_\infty\le\delta_R\),
\begin{equation}\label{eq:actual-sequence-normalized-sandwich}
 \exp\{-C_Rk\epsilon_n\}
 \le
 \frac{\rho_{n,g}^k\CW_g^{\vartheta_n}(k,r)}
      {\rho_{g+h_n}^k\CW_{g+h_n}^{\vartheta}(k,r)}
 \le
 \exp\{C_Rk\epsilon_n\}.
\end{equation}

\smallskip
\noindent\emph{Step 3: zero field coefficient transport.}
Taking \(g=0\) in
\eqref{eq:actual-sequence-normalized-sandwich} gives
\[
 \CW^{\vartheta_n}(k,r)
 =
 \rho_{h_n}^k\CW_{h_n}^{\vartheta}(k,r)
 \exp\{O_R(k\epsilon_n)\}.
\]
By \eqref{eq:actual-sequence-field-segment},
\(h_n\in B_{\mathscr X}(0,\varepsilon_R)\) for all sufficiently large
\(n\).  Applying \eqref{eq:local-field-variation} with \(g=h_n\) gives
\[
 \rho_{h_n}^k\CW_{h_n}^{\vartheta}(k,r)
 =
 \CW^{\vartheta}(k,r)
 \{1+O_R(\|h_n\|_\infty)\}.
\]
Since
\[
 \|h_n\|_\infty=O(\delta_n),
 \qquad
 \sup_{k\in\mathcal K_n}k\epsilon_n
 =
 O(\overline k_n\delta_n)=o(1),
\]
we obtain, uniformly for \(k\in\mathcal K_n\) and \(0\le r\le R\),
\begin{equation}\label{eq:actual-sequence-base-transport}
 \CW^{\vartheta_n}(k,r)
 =
 \CW^{\vartheta}(k,r)\{1+o(1)\}.
\end{equation}
Combining this with \eqref{eq:base-PW-interface} and
\(\beta_n\to\beta\) proves
\eqref{eq:actual-sequence-base-PW}.

\smallskip
\noindent\emph{Step 4: field variation.}
For the fixed base model set
\[
 R_{k,r}(q):=
 \frac{\rho_q^k\CW_q^{\vartheta}(k,r)}
      {\CW^{\vartheta}(k,r)},
\]
and for the current model set
\[
 R_{n,k,r}(q):=
 \frac{\rho_{n,q}^k\CW_q^{\vartheta_n}(k,r)}
      {\CW^{\vartheta_n}(k,r)}.
\]
Since both \(\vartheta\) and \(\vartheta_n\) are critical,
\(\rho_0=\rho_{n,0}=1\).  Dividing
\eqref{eq:actual-sequence-normalized-sandwich} at \(g\) by the same
estimate at \(g=0\) yields
\begin{equation}\label{eq:actual-sequence-variation-sandwich}
 e^{-2C_Rk\epsilon_n}
 \frac{R_{k,r}(g+h_n)}{R_{k,r}(h_n)}
 \le
 R_{n,k,r}(g)
 \le
 e^{2C_Rk\epsilon_n}
 \frac{R_{k,r}(g+h_n)}{R_{k,r}(h_n)}.
\end{equation}

The derivative estimate from \Cref{lem:one-mark-variation} for the fixed base model
, divided by the base lower bound in
\eqref{eq:base-PW-interface}, gives a common bound for
\(D_qR_{k,r}(q)\) on \(B_{\mathscr X}(0,\varepsilon_R)\).
By \eqref{eq:actual-sequence-field-segment}, the entire segment
\(h_n+sg\), \(0\le s\le1\), remains in this ball.  Integration along
that segment therefore gives
\[
 |R_{k,r}(g+h_n)-R_{k,r}(h_n)|
 \le C_R\|g\|_\infty.
\]
Together with the upper and lower bounds in
\eqref{eq:local-field-variation},
\eqref{eq:actual-sequence-variation-sandwich}, and \eqref{eq:actual-sequence-edge-perturbation},
one has, uniformly for all sufficiently large \(n\), \(k\in\mathcal K_n\), \(0\le r\le R\),
and \(\|g\|_\infty\le\delta_R\),
\begin{equation}\label{eq:actual-sequence-variation-error}
 c_R\le R_{n,k,r}(g)\le C_R,
 \qquad
 |R_{n,k,r}(g)-1|
 \le C_R\bigl(\|g\|_\infty+k\delta_n\bigr).
\end{equation}
Since
\(U\underline k_n^{-1/2}\le\delta_R\)
for all sufficiently large \(n\),
\eqref{eq:actual-sequence-variation-error} gives
\[
 \sup_{\substack{k\in\mathcal K_n,\ 0\le r\le R\\
                   \|g\|_\infty\le Uk^{-1/2}}}
  |R_{n,k,r}(g)-1|\le
 C_{R,U}\left(
  \underline k_n^{-1/2}
  +\overline k_n\delta_n
 \right)
 \longrightarrow0,
\]
which proves
\eqref{eq:actual-sequence-shrinking-variation}.
\end{proof}

\begin{corollary}[Scalar Perron--Wright asymptotics along the sequence]
\label{cor:actual-sequence-PW}
Under the assumptions of
\Cref{lem:finite-parameter-stability}, define
\[
 d_n:=\kappa_n\mu_n,
 \qquad
 g_{n,t}:=(\log t)\1-(t-1)d_n,
 \qquad
 \rho_n(t):=\rho_{n,g_{n,t}}.
\]
For every fixed \(R,L<\infty\),
\begin{equation}\label{eq:actual-sequence-PW-radius}
 \sup_{\substack{k\in\mathcal K_n,\ 0\le r\le R\\
                  t>0,\ |t-1|\le Lk^{-1/2}}}
 \left|
  \frac{\CW_t^{\vartheta_n}(k,r)}
  {w_r\beta_n^{(r-1)/2}e^{-k}\rho_n(t)^{-k}
   k^{k+3r/2-2}}
  -1
 \right|
 \longrightarrow0.
\end{equation}
Moreover,
\begin{equation}\label{eq:actual-sequence-PW}
 \sup_{\substack{k\in\mathcal K_n,\ 0\le r\le R\\
                  t>0,\ |t-1|\le Lk^{-1/2}}}
 \left|
  \frac{\CW_t^{\vartheta_n}(k,r)}
  {w_r\beta_n^{(r-1)/2}e^{-k}
   k^{k+3r/2-2}
   \exp\!\left\{
    -\dfrac{\alpha_n^2}{2\beta_n}k(t-1)^2
   \right\}}
  -1
 \right|
 \longrightarrow0.
\end{equation}
The same conclusion holds with
\(\alpha_n^2/(2\beta_n)\) in the exponential replaced by
\(\alpha^2/(2\beta)\).
\end{corollary}

\begin{proof}
Put
\(a_n(g):=\log\rho_{n,g}\).  The proof of
\Cref{lem:finite-parameter-stability} gives
\(\varepsilon>0\), independent of all sufficiently large \(n\), such
that \(D_g^3a_n\) is uniformly bounded on
\(B_{\mathscr X_{\mathbb C}}(0,\varepsilon)\).

Put \(b_n:=\1-d_n\).  Differentiating the equation for the critical pair and
the criticality constraint at the \(n\)-th model gives
\[
 D_ga_n(0)[\1]=-1,
 \qquad
 D_ga_n(0)[b_n]=0,
 \qquad
 D_g^2a_n(0)[b_n,b_n]
 =
 \frac{\alpha_n^2}{\beta_n}-1.
\]
Since
\[
 g_{n,t}
 =
 (t-1)b_n-\frac12(t-1)^2\1+O(|t-1|^3),
\]
Taylor's formula and the common bound for the third derivative yield, for some
\(\varepsilon,C>0\) independent of all sufficiently large \(n\),
\begin{equation}\label{eq:actual-sequence-radius-expansion}
 \left|
  \log\rho_n(t)
  -\frac{\alpha_n^2}{2\beta_n}(t-1)^2
 \right|
 \le C|t-1|^3,
 \qquad |t-1|<\varepsilon.
\end{equation}

We next rewrite the scalar model with parameter \(t\) as a field
perturbation.  Every graph counted by
\(\CW_t^{\vartheta_n}(k,r)\) has \(k+r-1\) edges, while
\[
 e^{g_{n,t}(x)-d_n(x)}=t e^{-td_n(x)}.
\]
Comparison of the weights configurationwise therefore gives
\[
 \CW_t^{\vartheta_n}(k,r)
 =
 t^{r-1}\CW_{g_{n,t}}^{\vartheta_n}(k,r).
\]

Uniformly for
\(k\in\mathcal K_n\), \(0\le r\le R\), and
\(|t-1|\le Lk^{-1/2}\),
\[
 \|g_{n,t}\|_\infty=O_L(k^{-1/2}),
 \qquad
 t^{r-1}=1+o(1).
\]
Hence \eqref{eq:actual-sequence-shrinking-variation} gives
\[
 \CW_t^{\vartheta_n}(k,r)
 =
 t^{r-1}\rho_n(t)^{-k}
 \CW^{\vartheta_n}(k,r)\{1+o(1)\}.
\]
Together with \eqref{eq:actual-sequence-base-PW} and
\(t^{r-1}=1+o(1)\), this proves
\eqref{eq:actual-sequence-PW-radius}.  Moreover,
\eqref{eq:actual-sequence-radius-expansion} implies
\[
 k\left|
  \log\rho_n(t)
  -\frac{\alpha_n^2}{2\beta_n}(t-1)^2
 \right|
 \le CL^3k^{-1/2}=o(1),
\]
and consequently
\[
 \rho_n(t)^{-k}
 =
 \exp\left\{
  -\frac{\alpha_n^2}{2\beta_n}k(t-1)^2
 \right\}\{1+o(1)\}
\]
uniformly over the same range.  Substitution of
\eqref{eq:actual-sequence-PW-radius} now proves
\eqref{eq:actual-sequence-PW}.

Finally,
\[
 k(t-1)^2
 \left|
  \frac{\alpha_n^2}{2\beta_n}
  -\frac{\alpha^2}{2\beta}
 \right|
 \le
 L^2
 \left|
  \frac{\alpha_n^2}{2\beta_n}
  -\frac{\alpha^2}{2\beta}
 \right|
 \longrightarrow0,
\]
which proves the final assertion.
\end{proof}

\section{Perron profile concentration and quadratic control}
\label{sec:profiles}
This section turns the uniform field estimates from
\Cref{sec:enumeration} into probabilistic bounds for the empirical type
measures of connected components.  We first derive logarithmic Laplace
bounds and summable quadratic estimates, and then collect the
consequences used in the component calculations.

Recall that \(\cQ_{k,r,t}\) is the tilted law defined in
\eqref{eq:Q-definition}.  Fix \(q<\infty\), and let
\[
 (H_i,\boldsymbol X^{(i)})\sim\cQ_{k_i,r_i,t_i},
 \qquad 1\le i\le q,
\]
be independent.  Writing \(L_{k_i}\) for the empirical type measure of
the \(i\)-th component, define
\begin{equation}\label{eq:total-size-notation}
 k_\Sigma:=\sum_{i=1}^qk_i,
 \qquad
 L_\Sigma:=\frac1{k_\Sigma}\sum_{i=1}^qk_iL_{k_i},
 \qquad
 \delta_\Sigma:=L_\Sigma-p.
\end{equation}
For one component, \(q=1\), so \(k_\Sigma=k_1\) and
\(L_\Sigma=L_{k_1}\).

When these assertions are used along the sequence in
\Cref{lem:finite-parameter-stability}, all \(q\) factors refer to the
same current model \(\vartheta_n\), and
\[
 (k_1,\ldots,k_q)\in\mathcal K_n^q,
 \qquad
 \mathcal K_n^q
 :=\underbrace{\mathcal K_n\times\cdots\times\mathcal K_n}_{q\text{ times}}.
\]
Every quantity depending on the model is then evaluated at
\(\vartheta_n\).  In particular, \(p\) and \(d\) mean \(p_n\) and
\[
 d_n(x):=\int_S\kappa_n(x,y)\,\mu_n(\dd y).
\]

\subsection{Bounds at the critical scale for type profiles}

We first derive uniform logarithmic Laplace bounds for centered linear profile
statistics.

Recall the scalar change of field \(g_t\) from
\eqref{eq:scalar-change-of-field}.  Along a finite parameter
sequence, \(g_{n,t}\) denotes the same field with \(d\) replaced by
the current degree function \(d_n\).
For real \(h\in\mathscr X\), let \(\CW_{t,h}(k,r)\) be obtained from
\(\CW_t(k,r)\) by inserting the vertex factor
\[
 \exp\left\{\sum_{i=1}^k h(x_i)\right\}.
\]
Since every graph in the sum has \(k+r-1\) edges, direct comparison of
the configuration weights gives the exact change of field identity
\begin{equation}\label{eq:external-field-gauge}
 \CW_{t,h}(k,r)=t^{r-1}\CW_{g_t+h}(k,r).
\end{equation}

\begin{corollary}[Uniform logarithmic Laplace bound for profiles]
\label{cor:uniform-field-transfer}
Fix \(R,L<\infty\) and \(q\ge1\).  Assume
\[
 0\le r_i\le R,
 \qquad
 |t_i-1|\le Lk_i^{-1/2},
 \qquad 1\le i\le q.
\]
There are \(\eta_R>0\) and \(C_{R,L,q}<\infty\) such that, uniformly
for real \(f\in\mathscr X\) with \(\|f\|_\infty\le1\) and
\(|v|\le\eta_R\sqrt{k_\Sigma}\),
\begin{equation}\label{eq:profile-laplace-bound}
 \log\E
 \exp\left\{
  v\sqrt{k_\Sigma}\{L_\Sigma(f)-p(f)\}
 \right\}
 \le C_{R,L,q}\{1+v^2+L|v|\}.
\end{equation}
The same conclusion holds uniformly along the finite parameter
sequences of \Cref{lem:finite-parameter-stability}, with every quantity
interpreted at the current model.
\end{corollary}

\begin{proof}
First take \(q=1\) and write \(k=k_1\).  Put
\(h=vf/\sqrt{k}\), recall that \(a(g)=\log\rho_g\), and define
\[
 R_{k,r}(g):=\frac{\rho_g^k\CW_g(k,r)}{\CW(k,r)}.
\]
By \eqref{eq:Q-definition} and the linearity of \(p\),
\[
 \E_{\cQ_{k,r,t}}
 e^{v\sqrt{k}\{L_k(f)-p(f)\}}
 =
 e^{-kp(h)}
 \frac{\CW_{t,h}(k,r)}{\CW_t(k,r)}.
\]
Applying \eqref{eq:external-field-gauge} with fields \(h\) and \(0\)
then gives
\[
 \log\E_{\cQ_{k,r,t}}
 e^{v\sqrt{k}\{L_k(f)-p(f)\}}
 =
 -k[p(h)+a(g_t+h)-a(g_t)]
 +\log\frac{R_{k,r}(g_t+h)}{R_{k,r}(g_t)}.
\]

The radius estimate \eqref{eq:source-radius-Taylor} bounds the first term by
\[
 Ck\{\|h\|_\infty^2+\|g_t\|_\infty\|h\|_\infty\}
 \le C_L(v^2+L|v|).
\]

Choose \(\eta_R>0\) small enough that, for all sufficiently large \(k\),
\(g_t\) and \(g_t+h\) remain in the fixed field ball of
\eqref{eq:local-field-variation}.  The uniform bounds there give
\[
 c_R\le R_{k,r}(g_t),R_{k,r}(g_t+h)\le C_R,
\]
and therefore
\[
 \left|
 \log\frac{R_{k,r}(g_t+h)}{R_{k,r}(g_t)}
 \right|=O_R(1).
\]

Combining the preceding two estimates proves
\eqref{eq:profile-laplace-bound} for \(q=1\) and all sufficiently large
\(k\).  For the finitely many remaining admissible sizes,
\[
 \log\E_{\cQ_{k,r,t}}
 e^{v\sqrt{k}\{L_k(f)-p(f)\}}
 \le 2|v|\sqrt{k}\le v^2+k,
\]
so the same estimate follows after enlarging the constant.  Applying
this bound factorwise also covers collections containing both
bounded and diverging sizes.

Independence then gives
\[
 \log\E e^{v\sqrt{k_\Sigma}\{L_\Sigma(f)-p(f)\}}
 =
 \sum_{i=1}^q
 \log\E e^{v_i\sqrt{k_i}\{L_{k_i}(f)-p(f)\}},
 \qquad
 v_i:=v\sqrt{\frac{k_i}{k_\Sigma}}.
\]
Since
\[
 \sum_i v_i^2=v^2,\qquad
 \sum_i|v_i|\le\sqrt q\,|v|,
 \qquad
 |v_i|\le\eta_R\sqrt{k_i},
\]
summing the single component bounds proves the assertion.

Along the finite parameter sequence, repeat the preceding argument with the current
model quantities, replacing
\[
 (p,a,g_t)\quad\text{by}\quad(p_n,a_n,g_{n,t}).
\]
The radius contribution is controlled uniformly by
\eqref{eq:actual-source-radius-Taylor}, while the normalized coefficient
ratios are controlled by part~\ref{item:actual-variation} of
\Cref{lem:finite-parameter-stability}.  This proves the same bounds
uniformly along the sequence, including the form for a fixed number of components.
\end{proof}

We next pass from the logarithmic Laplace bounds to summable quadratic
functionals.  For nonzero real \(f\), put
\begin{equation}\label{eq:normalized-profile-fluctuation}
 Y_f:=\frac{\sqrt{k_\Sigma}\{L_\Sigma(f)-p(f)\}}
            {\|f\|_\infty}.
\end{equation}

\begin{lemma}[Summable quadratic profile bound on the component scale]
\label{lem:summable-quadratic-exponential}
For each admissible single component or fixed collection of components,
let \(b_j\ge0\) and \(f_j\in C(S)\).  These quantities may depend on the
component sizes and on the current finite type model.  Assume that
\[
 B:=\sup\sum_{j\ge1}b_j\|f_j\|_\infty^2<\infty,
\]
where the supremum is over the ranges under consideration.  Then, as \(k_\Sigma\to\infty\), for every fixed
\(0\le A<\infty\),
\begin{equation}\label{eq:summable-quadratic-exponential}
 \E\exp\left\{
  \frac A{\sqrt{k_\Sigma}}
  \sum_{j\ge1}b_j
  \bigl[\sqrt{k_\Sigma}\{
   L_\Sigma(f_j)-p(f_j)\}\bigr]^2
 \right\}
 =1+o(1).
\end{equation}
The convergence is uniform for a single component, for a fixed number
of components, and along the finite parameter sequences of
\Cref{lem:finite-parameter-stability}; along such a sequence \(p\) is
read as \(p_n\).
\end{lemma}

\begin{proof}
For every fixed \(U>0\), \eqref{eq:profile-laplace-bound} and
exponential Markov inequality give
\begin{equation}\label{eq:Y-linear-tail}
 \Prob\{|Y_f|>x\}\le2e^{-Ux+C_U},
 \qquad x\ge0,
\end{equation}
uniformly in nonzero \(f\) and in all stated ranges.  Since
\(L_\Sigma\) and \(p\) are probability measures,
\[
 |Y_f|\le2\sqrt{k_\Sigma}.
\]

Fix \(A\ge0\), and choose \(U>2AB\).  For \(0\le a\le AB\), tail
integration and
\(ax^2/\sqrt{k_\Sigma}\le2ax\) on
\(0\le x\le2\sqrt{k_\Sigma}\) give
\begin{align}
 \sup_{f\ne0}
 \left(
  \E\exp\left\{\frac a{\sqrt{k_\Sigma}}Y_f^2\right\}-1
 \right)
 &\le
 \frac{4ae^{C_U}}{\sqrt{k_\Sigma}}
 \int_0^\infty x e^{-(U-2a)x}\,\dd x \notag\\
 &=O_{A,B}(k_\Sigma^{-1/2}).
 \label{eq:one-dimensional-quadratic-control}
\end{align}
The bound is uniform in \(a\).

Put
\[
 c_j:=b_j\|f_j\|_\infty^2,
 \qquad
 B_0:=\sum_{j\ge1}c_j\le B,
\]
and omit indices for which \(f_j=0\).  If \(B=0\), there is nothing to
prove.  Otherwise convexity, with the unused weight
\(1-B_0/B\) assigned to zero, gives
\[
 \exp\left\{
  \frac A{\sqrt{k_\Sigma}}\sum_jc_jY_{f_j}^2
 \right\}
 \le
 1-\frac{B_0}{B}
 +\sum_j\frac{c_j}{B}
  \exp\left\{
   \frac{AB}{\sqrt{k_\Sigma}}Y_{f_j}^2
  \right\}.
\]
Taking expectations and applying
\eqref{eq:one-dimensional-quadratic-control} with \(a=AB\) proves
\eqref{eq:summable-quadratic-exponential}.  Uniformity along a finite
parameter sequence follows from the corresponding uniform version of
\eqref{eq:Y-linear-tail}.
\end{proof}

\subsection{Profile consequences}

The following is the sole profile interface used below, in both its
single component and fixed collection forms.

\begin{lemma}[Perron profile estimates]
\label{lem:profile}
Fix \(R,L_0<\infty\).  Use the notation above,
with \(q=1\) allowed, and assume
\[
 0\le r_i\le R,
 \qquad
 |t_i-1|\le L_0k_i^{-1/2},
 \qquad 1\le i\le q.
\]
Recall that
\[
 L_{k_i}:=\frac1{k_i}\sum_{j=1}^{k_i}\delta_{X_j^{(i)}},
\]
and let \(k_\Sigma\) and \(L_\Sigma\) be as in
\eqref{eq:total-size-notation}.  All constants below are uniform over
these ranges.

\begin{enumerate}[label=\textup{(\roman*)}]

\item 
For every fixed finite family \(f_1,\ldots,f_m\in C(S)\), there are
\(c_m,C_m>0\) such that
\begin{equation}\label{eq:linear-statistics}
 \Prob\left\{
  \max_{j\le m}|L_\Sigma(f_j)-p(f_j)|>\frac{x}{\sqrt{k_\Sigma}}
 \right\}
 \le C_m\exp\{-c_mx^2+C_mx\},
 \qquad 1\le x\le c_m\sqrt{k_\Sigma}.
\end{equation}

In particular, let \(d_\star=d_{\TV}\) in finite type and
\(d_\star=d_{\BL}\) in compact type.  For every \(\varepsilon>0\),
there are \(c_\varepsilon,C_\varepsilon>0\) such that
\begin{equation}
\label{eq:profile-main-tail}
 \Prob\{d_\star(L_\Sigma,p)>\varepsilon\}
 \le C_\varepsilon e^{-c_\varepsilon k_\Sigma}.
\end{equation}

The estimates in this part remain valid along the finite type
sequences of \Cref{lem:finite-parameter-stability}, with \(p=p_n\)
and all Perron quantities interpreted at the current model.
\item Let \(H\in\mathscr P_{\mathrm s}(S)\), the symmetric projective
kernel class defined in \Cref{def:projectiveclass}, with projective
norm \(\|\cdot\|_{\pi,\mathrm s}\).  Put
\begin{equation}\label{eq:QP-profile-statistic}
 \mathcal Y_{k_\Sigma}(H)
 :=
 k_\Sigma\iint_{S^2}H(x,y)\,
 (L_\Sigma-p)(\dd x)(L_\Sigma-p)(\dd y).
\end{equation}
Uniformly over the range for a fixed number of components, for every fixed
\(0\le A<\infty\),
\begin{equation}\label{eq:projective-quadratic-control}
 \E\exp\left\{
  \frac A{\sqrt{k_\Sigma}}|\mathcal Y_{k_\Sigma}(H)|
 \right\}
 =1+o(1).
\end{equation}

Along a finite parameter sequence, the same estimate holds
with \(p=p_n\), \(H=H_n\), and the admissible sizes specified by the
sequence convention, provided
\[
 \sup_n\|H_n\|_{\pi,\mathrm s}<\infty.
\]

\item In finite type, define the aggregate type counts and their
centered fluctuations by
\begin{equation}\label{eq:finiteproductprofilenotation}
     K_a:=
 \sum_{i=1}^q\sum_{j=1}^{k_i}
 \1_{\{X_j^{(i)}=a\}},
 \qquad
 \Delta_a:=K_a-k_\Sigma p_a.
\end{equation}

If \(k_\Sigma=O(n^{2/3})\), then, for every fixed \(A>0\),
\begin{equation}\label{eq:Pearson-UI}
 \E\exp\left\{
  \frac A n\sum_{a\in S}\frac{\Delta_a^2}{p_a}
 \right\}
 =1+o(1).
\end{equation}
The estimate remains valid along the finite type sequences of
\Cref{lem:finite-parameter-stability}, with \(p=p_n\).

\item In compact type under \Cref{ass:QP}, define
\(\kappa_2(x,y):=\int_S\kappa(x,z)\kappa(z,y)\,\mu(\dd z)\) and
\(\overline H(x,y):=\kappa_2(x,y)-\kappa(x,y)-d(x)d(y)+d(x)+d(y)-1\).
Form \(\mathcal Y_{k_\Sigma}(\overline H)\) as in
\eqref{eq:QP-profile-statistic}.  For every fixed
\(0\le A,C_0<\infty\), uniformly for \(k_\Sigma\le C_0n^{2/3}\),
\begin{equation}\label{eq:QP-BJR-scale}
 \E\exp\left\{
  A\frac{k_\Sigma}{n}\bigl|\mathcal Y_{k_\Sigma}(\overline H)\bigr|
 \right\}
 =1+o(1).
\end{equation}

\end{enumerate}

\end{lemma}

\begin{proof}
For part~\textup{(i)}, put
\[
 M:=\max\left\{1,\max_{j\le m}\|f_j\|_\infty\right\}.
\]
Apply \eqref{eq:profile-laplace-bound} with test functions
\(f_j/M\) and \(-f_j/M\).
For the upper tail, Chernoff's inequality with
\(v=\lambda x/M\), where \(\lambda>0\) is fixed sufficiently small,
gives
\[
 \Prob\left\{
  \sqrt{k_\Sigma}\{L_\Sigma(f_j)-p(f_j)\}>x
 \right\}
 \le
 \exp\left\{
  -\frac{\lambda x^2}{M^2}
  +C\left(1+\frac{\lambda^2x^2}{M^2}
            +\frac{L_0\lambda x}{M}\right)
 \right\}.
\]
After decreasing \(c_m\), the condition
\(x\le c_m\sqrt{k_\Sigma}\) keeps \(v\) in the expanding range.  Choosing
\(\lambda\) sufficiently small and applying the same argument to
\(-f_j\) therefore gives
\[
 \Prob\left\{
  |L_\Sigma(f_j)-p(f_j)|>\frac{x}{\sqrt{k_\Sigma}}
 \right\}
 \le C e^{-cx^2+Cx}.
\]
A union bound over \(1\le j\le m\) proves
\eqref{eq:linear-statistics}.

In finite type, total variation is controlled by the fixed finite
family of coordinate functions.  In compact type, the
bounded-Lipschitz unit ball is totally bounded in the uniform norm by
Arzel\`a--Ascoli.  Given \(\varepsilon>0\), choose a finite uniform
\(\varepsilon/4\)-net.  If
\[
 d_{\BL}(L_\Sigma,p)>\varepsilon,
\]
one member \(f\) of the net satisfies
\[
 |L_\Sigma(f)-p(f)|>\varepsilon/2.
\]
Taking
\[
 b_\varepsilon:=\min\{\varepsilon/2,c_m/2\}
\]
and applying \eqref{eq:linear-statistics} to the chosen finite family
at \(x=b_\varepsilon\sqrt{k_\Sigma}\) proves
\eqref{eq:profile-main-tail}.  The finite type argument is identical
with the coordinate family in place of the bounded-Lipschitz net.
The same proof is uniform along the finite type sequences by
the last clause of \Cref{cor:uniform-field-transfer}.

For part~\textup{(ii)}, choose a projective representation
\[
 H=\sum_{j\ge1}a_jf_j\otimes f_j,
 \qquad
 \sum_{j\ge1}|a_j|\|f_j\|_\infty^2<\infty.
\]
Then
\[
 \mathcal Y_{k_\Sigma}(H)
 =
 \sum_{j\ge1}a_j
 \bigl[
  \sqrt{k_\Sigma}\{L_\Sigma(f_j)-p(f_j)\}
 \bigr]^2,
\]
and hence
\[
 |\mathcal Y_{k_\Sigma}(H)|
 \le
 \sum_{j\ge1}|a_j|
 \bigl[
  \sqrt{k_\Sigma}\{L_\Sigma(f_j)-p(f_j)\}
 \bigr]^2.
\]
Thus \eqref{eq:projective-quadratic-control} follows directly from
\Cref{lem:summable-quadratic-exponential}.

Along a finite parameter sequence, choose for each \(H_n\) a
projective representation with budget at most
\(2\|H_n\|_{\pi,\mathrm s}\).  The resulting budgets are uniformly
bounded, so the sequence clause of
\Cref{lem:summable-quadratic-exponential} gives the stated extension.

For part~\textup{(iii)}, put
\[
 V:=\sum_{a\in S}\frac{\Delta_a^2}{p_a}.
\]
Since
\[
 L_\Sigma(\1_{\{a\}})=\frac{K_a}{k_\Sigma},
\]
we have
\[
 V
 =
 k_\Sigma\sum_{a\in S}\frac1{p_a}
 \bigl[
  \sqrt{k_\Sigma}\{
   L_\Sigma(\1_{\{a\}})-p_a
  \}
 \bigr]^2.
\]
Choose \(C_0<\infty\) so that
\(k_\Sigma\le C_0n^{2/3}\) throughout the range under consideration.  Then
\[
 \frac{k_\Sigma}{n}\le\frac{C_0^{3/2}}{\sqrt{k_\Sigma}},
\]
and therefore
\[
 \frac A nV
 \le
 \frac{AC_0^{3/2}}{\sqrt{k_\Sigma}}
 \sum_{a\in S}\frac1{p_a}
 \bigl[
  \sqrt{k_\Sigma}\{
   L_\Sigma(\1_{\{a\}})-p_a
  \}
 \bigr]^2.
\]
Applying \Cref{lem:summable-quadratic-exponential} to the finite
coordinate family with weights \(1/p_a\) proves
\eqref{eq:Pearson-UI}.  Along a finite type sequence, the
positive coordinates of \(p_n\) remain uniformly bounded away from
zero, so the same argument is uniform.  If \(k_\Sigma\) remains bounded, the
conclusion follows directly from \(V=O(1)\) and \(n^{-1}\to0\).

For part~\textup{(iv)}, the Perron identities give
\begin{equation}\label{eq:profile-canonical-centering}
 \int_S\overline H(x,y)\,p(\dd y)=0,
 \qquad x\in S.
\end{equation}
Indeed, the first two terms in the definition above both integrate
to \(u(x)\);
moreover \(p(d)=1\), and the remaining terms cancel.  Since
\(\overline H\) is symmetric, the analogous identity holds in the
first variable.  Consequently,
\[
 \mathcal Y_{k_\Sigma}(\overline H)
 =
 k_\Sigma\iint_{S^2}\overline H(x,y)\,
       L_\Sigma(\dd x)L_\Sigma(\dd y).
\]
If \(k_\Sigma\le C_0n^{2/3}\), then
\[
 \frac{k_\Sigma}{n}\le\frac{C_0^{3/2}}{\sqrt{k_\Sigma}}.
\]
Part~\textup{(ii)}, applied with \(AC_0^{3/2}\), therefore proves
\eqref{eq:QP-BJR-scale}.  If \(k_\Sigma\) remains bounded, the conclusion
follows directly from boundedness of
\(\mathcal Y_{k_\Sigma}(\overline H)\) and \(k_\Sigma/n\to0\).
\end{proof}

\section{Component counts: exact formulas and critical asymptotics}
\label{sec:component-formulae}

In this section, we first establish an exact formula for ordered component counts and
then derive its critical asymptotic in the deterministic finite type and
annealed compact type BJR models.

Fix \(q\ge1\) and surpluses \(r_1,\ldots,r_q\), which remain fixed
throughout.  The component sizes \(k_1,\ldots,k_q\) may depend on
\(n\).  Retain the total size \(k_\Sigma=\sum_i k_i\) from
\eqref{eq:total-size-notation}, and set
\begin{equation}\label{eq:KRM}
 M=\sum_{i=1}^q(k_i+r_i-1)
   =k_\Sigma+\sum_{i=1}^qr_i-q.
\end{equation}
Thus \(M\) is the total number of internal edges in the selected
ordered components.  Write \(\bk=(k_1,\ldots,k_q)\) and
\(\boldsymbol r=(r_1,\ldots,r_q)\), and define the ordered
distinct component count by
\begin{equation}\label{eq:orderedcomponentcount}
 N_n(\bk,\boldsymbol r)
 =\sum_{C_1,\ldots,C_q\in\cC(G_n)}^{\ne}
   \prod_{i=1}^q
   \1_{\{|C_i|=k_i,\ \sur(C_i)=r_i\}}.
\end{equation}
We use the falling factorial notation from
\eqref{eq:fallingfactorialdefinition}.

The two models are stated simultaneously below.  Their component
enumerators are denoted by \(\widetilde{\CW}_n\), and the factors left
after extracting those enumerators by \(\widetilde{\cR}_n\).

\begin{proposition}[Exact formula and critical asymptotic for component counts]
\label{prop:component-count-formula}
Consider one of the following two settings.
\begin{enumerate}[label=\textup{(\alph*)}]
\item\label{item:component-count-deterministic}
In the deterministic finite type case, let
\(\nu^{(n)}=(n_a/n)_a\) be the empirical type law and let
\(\kappa^{(n)}\) be an exactly critical kernel for this law:
\[
 \spr\{\kappa^{(n)}\diag(\nu^{(n)})\}=1.
\]
The parameter \(t_n\) below is the scalar edge parameter paired with
\(\kappa^{(n)}\).  For the asymptotic assertion, assume that the
sequence of exactly critical pairs
\((\nu^{(n)},\kappa^{(n)})\) and the relevant component sizes satisfy
the hypotheses of \Cref{lem:finite-parameter-stability} for some limiting
critical pair, with the corresponding size set containing
\(k_1,\ldots,k_q\).  Set
\[
 \widetilde{\CW}_n=\CW_{t_n}^{(n)},\qquad
 \widetilde{\cR}_n=\cR_n^{\rm det},\qquad
 \widetilde\alpha_n=\alpha_n,\qquad
 \widetilde\beta_n=\beta_n,
\]
where \(\cR_n^{\rm det}\) is defined in \eqref{eq:Rdet-exact}.

\item\label{item:component-count-annealed}
In the annealed compact type BJR model satisfying
\textup{(B1)--(B2)} of \Cref{def:BJR-model}, set
\[
 \widetilde{\CW}_n=\CW_{t_n},\qquad
 \widetilde{\cR}_n=\cR_n^{\rm ann},\qquad
 \widetilde\alpha_n=\alpha,\qquad
 \widetilde\beta_n=\beta,
\]
where \(\cR_n^{\rm ann}\) is defined in \eqref{eq:Rann-exact}.
\end{enumerate}
Then the ordered component count in
\eqref{eq:orderedcomponentcount} satisfies the exact identity
\begin{equation}\label{eq:abstract-size-formula}
 \E N_n(\bk,\boldsymbol r)=
 \frac{n^{k_\Sigma}}{\prod_{i=1}^qk_i!}\,
 n^{-M}\prod_{i=1}^q\widetilde{\CW}_n(k_i,r_i)
 \widetilde{\cR}_n.
\end{equation}
Moreover, for every fixed \(L<\infty\), uniformly over
\(|t_n-1|\le Ln^{-1/3}\) and over all \(k_i\) for which
\(k_i/n^{2/3}\) lie in a fixed compact subset of \((0,\infty)\),
\begin{equation}\label{eq:component-count-residual-asymptotic}
 \widetilde{\cR}_n=
 \exp\left\{
 \frac{\widetilde\alpha_n(t_n-1)k_\Sigma^2}{2n}
 -\frac{\widetilde\beta_n k_\Sigma^3}{6n^2}\right\}(1+o(1)).
\end{equation}
For this estimate in
case~\textup{\ref{item:component-count-annealed}}, assume additionally
\Cref{ass:QP}.  In
case~\textup{\ref{item:component-count-deterministic}}, the estimate
is uniform along the finite type parameter sequence; in
case~\textup{\ref{item:component-count-annealed}}, the model is fixed
and the estimate is uniform over the displayed \(t_n\) and size ranges.
\end{proposition}

Part~\textup{\ref{item:component-count-deterministic}} applies to the
model of \Cref{def:finite-model} after setting
\begin{equation}\label{eq:finite-critical-normalization}
 \sigma_n=\spr\{\kappa\diag(\nu^{(n)})\},\qquad
 \kappa^{(n)}=\kappa/\sigma_n,\qquad
 \bar t_n=t_n\sigma_n.
\end{equation}
Since \(\bar t_n\kappa^{(n)}=t_n\kappa\), the edge probabilities are
unchanged, and \Cref{lem:finite-parameter-stability} applies on the
stated size range.  The same normalization is used for the finite
truncations in the countable type construction of \Cref{sec:countable}.
The factors \(\widetilde{\CW}_n(k_i,r_i)\) are left unexpanded here;
\Cref{sec:factorial-process} combines \Cref{thm:Perron-Wright} with
Stirling's formula to obtain the factorial density and process limit
in \Cref{thm:local-main}.

For selected typed connected graphs \(H_1,\ldots,H_q\) on disjoint
labeled blocks, write
\begin{equation}\label{eq:selected-edge-notation}
 E_\cup=\bigcup_{i=1}^qE(H_i),\qquad
 \kappa_e=\kappa(\operatorname{type}(v),\operatorname{type}(w))
 \quad\text{for \(e=\{v,w\}\), \(v\ne w\)}.
\end{equation}

\subsection{Deterministic finite type counts}
\label{subsec:deterministic-component-counts}

We first refine the deterministic component count by prescribing the
type profile of each selected component.  Throughout this subsection,
we suppress the superscript \(n\) on \(\nu^{(n)}\) and
\(\kappa^{(n)}\) from
\Cref{prop:component-count-formula}\textup{(a)}.  For component \(i\), let
\(\boldsymbol\ell^{(i)}=(\ell_{ia})_{a\in S}\) be its type count
vector, so that \(\sum_a\ell_{ia}=k_i\).  For
\(\boldsymbol\ell=(\boldsymbol\ell^{(1)},\ldots,
\boldsymbol\ell^{(q)})\), define
\begin{align}
 N_n^{\rm prof}(\boldsymbol\ell,\boldsymbol r)
 =\sum_{C_1,\ldots,C_q\in\cC(G_n)}^{\ne}
 \prod_{i=1}^q\bigg[
 &\1_{\{\sur(C_i)=r_i\}}
 \prod_{a\in S}
 \1_{\{\#\{v\in C_i:\operatorname{type}(v)=a\}=\ell_{ia}\}}
 \bigg].
 \label{eq:profile-component-count}
\end{align}
For these fixed profiles, put \(K_a=\sum_{i=1}^q\ell_{ia}\) and
\(\bK=(K_a)_{a\in S}\), consistently with the random aggregate profile
notation in \eqref{eq:finiteproductprofilenotation}.  Summing
\(N_n^{\rm prof}(\boldsymbol\ell,\boldsymbol r)\) over all
profiles satisfying \(\sum_a\ell_{ia}=k_i\), \(1\le i\le q\),
recovers \(N_n(\bk,\boldsymbol r)\) in
\eqref{eq:orderedcomponentcount}.

\begin{proposition}[Exact deterministic profile formula]
\label{prop:exact-deterministic-profile}
The profile refined count in \eqref{eq:profile-component-count} satisfies
\begin{equation}\label{eq:exact-deterministic-profile}
 \E N_n^{\rm prof}(\boldsymbol\ell,\boldsymbol r)
 =\left\{\prod_{a\in S}
 \frac{(n_a)_{K_a}}{\prod_{i=1}^q\ell_{ia}!}\right\}
 \exp\left\{-\frac{t_n}{n}\mathcal B_n(\bK)\right\}
 \prod_{i=1}^q\CW_{n,\kappa}(\boldsymbol\ell^{(i)},r_i).
\end{equation}
Here \(d_a^\nu=\sum_{b\in S}\kappa_{ab}\nu_b\) is the kernel degree
of type \(a\) under the empirical type law \(\nu\).
For a fixed labeled vertex set with prescribed type count vector
\(\boldsymbol m=(m_a)_{a\in S}\), put
\(|\boldsymbol m|=\sum_am_a\) and define
\begin{equation}\label{eq:odds-connected-sum}
 \CW_{n,\kappa}(\boldsymbol m,r)=
 \sum_{\substack{H\ {\rm connected}\\
                 e(H)=|\boldsymbol m|+r-1}}
 \prod_{\{v,w\}\in E(H)}
 \left(e^{t_n\kappa(\operatorname{type}(v),
                         \operatorname{type}(w))/n}-1\right),
\end{equation}
where \(\operatorname{type}(v)\) is the prescribed type of \(v\).
The edge factors in \eqref{eq:odds-connected-sum} are the exact odds
ratios between presence and absence.  Finally, define
\begin{equation}\label{eq:B-isolation}
 \mathcal B_n(\bK)=
 n\sum_aK_ad_a^\nu-\frac12\bK^{\mathsf T}\kappa\bK
 -\frac12\sum_a\kappa_{aa}K_a.
\end{equation}
The quantity \(\mathcal B_n(\bK)\) is the total \(\kappa\)-weight of
the unordered pairs whose absence isolates the selected union.
\end{proposition}

\begin{proof}
Fix ordered disjoint sets \(V_1,\ldots,V_q\) with type count vectors
\(\boldsymbol\ell^{(1)},\ldots,\boldsymbol\ell^{(q)}\), put
\(U=\bigcup_iV_i\), and let \(H_i\) be connected on \(V_i\) with
\(e(H_i)=k_i+r_i-1\).  Let \(\operatorname{Pairs}(U)\) be the set of unordered
pairs of distinct vertices with at least one endpoint in \(U\), and
write \(x_e=t_n\kappa_e/n\).  Edge independence gives
\begin{equation}\label{eq:raw-deterministic-component-probability}
\begin{aligned}
 &\Prob\{V_i\text{ is a component with induced graph }H_i
          \text{ for }1\le i\le q\}\\
 &\quad=\prod_{e\in E_\cup}(1-e^{-x_e})
  \prod_{e\in\operatorname{Pairs}(U)\setminus E_\cup}e^{-x_e}\\
&\quad=\exp\left\{-\sum_{e\in\operatorname{Pairs}(U)}x_e\right\}
  \prod_{e\in E_\cup}(e^{x_e}-1).
\end{aligned}
\end{equation}
Here \(e^{x_e}-1=(1-e^{-x_e})/e^{-x_e}\) is the exact odds ratio.  The
isolation exponent depends only on the aggregate type counts, since
\begin{equation}\label{eq:pair-weight-calculation}
 \sum_{e\in\operatorname{Pairs}(U)}\kappa_e
 =n\sum_aK_ad_a^\nu-\frac12\bK^{\mathsf T}\kappa\bK
  -\frac12\sum_a\kappa_{aa}K_a
 =\mathcal B_n(\bK).
\end{equation}
Indeed, the first term counts pairs between selected and unselected
vertices once and pairs of selected vertices twice, while the last two
terms correct the double and diagonal counts.  For each type \(a\), the ordered choice of
the candidate blocks contributes
\((n_a)_{K_a}/\prod_i\ell_{ia}!\).  Summing
\eqref{eq:raw-deterministic-component-probability} over these blocks
and over the graphs \(H_i\), the remaining graph sum factors as
\[
 \sum_{H_1,\ldots,H_q}\prod_{i=1}^q\prod_{e\in E(H_i)}
 (e^{t_n\kappa_e/n}-1)
 =\prod_{i=1}^q\CW_{n,\kappa}(\boldsymbol\ell^{(i)},r_i).
\]
This proves \eqref{eq:exact-deterministic-profile}.
\end{proof}

To pass from the profile identity in
\Cref{prop:exact-deterministic-profile} to the size identity, put
\(\vartheta_n:=(\nu,\kappa)\).  Let \(u^{(n)}\) and \(p_n\) be the
Perron eigenfunction and Perron law of \(\kappa\diag(\nu)\), and let
\(\alpha_n,\beta_n\) be the corresponding constants, normalized as in
\eqref{eq:Perron-quantities}.  We write
\(\CW_t^{(n)}:=\CW_t^{\vartheta_n}\) and use
\(\cQ_{m,r,t}^{(n)}\) for a component of size \(m\) under the corresponding law in
\eqref{eq:Q-definition}.  Sample the \(q\) typed connected graphs independently
with laws \(\cQ_{k_i,r_i,t_n}^{(n)}\) and place them on disjoint labeled
blocks.  We use the aggregate type counts from
\eqref{eq:finiteproductprofilenotation} and the edge notation in
\eqref{eq:selected-edge-notation}.  Define
\begin{equation}\label{eq:L-falling}
 \cL_n(\bK)=
 \prod_{a\in S}\frac{(n\nu_a)_{K_a}}{(n\nu_a)^{K_a}}.
\end{equation}
Then
\begin{equation}
 \cR_n^{\rm det}=\E_{\bigotimes_{i=1}^q\cQ_{k_i,r_i,t_n}^{(n)}}\!\left[\cL_n(\bK)
 \exp\left\{\frac{t_n}{2n}\left(
 \bK^{\mathsf T}\kappa\bK+\sum_a\kappa_{aa}K_a\right)\right\}
 \prod_{e\in E_\cup}\frac{e^{t_n\kappa_e/n}-1}{t_n\kappa_e/n}
 \right].
 \label{eq:Rdet-exact}
\end{equation}
The quotient is defined as \(1\) when \(\kappa_e=0\).

\begin{proof}[Proof of \Cref{prop:component-count-formula}\textup{(a)}: exact identity]
For a fixed profile, the vertex selection factor is
\begin{equation}\label{eq:deterministic-profile-factorization}
 \prod_a\frac{(n\nu_a)_{K_a}}{\prod_i\ell_{ia}!}
 =\frac{n^{k_\Sigma}}{\prod_i k_i!}
  \left\{\prod_i\frac{k_i!}{\prod_a\ell_{ia}!}
              \prod_a\nu_a^{K_a}\right\}\cL_n(\bK).
\end{equation}
\nopagebreak[4]
The expression in braces in
\eqref{eq:deterministic-profile-factorization} is the product, over
\(1\le i\le q\), of the multinomial probabilities that \(k_i\)
independent samples from \(\nu\) have type count vector
\(\boldsymbol\ell^{(i)}\).  Hence
\begin{align}
 \E N_n(\bk,\boldsymbol r)
 &=\frac{n^{k_\Sigma}}{\prod_i k_i!}
 \sum_{\substack{\boldsymbol\ell:\,\sum_a\ell_{ia}=k_i\\1\le i\le q}}
 \left\{\prod_i\frac{k_i!}{\prod_a\ell_{ia}!}\prod_a\nu_a^{K_a}\right\}
 \cL_n(\bK)e^{-t_n\mathcal B_n(\bK)/n}
 \prod_i\CW_{n,\kappa}(\boldsymbol\ell^{(i)},r_i).
 \label{eq:deterministic-profile-sum}
\end{align}
Expanding the last product in \eqref{eq:deterministic-profile-sum} via
\eqref{eq:odds-connected-sum} gives a sum over
\((H_1,\ldots,H_q)\).  For each tuple,
\begin{align}
 e^{-t_n\mathcal B_n(\bK)/n}
 &=e^{-t_n\sum_aK_ad_a^\nu}
 \exp\left\{\frac{t_n}{2n}\left(\bK^{\mathsf T}\kappa\bK+\sum_a\kappa_{aa}K_a\right)\right\},
 \label{eq:deterministic-isolation-factorization}\\
 \prod_{i=1}^q\prod_{e\in E(H_i)}(e^{t_n\kappa_e/n}-1)
 &=n^{-M}\prod_{i=1}^q\prod_{e\in E(H_i)}t_n\kappa_e
 \prod_{e\in E_\cup}\frac{e^{t_n\kappa_e/n}-1}{t_n\kappa_e/n}.
 \label{eq:deterministic-odds-factorization}
\end{align}
Substitution into \eqref{eq:deterministic-profile-sum}, followed by
\eqref{eq:Q-definition} and \eqref{eq:Rdet-exact}, gives
\eqref{eq:abstract-size-formula} in case~\textup{(a)}.
\end{proof}

\begin{proof}[Proof of the deterministic critical asymptotic]
We prove \eqref{eq:component-count-residual-asymptotic} uniformly over
the stated parameter sequence and ranges.  The sizes
\(k_1,\ldots,k_q\), and hence
\(k_\Sigma\), are deterministic throughout.  Only the independently
sampled typed connected graphs
\((H_i,X_{i1},\ldots,X_{ik_i})\), and hence \(\bK\), are random.  All
expectations below are under
\(\bigotimes_{i=1}^q\cQ_{k_i,r_i,t_n}^{(n)}\).  The constant
\(C<\infty\) is uniform over the stated ranges and may change from line
to line.

\smallskip
\noindent\emph{Step 1: Uniform elementary expansions.}
Since the kernel is bounded,
\(|E_\cup|=M=O(k_\Sigma)\) by \eqref{eq:KRM}, and
\(k_\Sigma/n=o(1)\), the factor
\(\exp\{t_n\sum_a\kappa_{aa}K_a/(2n)\}\) and the normalized edge odds product in
\eqref{eq:Rdet-exact} are both \(1+o(1)\), uniformly: their logarithms
are, respectively, \(O(k_\Sigma/n)\) and \(O(M/n)\).
For every integer type count vector \(\bK\) with
\(K_a\ge0\) and \(\sum_aK_a=k_\Sigma\), a Taylor expansion of
\eqref{eq:L-falling} gives
\begin{equation}\label{eq:falling-expansion}
 \log\cL_n(\bK)=
 -\frac1{2n}\sum_a\frac{K_a^2}{\nu_a}
 -\frac1{6n^2}\sum_a\frac{K_a^3}{\nu_a^2}
 +o(1),
\end{equation}
uniformly for \(k_\Sigma=O(n^{2/3})\).  To justify the expansion,
note that \(K_a\le k_\Sigma=o(n)\).  Along the given parameter sequence,
\Cref{lem:finite-parameter-stability} gives
\(\min_a\nu_a\ge c>0\) for all large \(n\), and therefore
\[
 \max_a\frac{K_a}{n\nu_a}\le\frac{k_\Sigma}{cn}=o(1).
\]
Expanding each sum
\(\sum_{j=0}^{K_a-1}\log(1-j/(n\nu_a))\), the replacement of
\(K_a(K_a-1)\) by \(K_a^2\), the lower order cubic terms, and the
fourth order remainder are bounded by
\[
 C\left(\frac {k_\Sigma}n+\frac{k_\Sigma^2}{n^2}
          +\frac{k_\Sigma^4}{n^3}\right)=o(1).
\]
Here we use that the type space is fixed and
\(\min_a\nu_a>0\).

\smallskip
\noindent\emph{Step 2: Perron decomposition and cancellation.}
Recall from \eqref{eq:Perron-quantities} that
\(p_{n,a}=\nu_au_a^{(n)}\) is the Perron law of the current critical
pair.  Write
\[
 K_a=k_\Sigma p_{n,a}+\Delta_a,
 \qquad \bDelta=(\Delta_a)_{a\in S},
 \qquad \sum_a\Delta_a=0.
\]
Thus \(\bDelta\) is the aggregate profile fluctuation around the
Perron law of the current empirical model.  The eigenfunction equation
\eqref{eq:Perron-eigenfunction} and the definitions in
\eqref{eq:Perron-quantities} give
\begin{equation}\label{eq:Perron-cancellations}
 \kappa p_n=u^{(n)},\quad
 \sum_a\frac{p_{n,a}^2}{\nu_a}=\alpha_n,\quad
 \sum_a\frac{p_{n,a}^3}{\nu_a^2}=\beta_n.
\end{equation}
Consequently
\begin{align}
 \sum_a\frac{K_a^2}{\nu_a}
 &=\alpha_n k_\Sigma^2+2k_\Sigma\sum_au_a^{(n)}\Delta_a
   +\sum_a\frac{\Delta_a^2}{\nu_a},
 \label{eq:falling-quadratic}\\
 \bK^{\mathsf T}\kappa\bK
 &=\alpha_n k_\Sigma^2+2k_\Sigma\sum_au_a^{(n)}\Delta_a
   +\bDelta^{\mathsf T}\kappa\bDelta.
 \label{eq:isolation-quadratic}
\end{align}
The quadratic part obtained by combining the first term in
\eqref{eq:falling-expansion} with the selected pair isolation exponent
in \eqref{eq:Rdet-exact} is, by
\eqref{eq:falling-quadratic}--\eqref{eq:isolation-quadratic},
\begin{align}
 -\frac1{2n}\sum_a\frac{K_a^2}{\nu_a}
 +\frac{t_n}{2n}\bK^{\mathsf T}\kappa\bK
 &=\frac{\alpha_n(t_n-1)k_\Sigma^2}{2n}
 +\frac{(t_n-1)k_\Sigma}{n}\sum_au_a^{(n)}\Delta_a\notag\\
 &\quad+\frac1{2n}\left(
 t_n\bDelta^{\mathsf T}\kappa\bDelta
 -\sum_a\frac{\Delta_a^2}{\nu_a}\right).
 \label{eq:deterministic-quadratic-decomposition}
\end{align}
At the critical value \(t_n=1\), the deterministic term and the term
linear in \(\bDelta\) vanish.  For the actual \(t_n\), the first term
on the right of \eqref{eq:deterministic-quadratic-decomposition} is the
deterministic quadratic contribution retained in
\eqref{eq:component-count-residual-asymptotic}.  The cubic sum in the
second term of \eqref{eq:falling-expansion} satisfies
\begin{align*}
 \sum_a\frac{K_a^3}{\nu_a^2}
 &=\beta_n k_\Sigma^3
 +3k_\Sigma^2\sum_a(u_a^{(n)})^2\Delta_a
 +3k_\Sigma\sum_a\frac{p_{n,a}\Delta_a^2}{\nu_a^2}
 +\sum_a\frac{\Delta_a^3}{\nu_a^2}.
\end{align*}

\smallskip
\noindent\emph{Step 3: Uniform control of the profile remainder.}
To control the profile terms in
\eqref{eq:deterministic-quadratic-decomposition} and the preceding
cubic expansion, put
\[
 \mathcal V_n=\sum_a\frac{\Delta_a^2}{p_{n,a}}.
\]
By \Cref{lem:finite-parameter-stability}, the fixed finite type set and the
strict positivity in \eqref{eq:Perron-eigenfunction} imply that all
\(p_{n,a}\) and \(\nu_a\) are at least some \(c>0\), uniformly in
\(a\) and all large \(n\).  Hence
Cauchy--Schwarz gives
\[
 \left|\sum_a(u_a^{(n)})^2\Delta_a\right|\le C\sqrt{\mathcal V_n},\qquad
 \sum_a\frac{p_{n,a}\Delta_a^2}{\nu_a^2}\le C\mathcal V_n,\qquad
 \sum_a\frac{|\Delta_a|^3}{\nu_a^2}\le C\mathcal V_n^{3/2}.
\]
Since \(K_a/k_\Sigma\) and \(p_{n,a}\) are probability vectors,
\[
 \mathcal V_n=k_\Sigma^2\sum_a\frac{(K_a/k_\Sigma-p_{n,a})^2}{p_{n,a}}
 \le c^{-1}k_\Sigma^2\sum_a(K_a/k_\Sigma-p_{n,a})^2
 \le 2c^{-1}k_\Sigma^2.
\]
Hence Young's inequality and \(k_\Sigma=O(n^{2/3})\) give, for every
profile,
\begin{align}
 \frac1{n^2}\left|
 \sum_a\frac{K_a^3}{\nu_a^2}-\beta_n k_\Sigma^3\right|
 &\le C\left\{
   \frac{k_\Sigma^2\sqrt{\mathcal V_n}}{n^2}
  +\frac{k_\Sigma\mathcal V_n}{n^2}+\frac{\mathcal V_n^{3/2}}{n^2}\right\}\notag\\
 &\le \frac Cn\mathcal V_n+C\frac{k_\Sigma^4}{n^3}
 =\frac Cn\mathcal V_n+o(1).
 \label{eq:global-cubic-profile-bound}
\end{align}
The last term in
\eqref{eq:deterministic-quadratic-decomposition} is bounded by
\(C\mathcal V_n/n\).  Its middle term satisfies
\[
 C\frac{|t_n-1|k_\Sigma}{n}\sqrt{\mathcal V_n}
 \le \frac Cn\mathcal V_n+C\frac{(t_n-1)^2k_\Sigma^2}{n}
 =\frac Cn\mathcal V_n+o(1).
\]
Define the centered logarithm of the integrand in
\eqref{eq:Rdet-exact} by
\begin{align}
 \mathcal E_n
 &=\log\cL_n(\bK)
 +\frac{t_n}{2n}\left(\bK^{\mathsf T}\kappa\bK
                   +\sum_a\kappa_{aa}K_a\right)\notag\\
 &\quad
 +\sum_{e\in E_\cup}\log\frac{e^{t_n\kappa_e/n}-1}{t_n\kappa_e/n}
 -\frac{\alpha_n(t_n-1)k_\Sigma^2}{2n}
 +\frac{\beta_n k_\Sigma^3}{6n^2}.
 \label{eq:deterministic-centered-log-integrand}
\end{align}
The quotient and its logarithm are defined as \(1\) and \(0\),
respectively, when \(\kappa_e=0\).  Equations
\eqref{eq:falling-expansion} and
\eqref{eq:global-cubic-profile-bound}, together with the uniform
diagonal and normalized edge odds estimates above, give
\[
 |\mathcal E_n|\le \frac Cn\mathcal V_n+o(1)
\]
for every outcome of the product component law, with an \(o(1)\) that is
also uniform over the selected connected graphs.

\smallskip
\noindent\emph{Step 4: Expectation under the product component law.}
By \eqref{eq:Rdet-exact} and
\eqref{eq:deterministic-centered-log-integrand},
\[
 \cR_n^{\rm det}
 =\exp\left\{
 \frac{\alpha_n(t_n-1)k_\Sigma^2}{2n}
 -\frac{\beta_n k_\Sigma^3}{6n^2}\right\}
 \E e^{\mathcal E_n}.
\]
By \Cref{lem:profile,lem:finite-parameter-stability}, specifically
\eqref{eq:Pearson-UI},
\begin{equation}\label{eq:deterministic-centered-residual-L1}
 \E|e^{\mathcal E_n}-1|
 \le e^{o(1)}\E\{e^{C\mathcal V_n/n}-1\}+o(1)=o(1).
\end{equation}
Therefore \(\E e^{\mathcal E_n}=1+o(1)\), which proves
\eqref{eq:component-count-residual-asymptotic} in the deterministic case.
\end{proof}

\subsection{Annealed compact type BJR counts}
\label{subsec:annealed-component-counts}

We now prove \Cref{prop:component-count-formula}\textup{(b)} in the
compact type BJR model of \Cref{def:BJR-model}.  The
\(n-k_\Sigma\) unselected vertex types are independent with law
\(\mu\), so their contribution to the isolation probability must be
averaged over \(\mu\).  For selected types
\(\boldsymbol x=(x_1,\ldots,x_{k_\Sigma})\in S^{k_\Sigma}\), define
\begin{align}
 D(\boldsymbol x)&=\sum_{i=1}^{k_\Sigma}d(x_i),&
 P_{\mathrm{int}}(\boldsymbol x)&=\sum_{1\le i<j\le k_\Sigma}\kappa(x_i,x_j),
 \label{eq:DP}\\
 S_{\boldsymbol x}(y)&=\sum_{i=1}^{k_\Sigma}\kappa(x_i,y),&
 A_n(\boldsymbol x)&=\int_S
 e^{-t_nS_{\boldsymbol x}(y)/n}\,\mu(\dd y).
 \label{eq:SA}
\end{align}
Here \(D\) is the total kernel degree of the selected vertices,
\(P_{\mathrm{int}}\) is their total internal pair weight, and
\(S_{\boldsymbol x}(y)\) the exposure of an unselected type \(y\) to
the selected union.  The quantity \(A_n(\boldsymbol x)\) is the
probability, averaged over the type of one unselected vertex, that it
has no edge to the selected union.
Under the product of the \(q\) laws
\(\cQ_{k_i,r_i,t_n}\), concatenate the types into
\(\boldsymbol X\), write
\(H_1,\ldots,H_q\) for the corresponding connected graphs, denote
expectation by \(\E_{\cQ}^{\rm ann}\), and use the notation in
\eqref{eq:selected-edge-notation}.  Put
\begin{equation}
 \cR_n^{\rm ann}=\frac{(n)_{k_\Sigma}}{n^{k_\Sigma}}\E_{\cQ}^{\rm ann}\!\left[e^{t_nD(\boldsymbol X)-t_nP_{\mathrm{int}}(\boldsymbol X)/n}A_n(\boldsymbol X)^{n-k_\Sigma}\prod_{e\in E_\cup}\frac{e^{t_n\kappa_e/n}-1}{t_n\kappa_e/n}\right].
 \label{eq:Rann-exact}
\end{equation}
As in \eqref{eq:Rdet-exact}, the quotient is defined as \(1\) when
\(\kappa_e=0\).

\begin{proof}[Proof of \Cref{prop:component-count-formula}\textup{(b)}: exact identity]
For each \(i\), let the sum below range over connected graphs \(H_i\)
on \([k_i]\) with \(e(H_i)=k_i+r_i-1\), and split the selected type
vector as
\(\boldsymbol x=(\boldsymbol x^{(1)},\ldots,\boldsymbol x^{(q)})\).
There are \((n)_{k_\Sigma}/\prod_i k_i!\) ordered disjoint candidate
blocks.  Conditional on \(\boldsymbol x\), factoring the absence
probability from every present internal edge gives
\begin{equation}\label{eq:annealed-fixed-configuration-probability}
 e^{-t_nP_{\mathrm{int}}(\boldsymbol x)/n}A_n(\boldsymbol x)^{n-k_\Sigma}\prod_{e\in E_\cup}(e^{t_n\kappa_e/n}-1).
\end{equation}
Consequently
\begin{equation}\label{eq:annealed-raw-component-sum}
 \E N_n(\bk,\boldsymbol r)=\frac{(n)_{k_\Sigma}}{\prod_i k_i!}\sum_{H_1,\ldots,H_q}\int_{S^{k_\Sigma}}e^{-t_nP_{\mathrm{int}}(\boldsymbol x)/n}A_n(\boldsymbol x)^{n-k_\Sigma}\prod_{e\in E_\cup}(e^{t_n\kappa_e/n}-1)\,\mu^{\otimes k_\Sigma}(\dd\boldsymbol x).
\end{equation}
For every configuration in this sum,
\begin{align}
 &e^{-t_nP_{\mathrm{int}}(\boldsymbol x)/n}A_n(\boldsymbol x)^{n-k_\Sigma}
   \prod_{e\in E_\cup}(e^{t_n\kappa_e/n}-1)\notag\\
 &\qquad=n^{-M}\prod_{i=1}^q\omega_{t_n}(H_i,\boldsymbol x^{(i)})
 e^{t_nD(\boldsymbol x)-t_nP_{\mathrm{int}}(\boldsymbol x)/n}
 A_n(\boldsymbol x)^{n-k_\Sigma}
 \prod_{e\in E_\cup}\frac{e^{t_n\kappa_e/n}-1}{t_n\kappa_e/n}.
 \label{eq:annealed-weight-factorization}
\end{align}
Substitution into \eqref{eq:annealed-raw-component-sum}, followed by
\eqref{eq:Q-definition} and \eqref{eq:Rann-exact}, gives
\eqref{eq:abstract-size-formula} in case~\textup{(b)}.
\end{proof}

The next lemma gives the uniform cumulant expansion needed for the
asymptotic analysis of \(\cR_n^{\rm ann}\) in
\eqref{eq:Rann-exact}.

\begin{lemma}[Uniform cumulant expansion]\label{lem:cumulant}
Fix \(L<\infty\), suppose that
\(|t_n-1|\le Ln^{-1/3}\) and \(k_\Sigma=O(n^{2/3})\), and abbreviate
\(A_n=A_n(\boldsymbol X)\), \(D=D(\boldsymbol X)\), and
\(P_{\mathrm{int}}=P_{\mathrm{int}}(\boldsymbol X)\).
Conditional on \(\boldsymbol X\), let \(\mathsf Y\sim\mu\) be
independent of the selected components and set
\(Z=S_{\boldsymbol X}(\mathsf Y)\).  Let \(m_1\), \(v_2\), and
\(c_3\) be the conditional mean, variance, and third cumulant of
\(Z\), respectively.  Thus \(Z\) is the exposure of a fresh type to
the selected vertices, and \(m_1,v_2,c_3\) are the first three
coefficients of its conditional logarithmic Laplace expansion.
Then
\begin{equation}
 (n-k_\Sigma)\log A_n+t_nD-\frac{t_nP_{\mathrm{int}}}{n}=\frac{t_n k_\Sigma D}{n}-\frac{t_nP_{\mathrm{int}}}{n}+\frac{t_n^2v_2}{2n}-\frac{t_n^2k_\Sigma v_2}{2n^2}-\frac{t_n^3c_3}{6n^2}+o(1),
 \label{eq:cumulant-expansion}
\end{equation}
uniformly over \(\boldsymbol X\in S^{k_\Sigma}\).
\end{lemma}

\begin{proof}
Since \(m_1=D\) and
\(0\le Z\le k_\Sigma\|\kappa\|_\infty\), Taylor's formula for the
logarithmic Laplace transform, uniformly in a neighborhood of the origin,
gives
\[
 \log\E e^{-t_nZ/n}
 =-\frac{t_nm_1}{n}+\frac{t_n^2v_2}{2n^2}
 -\frac{t_n^3c_3}{6n^3}+O(k_\Sigma^4/n^4).
\]
The fourth derivative is a fourth cumulant under an exponential tilt
and is bounded by \(Ck_\Sigma^4\); likewise,
\(|c_3|\le Ck_\Sigma^3\).  In multiplying by \(n-k_\Sigma\), the term
\(t_n^3k_\Sigma c_3/(6n^3)\) is
\(O(k_\Sigma^4/n^3)\) and is absorbed into the remainder.  Thus
\begin{align*}
 (n-k_\Sigma)\log A_n
 &=-t_nD+\frac{t_n k_\Sigma D}{n}
   +\frac{t_n^2v_2}{2n}
   -\frac{t_n^2k_\Sigma v_2}{2n^2}
   -\frac{t_n^3c_3}{6n^2}
   +O(k_\Sigma^4/n^3).
\end{align*}
Since \(k_\Sigma=O(n^{2/3})\), the error is
\(O(n^{-1/3})=o(1)\).  Adding \(t_nD-t_nP_{\mathrm{int}}/n\) proves
\eqref{eq:cumulant-expansion}.
\end{proof}

For a finite signed Borel measure \(\zeta\) on \(S\), define the
continuous function \(\cK\zeta\in C(S)\) by
\[
 (\cK\zeta)(y)=\int_S\kappa(x,y)\,\zeta(\dd x).
\]
Thus \(\cK\zeta\) is obtained by integrating the kernel against
\(\zeta\).
By symmetry of \(\kappa\), \(\cK(f\mu)=Kf\), where \(K\) is the
integral operator in \eqref{eq:integral-operator}.

For \(\varpi\in\cP(S)\) and \(t\) near \(1\), define
\begin{align}
 \Psi(\varpi)
 &=-\frac12+\varpi(d)
 -\frac12\iint\kappa\,\dd \varpi^{\otimes2}
 +\frac12\int(\cK \varpi)^2\,\dd\mu
 -\frac12\varpi(d)^2,
 \label{eq:Psi-profile}\\
 \Phi_t(\varpi)
 &=-\frac12+t\varpi(d)
   -\frac t2\iint\kappa\,\dd\varpi^{\otimes2}
   +\frac{t^2}{2}
    \left\{\int(\cK\varpi)^2\,\dd\mu-\varpi(d)^2\right\},
 \label{eq:Phi-profile}\\
 \Gamma_t(\varpi)
 &=-\frac16-\frac{t^2}{2}
    \left\{\int(\cK\varpi)^2\,\dd\mu-\varpi(d)^2\right\}
   -\frac{t^3}{6}\int\{\cK\varpi-\varpi(d)\}^3\,\dd\mu.
 \label{eq:Gamma-profile}
\end{align}
In the expansion of \(\log\cR_n^{\rm ann}\), \(\Phi_t\) multiplies
\(k_\Sigma^2/n\), \(\Psi=\Phi_1\) is its critical form, and
\(\Gamma_t\) multiplies \(k_\Sigma^3/n^2\).  The next lemma expands
\(\cR_n^{\rm ann}\) about the Perron law.

\begin{lemma}[Annealed profile expansion and Perron centering]
\label{lem:annealed-profile-expansion}
Assume the compact type BJR conditions \textup{(B1)--(B2)} in
\Cref{def:BJR-model}.  Fix the number of components and their
surpluses, let \(|t_n-1|\le Ln^{-1/3}\), and let each
\(k_i/n^{2/3}\) range over a fixed compact subset of \((0,\infty)\).
Under the product of the component laws
\(\cQ_{k_i,r_i,t_n}\), sample the typed connected graphs
\((H_i,X_{i1},\ldots,X_{ik_i})\) independently, and denote expectation
by \(\E_{\cQ}^{\rm ann}\).  Write
\(\boldsymbol X=(X_1,\ldots,X_{k_\Sigma})\) for the concatenated type
list.  Its empirical profile is the combined measure \(L_\Sigma\)
from \eqref{eq:total-size-notation}; put \(\delta_\Sigma:=L_\Sigma-p\).
Then, uniformly over these ranges,
\begin{equation}\label{eq:Psi-centered}
 \Psi(L_\Sigma)=\frac12\left\{
 \int(\cK\delta_\Sigma)^2\,\dd\mu
 -\iint\kappa\,\dd\delta_\Sigma^{\otimes2}
 -\delta_\Sigma(d)^2\right\},
\end{equation}
so \(\Psi(p)=0\).  Moreover, there is a remainder
\(\epsilon_n=\epsilon_n(H_1,\ldots,H_q,\boldsymbol X)\), tending to
zero uniformly over the displayed parameter and size ranges and over
all configurations in the support of the product component law, such
that
\begin{equation}\label{eq:BJR-log-profile-decomposition}
 \log\cR_n^{\rm ann}
 =\log\E_{\cQ}^{\rm ann}\exp\left\{
 \frac{k_\Sigma^2}{n}\Phi_{t_n}(L_\Sigma)
 +\frac{k_\Sigma^3}{n^2}\Gamma_{t_n}(L_\Sigma)+\epsilon_n
 \right\},
\end{equation}
and, uniformly for \(t\) near \(1\),
\begin{align}
 \Phi_t(p)
 &=-\frac12+t-\frac{\alpha t}{2}
   +\frac{\alpha-1}{2}t^2
  =\frac\alpha2(t-1)+O((t-1)^2),
 \label{eq:Phi-at-p}\\
 \Gamma_t(p)
 &=-\frac16-\frac{\alpha-1}{2}t^2
   -\frac{\beta-3\alpha+2}{6}t^3
  =-\frac\beta6+O(|t-1|).
 \label{eq:Gamma-at-p}
\end{align}
\end{lemma}

\begin{proof}
Constants are uniform over the stated ranges and may change from line
to line.

Taylor expansion of the deterministic vertex selection factor gives
\begin{equation}\label{eq:falling-BJR}
 \log\frac{(n)_{k_\Sigma}}{n^{k_\Sigma}}
 =-\frac{k_\Sigma^2}{2n}-\frac{k_\Sigma^3}{6n^2}+o(1).
\end{equation}
Indeed, summing
\(\log(1-j/n)=-j/n-j^2/(2n^2)+O(j^3/n^3)\) over
\(0\le j<k_\Sigma\) gives \eqref{eq:falling-BJR}; replacing the
resulting falling powers by \(k_\Sigma^2\) and \(k_\Sigma^3\) costs
\(O(k_\Sigma/n+k_\Sigma^2/n^2)\), and the remainder is
\(O(k_\Sigma^4/n^3)=o(1)\).
Combining \eqref{eq:Rann-exact}, \eqref{eq:falling-BJR}, and
\Cref{lem:cumulant} gives
\begin{equation}\label{eq:Rann-explicit-log}
\begin{aligned}
 \log\cR_n^{\rm ann}
 =\log\E_{\cQ}^{\rm ann}\exp\Bigg\{
 &-\frac{k_\Sigma^2}{2n}-\frac{k_\Sigma^3}{6n^2}
 +\frac{t_n k_\Sigma D}{n}-\frac{t_nP_{\mathrm{int}}}{n}\\
 &+\frac{t_n^2v_2}{2n}-\frac{t_n^2k_\Sigma v_2}{2n^2}
 -\frac{t_n^3c_3}{6n^2}
 +\sum_{e\in E_\cup}\log\frac{e^{t_n\kappa_e/n}-1}{t_n\kappa_e/n}
 +\epsilon_n^{(0)}\Bigg\}.
\end{aligned}
\end{equation}
Here \(\epsilon_n^{(0)}\) tends to zero uniformly over the stated
ranges and configurations.
\smallskip
\noindent\emph{The quadratic profile identity and its critical cancellation.}
At \(t=1\), the terms on the \(k_\Sigma^2/n\) scale in the exponent of
\eqref{eq:Rann-explicit-log} satisfy
\begin{align}
 &-\frac{k_\Sigma^2}{2n}
 +\frac{k_\Sigma D}{n}-\frac{P_{\mathrm{int}}}{n}+\frac{v_2}{2n}\notag\\
 &=\frac{k_\Sigma^2}{n}\left[-\frac12+L_\Sigma(d)
 -\frac12\iint\kappa\,\dd L_\Sigma^{\otimes2}
 +\frac12\left\{\int(\cK L_\Sigma)^2\,\dd\mu
 -L_\Sigma(d)^2\right\}\right]+o(1)\notag\\
 &=\frac{k_\Sigma^2}{n}\Psi(L_\Sigma)+o(1).
 \label{eq:Psi-critical-quadratic-term}
\end{align}

Indeed, for \(D=D(\boldsymbol X)\) and \(P_{\mathrm{int}}=P_{\mathrm{int}}(\boldsymbol X)\)
from \eqref{eq:DP}, and for the conditional variance \(v_2\) of
\(S_{\boldsymbol X}(\mathsf Y)\) from \Cref{lem:cumulant}, the identity
\(L_\Sigma=k_\Sigma^{-1}\sum_{i=1}^{k_\Sigma}\delta_{X_i}\)
gives
\begin{align}
 D&=k_\Sigma L_\Sigma(d),\qquad
 P_{\mathrm{int}}=\frac{k_\Sigma^2}{2}\iint\kappa\,\dd L_\Sigma^{\otimes2}
 \notag-\frac12\sum_{i=1}^{k_\Sigma}\kappa(X_i,X_i),\notag\\
 v_2&=k_\Sigma^2\left\{\int(\cK L_\Sigma)^2\,\dd\mu
                    -L_\Sigma(d)^2\right\}.
 \label{eq:BJR-profile-cumulants}
\end{align}
The first equality in \eqref{eq:Psi-critical-quadratic-term} follows
from these identities.  Since \(k_\Sigma/n=o(1)\), the diagonal part of
\(P_{\mathrm{int}}\) is negligible.  The second equality is \eqref{eq:Psi-profile}.
The terms omitted from this quadratic scale
are the \(k_\Sigma^3/n^2\) terms retained below and the \(o(1)\)
normalized edge odds terms.

\smallskip
\noindent\emph{Quadratic centering at the Perron law.}
We verify \eqref{eq:Psi-centered}, which says that \(\Psi\) has neither
a constant term nor a linear term about the Perron law \(p\).

Symmetry, the Perron equation \eqref{eq:Perron-eigenfunction}, and
the definition of \(p\) in \eqref{eq:Perron-quantities} give
\(\cK p=u\).  The definition of \(d\) in
\eqref{eq:Perron-quantities} and symmetry also give
\[
 p(d)=\int_Su(x)d(x)\,\mu(\dd x)
 =\iint_{S^2}u(x)\kappa(x,y)\,\mu(\dd x)\mu(\dd y)
 =\int_S(Ku)(y)\,\mu(\dd y)=1.
\]
The two quadratic functionals in \eqref{eq:Psi-profile} have the same
first variation at \(p\).
Indeed, linearity of \(\cK\) and \(\cK p=u\) give
\begin{align*}
 \int(\cK L_\Sigma)^2\,\dd\mu
 &=\alpha+2\int u\,\dd\delta_\Sigma
   +\int(\cK\delta_\Sigma)^2\,\dd\mu,\\
 \iint\kappa\,\dd L_\Sigma^{\otimes2}
 &=\alpha+2\int u\,\dd\delta_\Sigma
   +\iint\kappa\,\dd\delta_\Sigma^{\otimes2}.
\end{align*}
Subtracting these identities yields
\begin{equation}\label{eq:quadratic-stationarity}
 \int(\cK L_\Sigma)^2\,\dd\mu
 -\iint\kappa\,\dd L_\Sigma^{\otimes2}
 =
 \int(\cK\delta_\Sigma)^2\,\dd\mu
 -\iint\kappa\,\dd\delta_\Sigma^{\otimes2}.
\end{equation}

Finally, substituting \(p(d)=1\) and
\eqref{eq:quadratic-stationarity} into \eqref{eq:Psi-profile} proves
\eqref{eq:Psi-centered}.

\smallskip
\noindent\emph{Logarithmic expansion and evaluation at the Perron law.}
The third cumulant from \Cref{lem:cumulant} is
\begin{equation}\label{eq:BJR-third-cumulant}
 c_3=k_\Sigma^3\int\{\cK L_\Sigma-L_\Sigma(d)\}^3\,\dd\mu.
\end{equation}
Substituting \eqref{eq:BJR-profile-cumulants} and
\eqref{eq:BJR-third-cumulant} into \eqref{eq:Rann-explicit-log} gives
\eqref{eq:BJR-log-profile-decomposition}.  Indeed, boundedness of
\(\kappa\) makes the diagonal contribution \(O(k_\Sigma/n)=o(1)\),
while the logarithm of the normalized edge odds product is
\(O(M/n)=o(1)\) by \eqref{eq:KRM}.  These terms, together with
\(\epsilon_n^{(0)}\), form \(\epsilon_n\) with the asserted uniform
bound.

To verify \eqref{eq:Phi-at-p}--\eqref{eq:Gamma-at-p}, the identities
established above give
\[
 \iint\kappa\,\dd p^{\otimes2}=\int u^2\,\dd\mu=\alpha,
 \qquad \int(\cK p)^2\,\dd\mu=\alpha.
\]
Moreover, expanding \(\int(u-1)^3\,\dd\mu\) and using
\(\cK p=u\), \(p(d)=1\), and \(\int u\,\dd\mu=1\) gives
\begin{equation}\label{eq:BJR-third-cumulant-at-p}
 \int(\cK p-p(d))^3\,\dd\mu=\beta-3\alpha+2.
\end{equation}
Substitution into \eqref{eq:Phi-profile}--\eqref{eq:Gamma-profile}
proves \eqref{eq:Phi-at-p}--\eqref{eq:Gamma-at-p}.
Since \(k_\Sigma=O(n^{2/3})\) and \(t_n-1=O(n^{-1/3})\), multiplying
the error terms in \eqref{eq:Phi-at-p}--\eqref{eq:Gamma-at-p} by
\(k_\Sigma^2/n\) and \(k_\Sigma^3/n^2\), respectively, gives \(o(1)\).
\end{proof}

\begin{lemma}[Negligibility of annealed profile remainders]
\label{lem:annealed-profile-remainders}
Under the hypotheses and notation of
\Cref{lem:annealed-profile-expansion}, assume in addition \QP{} from
\Cref{ass:QP}.  Define
\begin{align*}
 \mathcal Z_n^{(2)}&=\frac{k_\Sigma^2}{n}\Psi(L_\Sigma),\qquad
 \mathcal Z_n^{({\rm mix})}=\frac{k_\Sigma^2}{n}
 \{\Phi_{t_n}(L_\Sigma)-\Phi_{t_n}(p)-\Psi(L_\Sigma)\},\\
 \mathcal Z_n^{(3)}&=\frac{k_\Sigma^3}{n^2}
 \{\Gamma_{t_n}(L_\Sigma)-\Gamma_{t_n}(p)\},
 \qquad
 \mathcal Z_n=\mathcal Z_n^{(2)}+
 \mathcal Z_n^{({\rm mix})}+\mathcal Z_n^{(3)}.
\end{align*}
Then, uniformly over the parameter and size ranges in
\Cref{lem:annealed-profile-expansion},
\begin{equation}\label{eq:BJR-total-profile-remainder}
 \E_{\cQ}^{\rm ann}|e^{\mathcal Z_n}-1|\longrightarrow0.
\end{equation}
\end{lemma}

\begin{proof}
Here \(L_\Sigma\) is the random type profile and \(p\) is the
deterministic Perron law.  By definition, \(\mathcal Z_n\) is the
difference between the profile functionals in
\eqref{eq:BJR-log-profile-decomposition} evaluated at these two
measures.  We prove
\begin{align}
 &\E_{\cQ}^{\rm ann}|e^{\mathcal Z_n^{(\bullet)}}-1|\longrightarrow0
 \quad\text{for }\bullet\in\{2,{\rm mix},3\},
 \label{eq:BJR-profile-remainder-L1}\\
 &\E_{\cQ}^{\rm ann}e^{s|\mathcal Z_n^{(2)}|}=1+o(1),\qquad
 |\mathcal Z_n^{({\rm mix})}|+|\mathcal Z_n^{(3)}|\le C
 \quad\text{almost surely}
 \label{eq:BJR-profile-remainder-moments}
\end{align}
for every fixed \(s\ge0\), with \(C\) independent of \(n\).

\smallskip
\noindent\emph{The quadratic profile remainder.}
The canonical kernel \(\overline H\) defined in
\Cref{lem:profile}\textup{(iv)} represents the centered quadratic
functional in \(\Psi\).  By
\Cref{lem:QP-correction-kernel}, \QP{} for \(\kappa\) is equivalent to
projective regularity of \(\overline H\), and
\eqref{eq:profile-canonical-centering} shows that \(\overline H\) is
canonical with respect to \(p\).  Therefore the definition of
\(\mathcal Y_{k_\Sigma}(\overline H)\) in
\eqref{eq:QP-profile-statistic}, together with
\eqref{eq:Psi-centered}, gives
\begin{equation}\label{eq:Psi-QP-identity}
 \mathcal Y_{k_\Sigma}(\overline H)
 =k_\Sigma\iint\overline H\,\dd L_\Sigma^{\otimes2}
 =2k_\Sigma\Psi(L_\Sigma).
\end{equation}
Consequently
\[
 \mathcal Z_n^{(2)}
 =\frac{k_\Sigma}{2n}\mathcal Y_{k_\Sigma}(\overline H).
\]
The size range in \Cref{lem:annealed-profile-expansion} gives
\(k_\Sigma/n=O(k_\Sigma^{-1/2})\).  Therefore the product version of
the quadratic profile estimate \eqref{eq:QP-BJR-scale} gives the first
claim in \eqref{eq:BJR-profile-remainder-moments}.  The corresponding
claim in \eqref{eq:BJR-profile-remainder-L1} follows from
\(|e^z-1|\le e^{|z|}-1\).  No exponential moment with a fixed
coefficient multiplying \(\mathcal Y_{k_\Sigma}(\overline H)\) is
needed.

\smallskip
\noindent\emph{The mixed profile remainder.}
Since \(\Phi_1(\varpi)=\Psi(\varpi)\) and \(\Phi_1(p)=0\), direct
subtraction in \eqref{eq:Phi-profile} gives
\begin{align}
 &\Phi_t(\varpi)-\Phi_t(p)-\Psi(\varpi)\notag\\
 &\quad=(t-1)\Bigg\{\varpi(d)-p(d)
 -\frac12\left[\iint\kappa\,\dd\varpi^{\otimes2}
                -\iint\kappa\,\dd p^{\otimes2}\right]\notag\\
 &\quad
 +\frac{t+1}{2}\left[\int(\cK\varpi)^2\,\dd\mu-\varpi(d)^2
                -\int(\cK p)^2\,\dd\mu+p(d)^2\right]\Bigg\}.
 \label{eq:Phi-random-remainder}
\end{align}
The expression in braces is bounded and weakly continuous on
\(\cP(S)\).  For the terms involving \(\cK\), this follows from
continuity of \(\kappa\), compactness of \(S\), and the following
implication, valid for every sequence
\((\varpi_m)\subset\cP(S)\):
\begin{equation}\label{eq:kernel-weak-uniform-continuity}
 \varpi_m\Rightarrow\varpi
 \quad\Longrightarrow\quad
 \|\cK(\varpi_m-\varpi)\|_\infty\longrightarrow0.
\end{equation}
This implication follows by taking a finite uniform net in the compact
family \(\{\kappa(\cdot,y):y\in S\}\subset C(S)\).  By
\eqref{eq:profile-main-tail}, the expression in braces in
\eqref{eq:Phi-random-remainder}, evaluated at \(L_\Sigma\),
converges to zero in probability.  Since
\(
 |t_n-1|k_\Sigma^2/n=O(1)
\), \(\mathcal Z_n^{({\rm mix})}\to0\) in probability and is
uniformly bounded.  Dominated convergence proves the mixed claims in
\eqref{eq:BJR-profile-remainder-L1}--\eqref{eq:BJR-profile-remainder-moments}.

\smallskip
\noindent\emph{The cubic profile remainder.}
By \eqref{eq:kernel-weak-uniform-continuity},
\(\Gamma_t(\varpi)\) is bounded and jointly continuous for \(t\) near
\(1\) and \(\varpi\) in the weak topology.  Therefore
\eqref{eq:profile-main-tail} gives
\(\Gamma_{t_n}(L_\Sigma)\to\Gamma_1(p)\) in probability.  Since
\(\Gamma_{t_n}(p)\to\Gamma_1(p)\),
\(k_\Sigma^3/n^2=O(1)\), and \(\Gamma\) is uniformly bounded, dominated
convergence proves the cubic claim in
\eqref{eq:BJR-profile-remainder-L1}.  Uniform boundedness of \(\Gamma\)
also gives the remaining bound in
\eqref{eq:BJR-profile-remainder-moments}.

Combining the three estimates, we have
\[
 |e^{\mathcal Z_n}-1|
 \le e^{|\mathcal Z_n^{(2)}|+|\mathcal Z_n^{({\rm mix})}|}
       |e^{\mathcal Z_n^{(3)}}-1|
    +e^{|\mathcal Z_n^{(2)}|}
       |e^{\mathcal Z_n^{({\rm mix})}}-1|
    +|e^{\mathcal Z_n^{(2)}}-1|.
\]
H\"older's inequality and
\eqref{eq:BJR-profile-remainder-L1}--\eqref{eq:BJR-profile-remainder-moments},
used also at a larger fixed H\"older order, give
\eqref{eq:BJR-total-profile-remainder}.
\end{proof}

\begin{proof}[Proof of \Cref{prop:component-count-formula}\textup{(b)}:
critical asymptotic]
By \eqref{eq:BJR-log-profile-decomposition},
\[
 \cR_n^{\rm ann}
 =\exp\left\{
 \frac{k_\Sigma^2}{n}\Phi_{t_n}(p)
 +\frac{k_\Sigma^3}{n^2}\Gamma_{t_n}(p)\right\}
 \E_{\cQ}^{\rm ann}e^{\mathcal Z_n+\epsilon_n}.
\]
Let \(\bar\epsilon_n=\sup|\epsilon_n|\), where the supremum is over the
parameter, size, and configuration ranges in
\Cref{lem:annealed-profile-expansion}.  Then
\(\bar\epsilon_n\to0\).  By
\Cref{lem:annealed-profile-remainders},
\begin{equation}\label{eq:annealed-centered-residual-L1}
 \E_{\cQ}^{\rm ann}|e^{\mathcal Z_n+\epsilon_n}-1|
 \le e^{\bar\epsilon_n}
      \E_{\cQ}^{\rm ann}|e^{\mathcal Z_n}-1|
      +(e^{\bar\epsilon_n}-1)=o(1).
\end{equation}
Consequently,
\(\E_{\cQ}^{\rm ann}e^{\mathcal Z_n+\epsilon_n}=1+o(1)\).  Moreover,
\eqref{eq:Phi-at-p}--\eqref{eq:Gamma-at-p} and the critical scaling give
\[
 \frac{k_\Sigma^2}{n}\Phi_{t_n}(p)
 +\frac{k_\Sigma^3}{n^2}\Gamma_{t_n}(p)
 =\frac{\alpha(t_n-1)k_\Sigma^2}{2n}
  -\frac{\beta k_\Sigma^3}{6n^2}+o(1).
\]
Combining these displays proves
\eqref{eq:component-count-residual-asymptotic} in the annealed case.
\end{proof}

\section{Factorial moments and point process convergence}
\label{sec:factorial-process}

In this section we prove \Cref{thm:local-main}.  We combine the
Perron--Wright formula with the component count formula to obtain the
fixed size factorial asymptotics, and then use factorial moment measures
and type profile concentration to prove the unmarked and marked process
limits.  Related local estimates for the homogeneous model appear in
\cite{vanDerHofstadKagerMuller2009}; the finite \(n\) factorial measure
method is developed in \cite{Sun2026ER}.  The uniform finite type form
used at fixed countable truncations is stated in
\Cref{cor:uniform-finite-type-estimates}.

Unless otherwise indicated, \(\E\) denotes expectation with respect to
the law of \(G_n\) in the model currently under consideration.

\subsection{Fixed size factorial asymptotics}

We first combine the component count formula with the critical window
Perron--Wright asymptotic to prove
\Cref{thm:local-main}\textup{(i)}.

\begin{proof}[Proof of \Cref{thm:local-main}\textup{(i)}]
Fix \(q,r_1,\ldots,r_q,a,b\) as in the theorem.  For
\((x_1,\ldots,x_q)\in[a,b]^q\), set
\(k_i=\lfloor x_i n^{2/3}\rfloor\).  Retain
\(\bk=(k_1,\ldots,k_q)\), \(\boldsymbol r=(r_1,\ldots,r_q)\),
\(k_\Sigma\), \(M\), and \(N_n(\bk,\boldsymbol r)\) from
\eqref{eq:KRM} and
\eqref{eq:orderedcomponentcount}.

For the deterministic finite type model, apply
\Cref{prop:component-count-formula}\textup{(a)} with the exactly
critical normalization \eqref{eq:finite-critical-normalization}; for
the annealed compact type BJR model, apply part~\textup{(b)} directly.
In either case, let \(\tau_n\) denote the scalar parameter paired with
the exactly critical kernel, and let \(\alpha_n,\beta_n\) be its Perron
constants.
Thus \(\tau_n=\bar t_n\) in the finite type model, whereas
\(\tau_n=t_n\) and
\((\alpha_n,\beta_n)=(\alpha,\beta)\) in the compact model.
In the finite type case, put
\(\vartheta_n=(\nu^{(n)},\kappa^{(n)})\) and write
\(\CW_t^{(n)}:=\CW_t^{\vartheta_n}\); in the compact case the
superscript is omitted.
In the finite type model, the estimate
\(\nu^{(n)}=\mu+O(n^{-1})\), finite dimensional perturbation, and
\Cref{lem:finite-parameter-stability} yield
\begin{equation}\label{eq:finite-empirical-parameters}
 \tau_n=1+\lambda n^{-1/3}+O(n^{-1}),
 \qquad
 (\alpha_n,\beta_n)\longrightarrow(\alpha,\beta).
\end{equation}
In the compact model, the kernel and type law do not depend on \(n\),
so \(\tau_n=t_n=1+\lambda n^{-1/3}\) exactly and
\((\alpha_n,\beta_n)=(\alpha,\beta)\) for every \(n\).

The exact identity and its residual asymptotic in
\Cref{prop:component-count-formula} yield
\begin{equation}\label{eq:master-before-enum}
 \E N_n(\bk,\boldsymbol r)=\frac{n^{k_\Sigma-M}}{\prod_i k_i!}
 \prod_{i=1}^q\CW_{\tau_n}^{(n)}(k_i,r_i)
 \exp\!\left\{\frac{\alpha_n(\tau_n-1)k_\Sigma^2}{2n}
 -\frac{\beta_n k_\Sigma^3}{6n^2}\right\}(1+o(1)).
\end{equation}
Here the superscript \(n\) on \(\CW\) is omitted in the annealed
compact type BJR case.

The Perron--Wright formula \Cref{thm:Perron-Wright} applies to the
fixed compact type kernel, while \eqref{eq:actual-sequence-PW} applies
along the finite type parameter sequence.  In either case,
\[
 \CW_{\tau_n}^{(n)}(k_i,r_i)
 =(1+o(1))w_{r_i}\beta_n^{(r_i-1)/2}
 e^{-k_i}
 \exp\left\{-\frac{\alpha_n^2}{2\beta_n}
                  k_i(\tau_n-1)^2\right\}
 k_i^{\,k_i+3r_i/2-2}.
\]
Multiplying the Gaussian factors gives
\begin{equation}\label{eq:PW-Gaussian-product}
 \exp\left\{-\frac{\alpha_n^2}{2\beta_n}
                 k_\Sigma(\tau_n-1)^2\right\}
 =
 \exp\left\{
 -\frac{\alpha_n^2\lambda^2k_\Sigma}
         {2\beta_n n^{2/3}}\right\}(1+o(1)).
\end{equation}
Indeed, \eqref{eq:finite-empirical-parameters} gives
\((\tau_n-1)^2-\lambda^2n^{-2/3}=O(n^{-4/3})\); after multiplication
by \(k_\Sigma=O(n^{2/3})\), this difference is also \(o(1)\),
uniformly on \([a,b]^q\).

Put \(s_n=k_\Sigma/n^{2/3}\), the total size of the selected
components measured on the \(n^{2/3}\) scale.  Uniformly on \([a,b]^q\),
\(s_n=S_q+O(n^{-2/3})\), and
\(n^{1/3}(\tau_n-1)=\lambda+o(1)\).  Combining
\eqref{eq:finite-empirical-parameters} with the exponent in
\eqref{eq:master-before-enum} and the Gaussian term
\eqref{eq:PW-Gaussian-product} gives
\[
 -\frac{\alpha_n^2\lambda^2s_n}{2\beta_n}
 +\frac{\alpha_n n^{1/3}(\tau_n-1)s_n^2}{2}
 -\frac{\beta_ns_n^3}{6}
 =-F_{\alpha,\beta}(S_q,\lambda)+o(1).
\]
For each \(i\), Stirling's formula gives, uniformly for
\(x_i\in[a,b]\),
\begin{equation}\label{eq:Stirling-fixed-size}
 \frac{n^{1-r_i}e^{-k_i}k_i^{\,k_i+3r_i/2-2}}{k_i!}
 =\frac{1+o(1)}{\sqrt{2\pi}}\,n^{-2/3}x_i^{3r_i/2-5/2}.
\end{equation}
Since \(M=\sum_i(k_i+r_i-1)\), the nonexponential factor in
\eqref{eq:master-before-enum} is the product of the left sides of
\eqref{eq:Stirling-fixed-size}.  Multiplying that identity over \(i\),
inserting the Wright amplitudes \(w_{r_i}\), and using
\(\beta_n\to\beta\) give,
uniformly on \([a,b]^q\),
\[
 \E N_n(\bk,\boldsymbol r)
 =n^{-2q/3}
 m_{\alpha,\beta,\lambda}^{(q)}
 ((x_1,r_1),\ldots,(x_q,r_q))\{1+o(1)\}.
\]
This is \eqref{eq:factorial-asymptotic}.
\end{proof}

\subsection{From factorial measures to the process limit}

Part~\textup{(i)} gives the local factorial moment measures.  The
criterion in \Cref{lem:factorial-criterion}, proved in
\Cref{app:point-process-tools}, yields convergence of the unmarked
process, \Cref{prop:AW-density} identifies its limit, and
\Cref{lem:profile,lem:mark-replacement} then give the marked
convergence.

\begin{proof}[Proof of \Cref{thm:local-main}\textup{(ii)}]
\smallskip
\noindent\emph{Step 1: Local factorial moment measures.}
Fix a local window \(B=B_{a,b,R}\) and let
\(\xi_{n,B}=\Xi_n^\star|_B\) be the restriction of the component process
to \(B\).  For \(q\ge1\) and \(f\in C(B^q)\),
\eqref{eq:factorial-measure} and \eqref{eq:orderedcomponentcount} give
\begin{align}
 \int_{B^q}f\,\dd\mathfrak a_{\xi_{n,B}}^{(q)}
 &=\sum_{\boldsymbol r\in\{0,\ldots,R\}^q}
   \sum_{\substack{\boldsymbol k:\\
      \beta^{1/3}n^{-2/3}k_i\in[a,b]}}
 f\bigl((\beta^{1/3}n^{-2/3}k_i,r_i)_{i=1}^q\bigr)
 \E N_n(\bk,\boldsymbol r).
 \label{eq:factorial-Riemann-sum}
\end{align}
For these integer sizes, \(x_i=k_i/n^{2/3}\) remains in a fixed compact
subset of \((0,\infty)\).  The uniform estimate in
\Cref{thm:local-main}\textup{(i)} therefore turns
\eqref{eq:factorial-Riemann-sum} into a Riemann sum, uniformly over the
finitely many surplus vectors.  Its limit is
\begin{equation}
 \sum_{\boldsymbol r\in\{0,\ldots,R\}^q}\int_{(\beta^{-1/3}[a,b])^q}f\bigl((\beta^{1/3}x_i,r_i)_{i=1}^q\bigr)m_{\alpha,\beta,\lambda}^{(q)}\bigl((x_i,r_i)_{i=1}^q\bigr)\prod_{i=1}^q\dd x_i.
 \label{eq:factorial-Riemann-limit}
\end{equation}
Thus the factorial moment measures converge weakly at every fixed
order.  Applying the same estimate at order \(q+1\) with \(f=1\) gives
\begin{equation}\label{eq:local-factorial-bound}
 \sup_n\E(\xi_{n,B}(B))_{q+1}<\infty;
\end{equation}
the finitely many initial values of \(n\) contribute only a finite
maximum.

\smallskip
\noindent\emph{Step 2: Identification of the limiting factorial measures.}
In \eqref{eq:factorial-Riemann-limit}, make the change of variables
\[
 y_i=\beta^{1/3}x_i,\qquad
 Y=\sum_i y_i,\qquad
 \Lambda=\alpha\beta^{-2/3}\lambda.
\]
Then
\begin{equation}\label{eq:F-rescale}
 F_{\alpha,\beta}(S_q,\lambda)
 =\frac{Y^3}{6}-\frac{\Lambda Y^2}{2}
 +\frac{\Lambda^2Y}{2}.
\end{equation}
For each coordinate, \(x_i=\beta^{-1/3}y_i\) and
\(\dd x_i=\beta^{-1/3}\dd y_i\).  Consequently, the factor involving
\(\beta\), \(x_i\), and \(\dd x_i\) in \eqref{eq:m-density} becomes
\[
 \beta^{(r_i-1)/2}x_i^{3r_i/2-5/2}\dd x_i
 =\beta^{(r_i-1)/2-(3r_i/2-5/2)/3-1/3}
   y_i^{3r_i/2-5/2}\dd y_i
 =y_i^{3r_i/2-5/2}\dd y_i.
\]
Thus all powers of \(\beta\) cancel, and the transformed density in
\eqref{eq:factorial-Riemann-limit} is the Aldous--Wright density
\eqref{eq:AW-density}.  By \Cref{prop:AW-density}, these are the
factorial moment measures of \(\Xi_\Lambda^{\Ald}|_B\), and they
determine its joint count laws on finite partitions of \(B\).

\smallskip
\noindent\emph{Step 3: From a local window to the full point process.}
Let \(D_B=(\mathbb Q\cap[a,b])\cup\{a,b\}\), and let \(\mathcal A_B\) be
the countable algebra generated by the sets
\begin{equation}\label{eq:local-generating-ring}
 ((s,t]\cap[a,b])\times\{r\}
 \quad\text{and}\quad \{a\}\times\{r\},\qquad
 s,t\in D_B,\quad s<t,\quad r\in\{0,\ldots,R\},
\end{equation}
where the singleton sets supply the left endpoints omitted by the
half-open intervals.  This is the countable algebra of local test sets
used in the factorial moment criterion; it generates the Borel
\(\sigma\)-field of \(B\).  Since the limiting factorial measures have
densities in every size coordinate, all Cartesian products of its
members are continuity sets.  The convergence in
\eqref{eq:factorial-Riemann-limit}, the bound
\eqref{eq:local-factorial-bound}, these continuity sets, and the
determinacy in \Cref{prop:AW-density} verify the hypotheses of
\Cref{lem:factorial-criterion}.  Therefore
\[
 \Xi_n^\star|_B\Longrightarrow
 \Xi_\Lambda^{\Ald}|_B.
\]
The limiting first factorial moment measure is absolutely continuous
in the size coordinate.  Hence the Aldous process has no point on
\(\{a,b\}\times\{0,\ldots,R\}\) almost surely.  The windows
\(B_{1/m,m,m}\) increase to \(E\), every compact subset of \(E\) is
contained in the interior of one such window, and the limit charges
none of their boundaries.  The standard restriction and exhaustion
criterion for point measures
\cite[Chapter~4]{Kallenberg2017} therefore proves
\eqref{eq:process-main} on \(\Mp(E)\).

\smallskip
\noindent\emph{Step 4: Absence of atypical component marks.}
Retain \(B=B_{a,b,R}\) and fix \(\varepsilon>0\).  In the deterministic
finite type model, let \(\tau_n=\bar t_n\) and let \(p_n\) be the
Perron law of the exactly critical empirical pair in
\eqref{eq:finite-critical-normalization}.  In the annealed compact type
BJR model, put \(\tau_n=t_n\) and \(p_n=p\).  Set
\(d_\star=d_{\TV}\) in the finite type case and
\(d_\star=d_{\BL}\) in the compact case.

For \(k\asymp n^{2/3}\) and \(0\le r\le R\), define the atypical event
under the corresponding component law by
\[
 \mathcal A_{n,k,r}(\varepsilon)
 =\{d_\star(L_k,p_n)>\varepsilon\}.
\]
The concentration estimate \eqref{eq:profile-main-tail} and the
parameter sequence clause of \Cref{lem:profile} give
\begin{equation}\label{eq:bad-profile-component-law}
 \cQ_{k,r,\tau_n}^{(n)}
 \bigl(\mathcal A_{n,k,r}(\varepsilon)\bigr)
 \le C_\varepsilon e^{-c_\varepsilon k},
\end{equation}
uniformly over this range; the superscript \(n\) is omitted in the
compact case.

Specialize \eqref{eq:deterministic-centered-log-integrand} and
\eqref{eq:BJR-log-profile-decomposition} to one component.  Let
\(\mathcal U_{n,k,r}\) denote \(\mathcal E_n\) in the finite type case
and \(\mathcal Z_n+\epsilon_n\) in the compact case.  Then
\eqref{eq:deterministic-centered-residual-L1} and
\eqref{eq:annealed-centered-residual-L1} give
\begin{equation}\label{eq:centered-residual-L1-common}
 \E_{\cQ}\bigl|e^{\mathcal U_{n,k,r}}-1\bigr|=o(1)
\end{equation}
uniformly over the same size and surplus range.  Here and below
\(\E_{\cQ}\) denotes expectation under the current component law.
The deterministic exponential factored out in these residual
representations has exponent
\begin{equation}\label{eq:bad-mark-centered-exponent-bound}
 O\left(\frac{|\tau_n-1|k^2}{n}+\frac{k^3}{n^2}\right)=O(1)
\end{equation}
on the fixed size window.

Let \(N_n^{\rm bad}(k,r;\varepsilon)\) be the number of components
\(C\) satisfying
\[
 |C|=k,\qquad \sur(C)=r,\qquad d_\star(L_C,p_n)>\varepsilon.
\]
The exact identities for a single component remain valid with
\(\1_{\mathcal A_{n,k,r}(\varepsilon)}\) inside the expectation under
the component law.  Since
\[
 \E_{\cQ}\!\left[
 e^{\mathcal U_{n,k,r}}\1_{\mathcal A_{n,k,r}(\varepsilon)}\right]
 \le \cQ\bigl(\mathcal A_{n,k,r}(\varepsilon)\bigr)
     +\E_{\cQ}|e^{\mathcal U_{n,k,r}}-1|,
\]
\eqref{eq:bad-profile-component-law}--\eqref{eq:centered-residual-L1-common}
show that the residual expectation with the atypical indicator is
\(o(1)\), uniformly.

The coefficient estimate used in part~\textup{(i)}, namely
\Cref{thm:Perron-Wright} in the compact case and
\eqref{eq:actual-sequence-PW} in
the finite type case, together with \eqref{eq:Stirling-fixed-size}, gives
\begin{equation}\label{eq:bad-mark-size-prefactor}
 \frac{n^k}{k!}\,n^{-(k+r-1)}
 \CW_{\tau_n}^{(n)}(k,r)\le Cn^{-2/3}
\end{equation}
uniformly over the same range, again with the superscript omitted in
the compact case.
It follows that
\[
 \sup_{\substack{k:\,\beta^{1/3}n^{-2/3}k\in[a,b]\\0\le r\le R}}
 \E N_n^{\rm bad}(k,r;\varepsilon)=o(n^{-2/3}).
\]
There are \(O(n^{2/3})\) admissible sizes and finitely many surpluses.
Summation and Markov's inequality therefore give
\begin{equation}\label{eq:no-bad-marks}
 \Prob\left\{\exists C:
 (\beta^{1/3}n^{-2/3}|C|,\sur(C))\in B,\
 d_\star(L_C,p_n)>\varepsilon\right\}\longrightarrow0.
\end{equation}

\smallskip
\noindent\emph{Step 5: Replacement by the limiting Perron law.}
In the deterministic finite type case,
\(p_n\to p\) by \Cref{lem:finite-parameter-stability}; in the annealed
compact type case, \(p_n=p\).  Hence, for all sufficiently large \(n\),
\[
 d_\star(L_C,p)>\varepsilon
 \quad\Longrightarrow\quad
 d_\star(L_C,p_n)>\varepsilon/2.
\]
Applying \eqref{eq:no-bad-marks} with \(\varepsilon/2\), and using that
every compact \(K_E\subset E\) is contained in some window
\(B_{a',b',R'}\), proves
\eqref{eq:uniform-mark-replacement-condition}.  Apply
\Cref{lem:mark-replacement} with
\(\mathsf M=\cP(S)\), \(d_{\mathsf M}=d_\star\),
\(\xi_n=\Xi_n^\star\), and
\(\widehat\xi_n=\widehat\Xi_n^\star\).  In the finite type case,
\(\cP(S)\) is compact under \(d_{\TV}\); in the compact type case it is
compact in the weak topology metrized by \(d_{\BL}\).  The lemma and
\eqref{eq:process-main} now prove \eqref{eq:marked-process-main} in the
corresponding vague topology.
\end{proof}

The following corollary extracts the form of the finite type estimates
needed for the fixed truncations in \Cref{sec:countable}.  Its critical
window parameter may vary over a compact interval, and the kernel is
normalized to be exactly critical for the current empirical type law.

\begin{corollary}[Uniform finite type estimates under empirical critical normalization]
\label{cor:uniform-finite-type-estimates}
Fix a critical finite type pair
\((S,\mu,\kappa)\) and deterministic type counts as in
\Cref{def:finite-model}.  Let
\(\nu^{(n)}=(n_a/n)_{a\in S}\) be its empirical type law, and define
\begin{equation}\label{eq:uniform-finite-critical-normalization}
 \sigma_n=\spr\{\kappa\diag(\nu^{(n)})\},
 \qquad
 \kappa^{(n)}=\frac{\kappa}{\sigma_n}.
\end{equation}
For \(\lambda'\in\R\), let \(G_n^{\mathrm{crit}}(\lambda')\) be the
deterministic typed graph with these type counts and edge probabilities
\[
 1-\exp\left\{-\frac{(1+\lambda'n^{-1/3})
                         \kappa^{(n)}(a,b)}{n}\right\},
 \qquad a,b\in S,
\]
and write \(\E_n^{\lambda'}\) for expectation under its law.  Let
\((p,\alpha,\beta)\) be the Perron quantities of \((\mu,\kappa)\).

Let \(I\subset(0,\infty)\) and \(J\subset\R\) be compact intervals.
\begin{enumerate}[label=\textup{(\roman*)}]
\item Fix \(q\ge1\) and
\(\boldsymbol r=(r_1,\ldots,r_q)\in\Z_{\ge0}^q\).  For
\(\boldsymbol x\in I^q\), put
\(\bk_n(\boldsymbol x)=(\lfloor x_i n^{2/3}\rfloor)_{i=1}^q\), and let
\(N_n^{\lambda'}(\bk,\boldsymbol r)\) be the ordered component count
in \eqref{eq:orderedcomponentcount} for
\(G_n^{\mathrm{crit}}(\lambda')\).  Then
\begin{equation}\label{eq:uniform-finite-factorial}
 \sup_{\substack{\boldsymbol x\in I^q\\\lambda'\in J}}
 \left|n^{2q/3}\E_n^{\lambda'}
 N_n^{\lambda'}(\bk_n(\boldsymbol x),\boldsymbol r)
 -m_{\alpha,\beta,\lambda'}^{(q)}
   ((x_i,r_i)_{i=1}^q)\right|\longrightarrow0.
\end{equation}

\item For every \(R\in\Z_{\ge0}\) and \(\varepsilon>0\),
\begin{equation}\label{eq:uniform-finite-bad-profile}
 \sup_{\lambda'\in J}\E_n^{\lambda'}
 \sum_{C\in\cC(G_n^{\mathrm{crit}}(\lambda'))}
 \1_{\{|C|/n^{2/3}\in I,\ \sur(C)\le R,\
          d_{\TV}(L_C,p)>\varepsilon\}}
 \longrightarrow0.
\end{equation}
\end{enumerate}
Both limits remain valid simultaneously for any fixed finite collection
of these parameter choices.
\end{corollary}

\begin{proof}
Put
\(\delta_n=\|\nu^{(n)}-\mu\|_1+
\|\kappa^{(n)}-\kappa\|_{\max}\).  By
\Cref{def:finite-model} and finite dimensional Perron perturbation,
\(\sigma_n=1+O(n^{-1})\) and \(\delta_n=O(n^{-1})\); by homogeneity,
\(\kappa^{(n)}\) is exactly critical for \(\nu^{(n)}\).
If \(a=\inf I\), \(b=\sup I\), and
\(\mathcal K_n=\{k:(a/2)n^{2/3}\le k\le2bn^{2/3}\}\), then
\(\min\mathcal K_n\to\infty\) and
\((\max\mathcal K_n)\delta_n=O(n^{-1/3})\to0\).
Thus \Cref{lem:finite-parameter-stability} applies.  If
\((p_n,\alpha_n,\beta_n)\) are the current Perron quantities, then
\(d_{\TV}(p_n,p)\to0\), \(\alpha_n\to\alpha\), and
\(\beta_n\to\beta\).

For \(\lambda'\in J\), put
\(\tau_n(\lambda')=1+\lambda'n^{-1/3}\).  Compactness of \(I,J\)
places \(\tau_n(\lambda')\) uniformly in the range required by
\eqref{eq:actual-sequence-PW} and
\Cref{prop:component-count-formula} for every
\(k\in\mathcal K_n\).  That estimate,
the component residual in
\Cref{prop:component-count-formula}\textup{(a)}, and the uniform
Stirling estimate \eqref{eq:Stirling-fixed-size} give
\[
 \sup_{\substack{\boldsymbol x\in I^q\\\lambda'\in J}}
 \left|n^{2q/3}\E_n^{\lambda'}
 N_n^{\lambda'}(\bk_n(\boldsymbol x),\boldsymbol r)
 -m_{\alpha_n,\beta_n,\lambda'}^{(q)}
   ((x_i,r_i)_{i=1}^q)\right|\longrightarrow0.
\]
Joint continuity of the density \eqref{eq:m-density}, uniformly on the
compact parameter set, and the preceding Perron convergence prove
\eqref{eq:uniform-finite-factorial}.

For the second assertion, fix \(\eta>0\) and let
\(N_{n,\lambda'}^{\rm bad}(k,r;\eta)\) denote the number of components
with size \(k\), surplus \(r\), and
\(d_{\TV}(L_C,p_n)>\eta\).
If \((k,r)\) is not admissible, this count is zero.  Otherwise, under
the component law \(\cQ_{k,r,\tau_n(\lambda')}^{(n)}\), write
\(\E_{\cQ}\) for expectation and put
\(\mathcal A_{n,k,r}(\eta)=\{d_{\TV}(L_k,p_n)>\eta\}\),
and let \(\mathcal E_{n,k,r}^{\lambda'}\) be the centered residual
\eqref{eq:deterministic-centered-log-integrand} specialized to one
component and \(t_n=\tau_n(\lambda')\).
Uniformly for \(k\in\mathcal K_n\), \(0\le r\le R\), and
\(\lambda'\in J\), \Cref{lem:profile} and
\eqref{eq:deterministic-centered-residual-L1} give
\(\cQ_{k,r,\tau_n(\lambda')}^{(n)}
(\mathcal A_{n,k,r}(\eta))\le C_\eta e^{-c_\eta k}\) and
\(\E_{\cQ}|e^{\mathcal E_{n,k,r}^{\lambda'}}-1|=o(1)\).
Insert \(\1_{\mathcal A_{n,k,r}(\eta)}\) in the single component
identity of \Cref{prop:component-count-formula}\textup{(a)}.  The
 coefficient estimate \eqref{eq:actual-sequence-PW}, Stirling's
formula, and
the bounded deterministic exponent, together with
\(e^x\1_A\le\1_A+|e^x-1|\), then give
\[
 \sup_{\substack{\lambda'\in J,\ k\in\mathcal K_n\\0\le r\le R}}
 \E_n^{\lambda'}N_{n,\lambda'}^{\rm bad}(k,r;\eta)
 \le Cn^{-2/3}\{C_\eta e^{-c_\eta(a/2)n^{2/3}}+o(1)\}
 =o(n^{-2/3}).
\]
For all large \(n\), \(d_{\TV}(L_C,p)>\varepsilon\) implies
\(d_{\TV}(L_C,p_n)>\varepsilon/2\).
There are \(O(n^{2/3})\) sizes in the window and \(R+1\) surplus
values.  Summing the preceding estimate with \(\eta=\varepsilon/2\) proves
\eqref{eq:uniform-finite-bad-profile}.  A maximum over finitely many
error terms proves the simultaneous assertion.
\end{proof}

\section{Countable type construction and local limits}
\label{sec:countable}

In this section, we prove
\Cref{thm:countable-realization,thm:countable-unrenormalized}.  The
finite type estimates are not uniform in the truncation level, so each
truncation must be fixed before the graph size tends to infinity.  We
first establish the finite approximation results, then assemble them
into a diagonal graph sequence and remove the exact critical
normalization.

\subsection{Finite truncations and integer type counts}

We need three preliminary lemmas.  By
\Cref{lem:countable-truncation}, finite truncations preserve the Perron
quantities; \Cref{lem:integer-approximation} replaces a finite type law
by integer vertex counts; and \Cref{lem:countablefixedstage} identifies
each fixed truncation with the empirically normalized finite type model
covered by \Cref{cor:uniform-finite-type-estimates}.

Let \((\mu,\kappa)\) satisfy \Cref{def:countable-kernel}.  For
\(m\ge1\), put
\[
 A_m=\{1,\ldots,m\},\qquad M_m=\mu(A_m),
\]
where \(A_m\) is the first truncation support and \(M_m\) is the mass
retained by that truncation.  Let \(\mu^{[m]}\) be \(\mu\) conditioned
on \(A_m\):
\begin{equation}\label{eq:conditional-truncation-law}
 \mu_a^{[m]}=\frac{\mu_a}{M_m}\quad(a\in A_m),
 \qquad
 \boldsymbol\mu^{[m]}=(\mu_a^{[m]})_{a\in A_m}.
\end{equation}
Let \(\kappa_{A_m}=(\kappa_{ab})_{a,b\in A_m}\), and define
\begin{equation}\label{eq:truncated-critical-kernel}
 \sigma_m=\spr\{\kappa_{A_m}\diag(\boldsymbol\mu^{[m]})\},
 \qquad
 \kappa^{[m]}=\frac{\kappa_{A_m}}{\sigma_m}.
\end{equation}
Here \(\sigma_m\) is the Perron eigenvalue before normalization.
Thus the normalized truncation is exactly critical.
Its Perron eigenfunction is normalized by
\(\sum_{a\le m}\mu_a^{[m]}u_a^{[m]}=1\).  We write
\(p^{[m]},\alpha^{[m]},\beta^{[m]}\) for the associated Perron quantities and
extend \(p^{[m]}\) by zero to \(\N\).

\begin{lemma}[Truncation of the Perron quantities]
\label{lem:countable-truncation}
As \(m\to\infty\),
\begin{equation}\label{eq:countable-Perron-convergence}
 \sigma_m\to1,
 \qquad d_{\TV}(p^{[m]},p)\to0,
 \qquad \alpha^{[m]}\to\alpha,
 \qquad \beta^{[m]}\to\beta.
\end{equation}
\end{lemma}

\begin{proof}
\smallskip
\noindent\emph{Step 1: The limiting Perron eigenpair.}
Let \(\ell^2(\N)\) be the complex Hilbert space of square summable
sequences; the positive vectors below belong to its real subspace.
Consider the self-adjoint operator
\(\mathsf B:\ell^2(\N)\to\ell^2(\N)\) with matrix
\begin{equation}\label{eq:countable-symmetrized-operator}
 \mathsf B_{ab}=\sqrt{\mu_a}\,\kappa_{ab}\sqrt{\mu_b}.
\end{equation}
It is Hilbert--Schmidt since
\[
 \|\mathsf B\|_{\rm HS}^2
 :=\sum_{a,b\ge1}\mathsf B_{ab}^2
 \le\kappa_+^2\sum_{a,b\ge1}\mu_a\mu_b
 =\kappa_+^2,
\]
and hence compact.  The unitary map
\begin{equation}\label{eq:countable-unitary-map}
 U:L^2(\mu)\longrightarrow\ell^2(\N),
 \qquad (Uf)_a=\sqrt{\mu_a}\,f(a),
\end{equation}
satisfies \(UKU^{-1}=\mathsf B\).  Thus \(K\) and \(\mathsf B\) have
the same spectrum, with multiplicities.

All entries in \eqref{eq:countable-symmetrized-operator} are
strictly positive.  Hence \(\mathsf B\) is positivity improving, and
the Perron--Frobenius--Jentzsch theorem implies that its spectral radius
\(1\) is a simple isolated eigenvalue; see
\cite[Theorem~V.6.6]{Schaefer1974}.
Normalize the corresponding eigenfunction \(u\) of \(K\) by
\(\sum_a\mu_au_a=1\).  The Perron equation and
\eqref{eq:countable-bounds} give
\begin{equation}\label{eq:countable-u-bounds}
 \kappa_-
 \le u_a=\sum_{b\ge1}\kappa_{ab}\mu_bu_b
 \le\kappa_+
 \qquad(a\ge1).
\end{equation}
The normalization makes \(p\) in \eqref{eq:Perron-quantities} a
probability measure, while \eqref{eq:countable-u-bounds} makes
\(\alpha\) and \(\beta\) finite and strictly positive.

\smallskip
\noindent\emph{Step 2: Spectral convergence of the truncations.}
Let \(P_m\) be the coordinate projection onto the first \(m\)
coordinates.  The symmetrized operator of the unrenormalized truncated
pair \((\mu^{[m]},\kappa_{A_m})\), extended by zero outside \(A_m\), is
\begin{equation}\label{eq:truncated-symmetrized-operator}
 \mathsf B^{[m]}=M_m^{-1}P_m\mathsf BP_m.
\end{equation}
The Hilbert--Schmidt tail tends to zero:
\[
 \|P_m\mathsf BP_m-\mathsf B\|_{\rm HS}^2
 =\sum_{\{(a,b):\,a>m\ \text{or}\ b>m\}}\mathsf B_{ab}^2
 \longrightarrow0.
\]
Since \(M_m\to1\),
\eqref{eq:truncated-symmetrized-operator} gives
\(\|\mathsf B^{[m]}-\mathsf B\|_{\rm HS}\to0\), hence also convergence
in operator norm.

Restricted to the first \(m\) coordinates, \(\mathsf B^{[m]}\) is the
symmetrization of
\(\kappa_{A_m}\diag(\boldsymbol\mu^{[m]})\).  It is therefore
self-adjoint, has the same spectrum as that matrix, and has Perron
eigenvalue \(\sigma_m\).  The spectral theorem and the
Perron--Frobenius theorem give
\(\|\mathsf B^{[m]}\|_{2\to2}=\sigma_m\); similarly,
\(\|\mathsf B\|_{2\to2}=1\).  Hence
\[
 |\sigma_m-1|
 =\bigl|\|\mathsf B^{[m]}\|_{2\to2}
          -\|\mathsf B\|_{2\to2}\bigr|
 \le\|\mathsf B^{[m]}-\mathsf B\|_{2\to2}\longrightarrow0.
\]
Let \(\Gamma\) be a circle enclosing \(1\) and no other point of
\(\sigma(\mathsf B)\).  The resolvent identity and a Neumann series
show that \(\Gamma\) lies in the resolvent set of
\(\mathsf B^{[m]}\) for all large \(m\), uniformly on \(\Gamma\).
Consequently the associated Riesz projections \(\Pi_m\) converge in
operator norm to the rank one spectral projection \(\Pi\) of
\(\mathsf B\) at \(1\).
For all large \(m\), \(\Pi_m\) also has rank one; because
\(\sigma_m\to1\), its range is precisely the Perron eigenspace of
\(\mathsf B^{[m]}\).  See
\cite[Chapter~IV, Sections~3.3--3.5]{Kato1995} for these spectral
perturbation facts.

\smallskip
\noindent\emph{Step 3: Convergence under the mean one normalization.}
If
\(g_a=\sqrt{\mu_a}u_a\), then \(g\) spans the Perron eigenspace of
\(\mathsf B\).  Similarly,
\(g_a^{[m]}=\sqrt{\mu_a^{[m]}}u_a^{[m]}\), extended by zero, spans the
Perron eigenspace of \(\mathsf B^{[m]}\); division of the kernel by
\(\sigma_m\) does not change this eigendirection.  Since the operators
are self-adjoint, \(\Pi_m\) and \(\Pi\) are orthogonal projections.
Let \(e\) be the positive unit vector spanning
\(\operatorname{ran}(\Pi)\).  For all
large \(m\), set \(e_m=\Pi_m e/\|\Pi_m e\|_2\).  Then \(e_m\) is the
nonnegative unit Perron vector of \(\mathsf B^{[m]}\), and
\(\|e_m-e\|_2\to0\).

Put
\[
 \boldsymbol s=(\sqrt{\mu_a})_{a\ge1},\qquad
 \boldsymbol s^{[m]}
 =(\sqrt{\mu_a^{[m]}}\1_{\{a\le m\}})_{a\ge1}.
\]
The inner products with these vectors encode the normalization in
\eqref{eq:Perron-eigenfunction}.  Since \(M_m\to1\),
\[
 \|\boldsymbol s^{[m]}-\boldsymbol s\|_2^2
 =2(1-\sqrt{M_m})\longrightarrow0.
\]
The Perron normalizations say that
\[
 g=\frac{e}{\langle \boldsymbol s,e\rangle},
 \qquad
 g^{[m]}=\frac{e_m}{\langle \boldsymbol s^{[m]},e_m\rangle}.
\]
The denominators converge to
\(\langle \boldsymbol s,e\rangle>0\), and therefore
\(\|g^{[m]}-g\|_2\to0\).  Since
\(\sqrt{\mu_a}u_a^{[m]}=\sqrt{M_m}\,g_a^{[m]}\) for \(a\le m\), this
gives the full \(L^2(\mu)\) convergence
\begin{equation}\label{eq:truncated-u-L2-convergence}
 \sum_{a\le m}\mu_a|u_a^{[m]}-u_a|^2
 +\sum_{a>m}\mu_au_a^2
 =\|\sqrt{M_m}\,g^{[m]}-g\|_2^2\longrightarrow0.
\end{equation}

\smallskip
\noindent\emph{Step 4: Convergence of the Perron laws and constants.}
The truncated Perron equation also gives, uniformly in \(a\le m\)
and all large \(m\),
\begin{equation}\label{eq:truncated-u-uniform-bounds}
 \frac{\kappa_-}{\sigma_m}
 \le u_a^{[m]}
 =\frac1{\sigma_m}\sum_{b\le m}
   \kappa_{ab}\mu_b^{[m]}u_b^{[m]}
 \le\frac{\kappa_+}{\sigma_m}.
\end{equation}
Together with \eqref{eq:countable-u-bounds} and \(\sigma_m\to1\), this
places all the eigenfunction values in one bounded interval.  Hence
\eqref{eq:truncated-u-L2-convergence} and Cauchy--Schwarz imply
\begin{align*}
 d_{\TV}(p^{[m]},p)
 &\le \frac1{M_m}\sum_{a\le m}\mu_a
       |u_a^{[m]}-u_a|
   +|M_m^{-1}-1|\sum_{a\le m}\mu_au_a
   +\sum_{a>m}\mu_au_a
 \longrightarrow0.
\end{align*}
For \(q\in\{2,3\}\), the same bounds give
\(|x^q-y^q|\le C|x-y|\) on that interval.  Consequently,
\begin{align*}
 &\left|\sum_{a\le m}\mu_a^{[m]}(u_a^{[m]})^q
              -\sum_{a\ge1}\mu_au_a^q\right| \\
 &\quad\le \frac1{M_m}\sum_{a\le m}\mu_a
                |(u_a^{[m]})^q-u_a^q|
       +|M_m^{-1}-1|\sum_{a\le m}\mu_au_a^q
       +\sum_{a>m}\mu_au_a^q
 \longrightarrow0.
\end{align*}
Taking \(q=2\) and \(q=3\) proves, respectively,
\(\alpha^{[m]}\to\alpha\) and \(\beta^{[m]}\to\beta\), and completes
\eqref{eq:countable-Perron-convergence}.
\end{proof}

\paragraph{Integer type counts at a fixed truncation.}

The next lemma turns a fixed truncated law into deterministic integer
type counts.

\begin{lemma}[Integer approximation of a finite law]
\label{lem:integer-approximation}
Let \(\pi=(\pi_1,\ldots,\pi_m)\) be a probability vector and let
\(n\ge1\).  There are nonnegative integers
\(n_1,\ldots,n_m\), summing to \(n\), such that, with
\(\nu_a=n_a/n\) and
\(\boldsymbol\nu=(\nu_1,\ldots,\nu_m)\),
\begin{equation}\label{eq:integer-approximation}
 \|\boldsymbol\nu-\boldsymbol\pi\|_1
 :=\sum_{a=1}^m|\nu_a-\pi_a|
 \le\frac{m}{n}.
\end{equation}
If \(n\min_a\pi_a>2m\), the counts can all be chosen positive.
\end{lemma}

\begin{proof}
Start with \(n_a=\lfloor n\pi_a\rfloor\), and add one to the
\(n-\sum_an_a\) coordinates having the largest fractional parts of
\(n\pi_a\).  Then \(\sum_an_a=n\) and
\(|n_a-n\pi_a|<1\), so the left side of
\eqref{eq:integer-approximation} is at most \(m/n\).  Under the final
hypothesis, every initial value \(\lfloor n\pi_a\rfloor\) is positive.
\end{proof}

Fix \(m\) for the moment.  Apply \Cref{lem:integer-approximation} to
\(\pi=\mu^{[m]}\) for each \(n\), and denote the resulting empirical law
by \(\nu^{(n,m)}\).  Define
\begin{equation}\label{eq:fixedstagespectralradius}
 \sigma_{n,m}
 =\spr\{\kappa_{A_m}\diag(\nu^{(n,m)})\},
 \qquad
 \bar\kappa_{n,m}=\frac{\kappa_{A_m}}{\sigma_{n,m}}.
\end{equation}
Thus \(\sigma_{n,m}\) is the Perron eigenvalue for this empirical law, and
\((A_m,\nu^{(n,m)},\bar\kappa_{n,m})\) is exactly critical whenever all
coordinates are positive.  For fixed \(m\), positivity holds for all
sufficiently large \(n\).

Let \(p_{n,m},\alpha_{n,m},\beta_{n,m}\) be the Perron law and constants
of this empirical critical model.  For \(\lambda'\in\R\) and \(n\)
sufficiently large that \(1+\lambda'n^{-1/3}>0\), let
\(G_{n,m}(\lambda')\) be the deterministic finite type graph with these
type counts and edge probabilities
\[
 1-\exp\left\{-\frac{(1+\lambda'n^{-1/3})
                   \bar\kappa_{n,m}(a,b)}{n}\right\},
 \qquad a,b\in A_m,
\]
and write \(\E_{n,m}^{\lambda'}\) for expectation under its law.

\begin{lemma}[A fixed truncation as a finite type model]
\label{lem:countablefixedstage}
For every fixed \(m\), the empirical critical model converges to the
normalized truncation as \(n\to\infty\):
\begin{align}
 &\|\nu^{(n,m)}-\mu^{[m]}\|_1
 +|\sigma_{n,m}-\sigma_m|
 +\max_{a,b\le m}|\bar\kappa_{n,m}(a,b)-\kappa^{[m]}_{ab}|\notag\\
 &\qquad
 +d_{\TV}(p_{n,m},p^{[m]})
 +|\alpha_{n,m}-\alpha^{[m]}|
 +|\beta_{n,m}-\beta^{[m]}|\longrightarrow0.
 \label{eq:fixedstageconvergence}
\end{align}
Moreover,
\begin{equation}\label{eq:fixedstage-as-finite-normalization}
 \spr\{\kappa^{[m]}\diag(\nu^{(n,m)})\}
 =\frac{\sigma_{n,m}}{\sigma_m},
 \qquad
 \frac{\kappa^{[m]}}
 {\spr\{\kappa^{[m]}\diag(\nu^{(n,m)})\}}
 =\bar\kappa_{n,m}.
\end{equation}
Consequently, for every fixed \(m\), all conclusions of
\Cref{cor:uniform-finite-type-estimates} hold for
\(G_{n,m}(\lambda')\), uniformly on each fixed compact
\(\lambda'\)-interval, with \((p,\alpha,\beta)\) replaced by
\((p^{[m]},\alpha^{[m]},\beta^{[m]})\).
\end{lemma}

\begin{proof}
The construction in \Cref{lem:integer-approximation} gives
\(\|\nu^{(n,m)}-\mu^{[m]}\|_1\le m/n\).  For a positive probability
vector \(q=(q_a)_{a\in A_m}\), consider the matrix
\(\kappa_{A_m}\diag(q)\).  At \(q=\mu^{[m]}\) this matrix is strictly
positive, so its
Perron root is simple.  In a neighborhood of \(\mu^{[m]}\), its Perron
root and right Perron vector, normalized by \(\sum_aq_au_a(q)=1\), depend
analytically on \(q\).  Division of the kernel by its Perron root does
not change this normalized eigenvector.  Consequently, the Perron law
and the constants in \eqref{eq:Perron-quantities} are locally Lipschitz
functions of \(q\), and
\[
 |\sigma_{n,m}-\sigma_m|
 +d_{\TV}(p_{n,m},p^{[m]})
 +|\alpha_{n,m}-\alpha^{[m]}|
 +|\beta_{n,m}-\beta^{[m]}|=O_m(n^{-1}).
\]
Since \(\sigma_m>0\), this estimate also gives
\(\sigma_{n,m}^{-1}-\sigma_m^{-1}=O_m(n^{-1})\), and hence
\[
 \max_{a,b\le m}|\bar\kappa_{n,m}(a,b)-\kappa^{[m]}_{ab}|
 \le \kappa_+|\sigma_{n,m}^{-1}-\sigma_m^{-1}|
 =O_m(n^{-1}).
\]
These three estimates prove \eqref{eq:fixedstageconvergence}.

Finally, \(\kappa^{[m]}=\kappa_{A_m}/\sigma_m\), so homogeneity of the
spectral radius gives the first identity in
\eqref{eq:fixedstage-as-finite-normalization}.  Dividing
\(\kappa^{[m]}\) by this spectral radius gives the second.
The pair \((\mu^{[m]},\kappa^{[m]})\) satisfies
\Cref{def:finite-model}, and the integer construction gives
\(n_a=n\mu_a^{[m]}+O_m(1)\).  Thus
\eqref{eq:fixedstage-as-finite-normalization} identifies
\(G_{n,m}(\lambda')\) with the empirically normalized graph in
\Cref{cor:uniform-finite-type-estimates}, proving the final assertion.
\end{proof}

\subsection{Diagonal construction and countable type limits}

To prove \Cref{thm:countable-realization}, we choose truncation levels
approaching the countable kernel, impose four finite stage requirements,
and then select a stage \(j(n)\) for each graph size.  The resulting
triangular array has growing support, integer type counts, the required
local factorial limits, and marked convergence to the Perron law.

\begin{proof}[Proof of \Cref{thm:countable-realization}]
\smallskip
\noindent\emph{Step 1: Truncation levels and finite stage requirements.}
Let \(\mathscr D\) be the countable collection of tuples
\begin{equation}\label{eq:diagonal-requirements}
 \mathfrak d=(q,\boldsymbol r,a,b,R,L),
\end{equation}
where \(q\ge1\),
\(\boldsymbol r=(r_1,\ldots,r_q)\in\Z_{\ge0}^q\),
\(a,b\in\mathbb Q\) satisfy \(0<a<b\),
\(R\in\Z_{\ge0}\), and \(L\in\N\).  Here \([a,b]\) is the window for
\(k/n^{2/3}\), before the final multiplication by \(\beta^{1/3}\);
\(R\) is the surplus bound, and \([-L,L]\) is the range of critical
window parameters.  Enumerate
\(\mathscr D=\{\mathfrak d_1,\mathfrak d_2,\ldots\}\), and write
\[
 \mathfrak d_\ell
 =(q_\ell,\boldsymbol r_\ell,a_\ell,b_\ell,R_\ell,L_\ell)
\]
for the coordinates of its \(\ell\)-th member.

Put \(s_0=0\).  Choose integers \(s_j\) recursively so that
\begin{equation}\label{eq:diagonal-truncation-choice}
 s_j>\max\{s_{j-1},j\},\qquad
 \mu(A_{s_j}^c)<\frac1j,
\end{equation}
and also require
\begin{equation}\label{eq:diagonal-Perron-choice}
 |\sigma_{s_j}-1|
 +d_{\TV}(p^{[s_j]},p)
 +|\alpha^{[s_j]}-\alpha|
 +|\beta^{[s_j]}-\beta|<\frac1j.
\end{equation}
The first condition in \eqref{eq:diagonal-truncation-choice} makes the
truncation supports increase and exhaust \(\N\), while the second makes
the discarded type mass tend to zero.  Such a choice is possible by
\Cref{lem:countable-truncation}.  At stage
\(j\), apply \Cref{lem:integer-approximation} with \(m=s_j\).  To shorten
notation, write
\begin{equation}\label{eq:stage-notation}
 \nu^{(n,j)}=\nu^{(n,s_j)},
 \qquad \sigma_{n,j}=\sigma_{n,s_j},
 \qquad \bar\kappa_{n,j}=\kappa_{A_{s_j}}/\sigma_{n,j}.
\end{equation}
These are the empirical type law, its unrenormalized Perron eigenvalue,
and the exactly critical kernel at stage \(j\).
Denote the Perron quantities of the empirical critical triple
\((A_{s_j},\nu^{(n,j)},\bar\kappa_{n,j})\) by
\[
 (p^{(n,j)},\alpha_{n,j},\beta_{n,j}).
\]

Condition~\ref{item:diag-Perron} controls the empirical Perron
quantities, Condition~\ref{item:diag-factorial} the local factorial
densities, Condition~\ref{item:diag-marks} the component profiles, and
Condition~\ref{item:diag-integer-counts} both the integer approximation
error and the minimum retained type mass on the scale required below.  By
\Cref{lem:integer-approximation,lem:countablefixedstage,cor:uniform-finite-type-estimates},
choose \(N_j<\infty\) so
that, for every \(n\ge N_j\), the following four requirements hold
simultaneously:
\begin{enumerate}[label=\textup{(D\arabic*)}]
\item\label{item:diag-Perron}
The current Perron quantities satisfy
\begin{align}
 |\sigma_{n,j}-\sigma_{s_j}|
 +d_{\TV}(p^{(n,j)},p^{[s_j]})
 +|\alpha_{n,j}-\alpha^{[s_j]}|
 +|\beta_{n,j}-\beta^{[s_j]}|<\frac1j.
 \label{eq:diagonal-empirical-Perron-bound}
\end{align}

\item\label{item:diag-factorial}
For every requirement \(\mathfrak d_\ell\) with \(\ell\le j\), every
\(\lambda'\in[-L_\ell,L_\ell]\), and every
\(x_1,\ldots,x_{q_\ell}\in[a_\ell,b_\ell]\), the empirical critical
graph satisfies the explicit local factorial estimate
\begin{align}
 \bigg|n^{2q_\ell/3}\E_{n,s_j}^{\lambda'} N_n
 \bigl((\lfloor x_i n^{2/3}\rfloor)_{i=1}^{q_\ell},
          \boldsymbol r_\ell\bigr)
 -m_{\alpha^{[s_j]},\beta^{[s_j]},\lambda'}^{(q_\ell)}
   ((x_i,r_{\ell,i})_{i=1}^{q_\ell})\bigg|
 \le\frac1j.
 \label{eq:diagonal-factorial-bound}
\end{align}
Here the graph in the expectation uses
\(t_n=1+\lambda'n^{-1/3}\) and kernel \(\bar\kappa_{n,j}\).

\item\label{item:diag-marks}
For every \(\ell\le j\) and every
\(\lambda'\in[-L_\ell,L_\ell]\), the expected number of components
with
\[
 |C|/n^{2/3}\in[a_\ell,b_\ell],
 \qquad \sur(C)\le R_\ell,
 \qquad d_{\TV}(L_C,p^{[s_j]})>1/j
\]
is at most \(1/j\).

\item\label{item:diag-integer-counts}
The approximation error and minimum mass bounds
\begin{equation}\label{eq:diagonal-integer-count-bounds}
 \frac{2s_j}{n}\le\frac1j,
 \qquad
 n^{1/3}\min_{a\le s_j}\nu_a^{(n,j)}\ge j
\end{equation}
hold, and every resulting count is positive.
\end{enumerate}
For \ref{item:diag-marks}, first fix the tolerance \(1/j\), then
enlarge \(N_j\) until the Perron law of the empirical critical model
lies in a smaller ball
around \(p^{[s_j]}\).  The empirical profile estimate then gives the
stated bound.  For \ref{item:diag-integer-counts}, every coordinate of the
fixed vector \(\mu^{[s_j]}\) is positive, so
\(\min_a\nu_a^{(n,j)}\) is bounded below by a positive constant that
may depend on the stage.  Choose the thresholds recursively so that
\begin{equation}\label{eq:strict-thresholds}
 N_1<N_2<\cdots.
\end{equation}

\smallskip
\noindent\emph{Step 2: The triangular array and its Perron quantities.}
For \(n\ge N_1\), define
\begin{equation}\label{eq:diagonal-index}
 j(n)=\max\{j:N_j\le n\}.
\end{equation}
Thus \(j(n)\) is the last stage certified at sample size \(n\).  The
maximum is finite because \(N_j\uparrow\infty\).  Set
\begin{equation}\label{eq:diagonal-model-definition}
 S_n=A_{s_{j(n)}},\qquad
 \nu^{(n)}=\nu^{(n,j(n))},\qquad
 \sigma_n=\sigma_{n,j(n)},\qquad
 n_a=n\nu_a^{(n)}.
\end{equation}
These are the support, empirical type law, spectral normalization, and
integer type counts of the final triangular array.
For the finitely many \(n<N_1\), put \(S_n=A_1\), assign every vertex
type \(1\), and put \(\sigma_n=\kappa_{11}\).  These finitely many
choices have no effect on any limit.

We first verify \eqref{eq:countable-realization-convergence}.  Since
\(s_j\ge j\) and \(j(n)\to\infty\), the supports \(S_n\) increase to
\(\N\).  Extend \(\nu^{(n)}\) by zero outside \(S_n\) and put
\(j=j(n)\).  Then
\begin{align}
 d_{\TV}(\nu^{(n)},\mu)
 &\le \|\nu^{(n,j)}-\mu^{[s_j]}\|_1
      +d_{\TV}(\mu^{[s_j]},\mu)\notag\\
 &\le \frac{2s_j}{n}+2\mu(A_{s_j}^c)
 \le\frac3j\longrightarrow0.
 \label{eq:diagonal-TV-proof}
\end{align}
Here
\(d_{\TV}(\mu^{[s_j]},\mu)=2\mu(A_{s_j}^c)\), obtained by summing once
over \(A_{s_j}\) and once over its complement.  The second inequality
in \eqref{eq:diagonal-integer-count-bounds} gives
\[
 n^{1/3}\min_{a\in S_n}\nu_a^{(n)}\ge j(n)\longrightarrow\infty.
\]
This also proves positivity of every type count and completes
\eqref{eq:countable-realization-convergence}.  Moreover,
\ref{item:diag-Perron}, \eqref{eq:diagonal-empirical-Perron-bound}, and
\eqref{eq:diagonal-Perron-choice} imply that the empirical Perron
quantities converge to \((p,\alpha,\beta)\), with total variation
convergence of the Perron law.  The same bounds give \(\sigma_n\to1\),
without a claimed rate.

\smallskip
\noindent\emph{Step 3: Fixed size factorial asymptotics.}
Fix \(q\ge1\), a surplus vector
\(\boldsymbol r=(r_1,\ldots,r_q)\), a real \(\lambda\), and a compact
interval \([a,b]\subset(0,\infty)\).  Choose rational
\(0<a'<a<b<b'\), an integer \(L>|\lambda|\), and an integer
\(R\ge\max_i r_i\).  The tuple
\((q,\boldsymbol r,a',b',R,L)\) is some requirement
\(\mathfrak d_\ell\).  Once \(j(n)\ge\ell\),
\eqref{eq:diagonal-factorial-bound} applies uniformly on \([a,b]^q\).
By \eqref{eq:diagonal-Perron-choice},
\[
 (\alpha^{[s_{j(n)}]},\beta^{[s_{j(n)}]})
 \longrightarrow(\alpha,\beta).
\]
The function
\(m_{\alpha,\beta,\lambda}^{(q)}\) defined in \eqref{eq:m-density} is
jointly continuous in
\((\alpha,\beta,\lambda,x_1,\ldots,x_q)\) when the sizes stay in a
compact subset of \((0,\infty)\) and \(\beta>0\).  Hence the density term in
\eqref{eq:diagonal-factorial-bound} converges uniformly to
 \(m_{\alpha,\beta,\lambda}^{(q)}\), proving
 \eqref{eq:countable-factorial-asymptotic}.

\smallskip
\noindent\emph{Step 4: The unmarked point process.}
Fix a local window \(B=B_{c,d,R}\) in the scaled
state space \(E\), and let
\(\xi_{n,B}=\Xi_n^\star|_B\).  Thus \(\xi_{n,B}\) retains exactly the
components whose scaled size and surplus lie in \(B\).  The exact
factorial measure identity \eqref{eq:factorial-Riemann-sum} remains
valid.  Substituting \eqref{eq:countable-factorial-asymptotic} for the
fixed size estimate used there gives the limit in
\eqref{eq:factorial-Riemann-limit}, with \(a,b\) replaced by \(c,d\).
Applying the same argument at order \(q+1\) with the constant test
function gives
\begin{equation}\label{eq:countable-local-factorial-bound}
 \sup_n\E(\xi_{n,B}(B))_{q+1}<\infty.
\end{equation}
The finitely many initial \(n\)'s contribute only a finite maximum.

The change of variables in \eqref{eq:F-rescale} and the coordinatewise
calculation following that equation identify the limiting factorial
moment measures with those of
\(\Xi_\Lambda^{\Ald}|_B\), where
\(\Lambda=\alpha\beta^{-2/3}\lambda\), by
\Cref{prop:AW-density}.  On the countable algebra generated by the sets
in \eqref{eq:local-generating-ring}, all Cartesian products are
continuity sets because the limiting measures have densities in their
size coordinates.  Consequently,
\eqref{eq:factorial-Riemann-limit},
\eqref{eq:countable-local-factorial-bound}, and
\Cref{lem:factorial-criterion} give
\[
 \Xi_n^\star|_B\Longrightarrow\Xi_\Lambda^{\Ald}|_B.
\]
The limiting first factorial moment measure has a density in the size
coordinate, so every fixed window has almost surely zero boundary mass.
The restriction and exhaustion argument using the windows
\(B_{1/m,m,m}\), carried out in Step~3 of the proof of
\Cref{thm:local-main}\textup{(ii)}, therefore gives
\eqref{eq:countable-process-main}.

\smallskip
\noindent\emph{Step 5: The marked point process.}
It remains to prove \eqref{eq:countable-marked-process-main}.  We use
Condition~\ref{item:diag-marks} and \Cref{lem:mark-replacement}.  Fix a
compact interval \([a,b]\subset(0,\infty)\) for \(|C|/n^{2/3}\), a
surplus bound \(R\), and \(\varepsilon>0\).  Choose rational numbers
\(0<a'<a<b<b'\) and an integer \(L>|\lambda|\).  The tuple
\((1,(0),a',b',R,L)\) is a diagonal requirement; denote its index by
\(\ell\).  Put
\(j=j(n)\).  For all sufficiently large \(n\), one has \(j\ge\ell\)
and, by \eqref{eq:diagonal-Perron-choice},
\[
 d_{\TV}(p^{[s_j]},p)<1/j<\varepsilon/2.
\]
Hence
\[
 d_{\TV}(L_C,p)>\varepsilon
 \quad\Longrightarrow\quad
 d_{\TV}(L_C,p^{[s_j]})>1/j.
\]
By \ref{item:diag-marks}, the expected number of such components is
at most \(1/j\).  Markov's inequality gives
\begin{equation}\label{eq:countable-no-bad-marks}
 \Prob\left\{\exists C:\ |C|/n^{2/3}\in[a,b],\
 \sur(C)\le R,\ d_{\TV}(L_C,p)>\varepsilon\right\}
 \longrightarrow0.
\end{equation}

Equation \eqref{eq:countable-no-bad-marks} gives
\eqref{eq:uniform-mark-replacement-condition} with
\(\mathsf M=\cP(\N)\), \(d_{\mathsf M}=d_{\TV}\),
\(\xi_n=\Xi_n^\star\), and
\(\widehat\xi_n=\widehat\Xi_n^\star\): every compact \(K_E\subset E\)
is contained in some \(B_{c,d,R}\).  Apply
\eqref{eq:countable-no-bad-marks} to the interval
\([c\beta^{-1/3},d\beta^{-1/3}]\) for \(|C|/n^{2/3}\); it becomes
\([c,d]\) after multiplication by \(\beta^{1/3}\) in \(\Xi_n^\star\).
Applying
\Cref{lem:mark-replacement} and using
\eqref{eq:countable-process-main} proves
\eqref{eq:countable-marked-process-main} in the localized topology
defined in \Cref{subsec:topology}.
\end{proof}

\paragraph{Removing spectral renormalization.}
At scale \(n^{-1/3}\), the displacement of \(\sigma_n\) from one is
a shift of the critical window parameter.  We now make this relation
exact and prove \Cref{thm:countable-unrenormalized}.

\begin{proof}[Proof of \Cref{thm:countable-unrenormalized}]
Put
\begin{equation}\label{eq:effective-unrenormalized-parameter}
 \lambda_n^\star
 =n^{1/3}\{\sigma_n(1+\lambda n^{-1/3})-1\}
 =n^{1/3}(\sigma_n-1)+\lambda\sigma_n.
\end{equation}
Since
\((1+\lambda_n^\star n^{-1/3})/\sigma_n
=1+\lambda n^{-1/3}\), the graph \(\widetilde G_n\) has the same law as
the normalized graph in \eqref{eq:countable-edge} with \(\lambda\)
replaced by \(\lambda_n^\star\).

Condition~\eqref{eq:countablespectraldrift} implies that
\(\sigma_n\to1\) and \(\lambda_n^\star\to\lambda+\vartheta\), so
\(\lambda_n^\star\) eventually lies in some fixed interval
\([-L,L]\).

Fix the parameters in \eqref{eq:countable-factorial-asymptotic}, and
choose a diagonal requirement covering them and \([-L,L]\).  With
\(j=j(n)\), Condition~\ref{item:diag-factorial} applies at
\(\lambda'=\lambda_n^\star\) for all sufficiently large \(n\).  By
\eqref{eq:diagonal-Perron-choice}, the convergence of
\(\lambda_n^\star\), and the joint continuity in \eqref{eq:m-density},
the stage density in \eqref{eq:diagonal-factorial-bound} converges
uniformly to
\(m_{\alpha,\beta,\lambda+\vartheta}^{(q)}\) on the chosen size interval.
Together with \eqref{eq:diagonal-factorial-bound}, this proves
\eqref{eq:countable-factorial-asymptotic} with \(\lambda+\vartheta\)
in place of \(\lambda\).  The argument in Step~4 of the proof of
\Cref{thm:countable-realization} uses only these local factorial limits
and the corresponding bounds of order \(q+1\).  Both follow from
Condition~\ref{item:diag-factorial} uniformly on \([-L,L]\); hence the
same argument gives the unmarked process limit with Aldous parameter
\(\Lambda^\star=\alpha\beta^{-2/3}(\lambda+\vartheta)\).

For the marked limit, fix \(B_{c,d,R}\subset E\) and \(\varepsilon>0\),
and choose a diagonal requirement covering the unscaled size interval
\([c\beta^{-1/3},d\beta^{-1/3}]\), the surplus bound \(R\), and
\([-L,L]\).  For all sufficiently large \(n\),
Condition~\ref{item:diag-marks} applies at \(\lambda_n^\star\), while
\eqref{eq:diagonal-Perron-choice} gives
\(d_{\TV}(p^{[s_j]},p)<1/j<\varepsilon/2\).  Consequently,
\(d_{\TV}(L_C,p)>\varepsilon\) implies
\(d_{\TV}(L_C,p^{[s_j]})>1/j\).
The expected number of such components in the window therefore tends
to zero by Condition~\ref{item:diag-marks}.  Markov's inequality and
\Cref{lem:mark-replacement} give the marked process limit with the same
parameter \(\Lambda^\star\).  Finally, if
\(\sigma_n=1+o(n^{-1/3})\), then \(\vartheta=0\), yielding the original
parameter \(\lambda\).
\end{proof}

\appendix
\section{Holomorphic and transfer tools}

The analytic arguments in the main text require uniformity in both the
complex singular variable and the Banach space parameter.  This
appendix collects the four abstract tools used for that purpose:
nested Cauchy estimates, uniform implicit continuation, a parameterized
holomorphic fold, and uniform Banach space valued singularity transfer.
\begin{lemma}[Cauchy estimates on nested product neighborhoods]
\label{lem:nested-product-Cauchy}
Let \(\mathscr P,\mathscr Y\) be complex Banach spaces, and suppose that
\[
 F:\{|X|<r_+\}\times B_{\mathscr P}(0,\varepsilon_+)
 \longrightarrow\mathscr Y
\]
is jointly holomorphic and uniformly bounded. Then, for every \(r<r_+\), \(\varepsilon<\varepsilon_+\), and fixed
\(j,m\ge0\),
\[
 \sup_{\substack{|X|\le r\\ \|p\|<\varepsilon}}
 \bigl\|\partial_X^jD_p^mF(X,p)\bigr\|_
 {\mathcal L^m(\mathscr P;\mathscr Y)}
 <\infty,
\]
with the usual interpretation when \(m=0\).
The same conclusion holds uniformly for any family of such maps that
is uniformly bounded on the displayed larger product neighborhood.

\end{lemma}

\begin{proof}
This is the standard Cauchy estimate for holomorphic maps between
complex Banach spaces; see, for example, \cite{Mujica1986}.
Let \(M\) be the common bound on the larger product neighborhood, and set
\[
 \eta_X:=\frac{r_+-r}{2},
 \qquad
 \eta_p:=\frac{\varepsilon_+-\varepsilon}{2}.
\]
For \(|X|\le r\), \(\|p\|<\varepsilon\), and
\(\|h\|_{\mathscr P}\le1\), successive applications of the
one-variable Cauchy formula give
\[
 \partial_X^jD_p^mF(X,p)[h,\ldots,h]
 =
 \frac{j!m!}{(2\pi i)^2}
 \int_{|\xi|=\eta_X}\int_{|\zeta|=\eta_p}
 \frac{F(X+\xi,p+\zeta h)}
      {\xi^{j+1}\zeta^{m+1}}
 \,\dd\zeta\,\dd\xi ,
\]
with the evident one-variable interpretation when \(m=0\).
Hence
\[
 \bigl\|
 \partial_X^jD_p^mF(X,p)[h,\ldots,h]
 \bigr\|
 \le
 \frac{j!m!M}{\eta_X^j\eta_p^m}.
\]
Since \(D_p^mF(X,p)\) is symmetric, the complex polarization
identity upgrades this diagonal estimate to the asserted
operator-norm bound, with a constant depending only on \(m\).
\end{proof}

\begin{lemma}[Uniform holomorphic implicit perturbation]
\label{lem:uniform-holomorphic-implicit}
Let $\mathscr X$ and $\mathscr P$ be complex Banach spaces, let
$\mathcal O\subset\mathbb C$ be open, and let
$\Omega\Subset\mathcal O$ be open.  Suppose that
$U_0:\mathcal O\to\mathscr X$ is holomorphic and that
$\mathcal V\subset\mathbb C\times\mathscr X\times\mathscr P$ is an
open neighborhood of
\[
 \Gamma_0:=\{(w,U_0(w),0):w\in\overline\Omega\}.
\]
Let $\mathcal F:\mathcal V\to\mathscr X$ be holomorphic, and
assume that, for every $w\in\overline\Omega$,
\[
 \mathcal F(w,U_0(w),0)=0,
 \qquad
 L_0(w):=D_U\mathcal F(w,U_0(w),0)
 \quad\hbox{is invertible}.
\]
Then there are \(r,\varepsilon>0\), \(C<\infty\), and a jointly
holomorphic map
\[
 U:\Omega\times B_{\mathscr P}(0,\varepsilon)\longrightarrow\mathscr X
\]
such that
\[
 \mathcal F(w,U(w,p),p)=0,
 \qquad
 \|U(w,p)-U_0(w)\|_{\mathscr X}<r.
\]
For each $(w,p)$, this is the unique zero in
$\overline B_{\mathscr X}(U_0(w),r)$.  Moreover,
\begin{equation}\label{eq:uniform-implicit-bounds}
 \sup_{\substack{w\in\Omega\\ \|p\|_{\mathscr P}<\varepsilon}}
 \left\{
  \|U(w,p)\|_{\mathscr X}
  +\|D_U\mathcal F(w,U(w,p),p)^{-1}\|_{\mathcal L(\mathscr X)}
 \right\}
 \le C.
\end{equation}

\end{lemma}

\begin{proof}
The graph $\Gamma_0$ is compact, and continuity of inversion gives
\[
 M:=\sup_{w\in\overline\Omega}\|L_0(w)^{-1}\|<\infty.
\]
By compactness of \(\Gamma_0\), openness of \(\mathcal V\), and
continuity of \(\mathcal F\) and \(D_U\mathcal F\), there is a positive
uniform margin around the base graph; hence one may first choose
\(r>0 \) and then \(\varepsilon>0\) so that the resulting tube 
\[
 \left\{(w,V,p):w\in\overline\Omega,
  \ \|V-U_0(w)\|\le r,
  \ \|p\|<\varepsilon\right\}
\]
lies in $\mathcal V$ and, throughout this tube,
\begin{align}
 \bigl\|L_0(w)^{-1}
  \{D_U\mathcal F(w,V,p)-L_0(w)\}\bigr\|&\le\frac12,
 \label{eq:uniform-implicit-contraction}\\
 \|L_0(w)^{-1}\mathcal F(w,U_0(w),p)\|&\le\frac r4.
 \label{eq:uniform-implicit-displacement}
\end{align}

For fixed $(w,p)$ put
\[
 \Psi_{w,p}(V):=V-L_0(w)^{-1}\mathcal F(w,V,p).
\]
Equations \eqref{eq:uniform-implicit-contraction}--
\eqref{eq:uniform-implicit-displacement} show that $\Psi_{w,p}$ is a
contraction of $\overline B(U_0(w),r)$ strictly into itself.  Its unique
fixed point is the asserted zero $U(w,p)$, and in fact
$\|U(w,p)-U_0(w)\|\le r/2$.  Starting the Picard iteration at
$U_0(w)$ gives holomorphic iterates whose successive differences
converge geometrically, uniformly on the product.  The Banach space valued
Weierstrass theorem therefore makes $U$ jointly holomorphic, and
uniqueness gives $U(w,0)=U_0(w)$.

At the solution,
\[
 D_U\mathcal F
 =L_0(w)\left[I+L_0(w)^{-1}
   \bigl\{D_U\mathcal F-L_0(w)\bigr\}\right].
\]
The Neumann series and \eqref{eq:uniform-implicit-contraction} bound its
inverse by $2M$, while
$\|U(w,p)\|\le\sup_{\overline\Omega}\|U_0\|+r$.  This proves
\eqref{eq:uniform-implicit-bounds}.  
\end{proof}

\begin{lemma}[Parameterized holomorphic fold]
\label{lem:holomorphic-fold}
Let $\mathscr X$ and $\mathscr P$ be complex Banach spaces, and let
\[
 \mathfrak F(s,p,V)
\]
be $\mathscr X$-valued and holomorphic near $(0,0,0)$.  Suppose that
$\mathfrak F(0,p,0)=0$ for all sufficiently small $p$, and put
$L_p=D_V\mathfrak F(0,p,0)$.  Assume that there are holomorphic maps
$p\mapsto u_p\in\mathscr X$ and $p\mapsto\ell_p\in\mathscr X^*$ such that
\[
 \ell_p(u_p)=1,
 \qquad \ker L_p=\mathbb C u_p,
 \qquad \operatorname{ran}L_p=\ker\ell_p.
\]
Define
\[
 A(p):=\ell_p\bigl(\partial_s\mathfrak F(0,p,0)\bigr),
 \qquad
 B(p):=\ell_p\bigl(D_V^2\mathfrak F(0,p,0)[u_p,u_p]\bigr),
\]
and suppose that $A(0)B(0)\ne0$.  For either choice
$\gamma_0^2=-2A(0)/B(0)$, there are $r,\varepsilon>0$, a holomorphic
square root $\gamma(p)$ of $-2A(p)/B(p)$ with $\gamma(0)=\gamma_0$, and
holomorphic maps
\[
 R_\pm:\{|X|<r\}\times B_{\mathscr P}(0,\varepsilon)\longrightarrow
 \mathscr X
\]
such that all sufficiently small solutions of
$\mathfrak F(X^2,p,V)=0$ are precisely
\begin{equation}\label{eq:abstract-fold-branches}
 V_\pm(X,p)=\pm\gamma(p)u_pX+X^2R_\pm(X,p).
\end{equation}
  If these objects are
complexifications of real objects and commute with conjugation, and if
$\gamma_0$ is real, then both branches are real for real $(X,p)$.

Moreover, after decreasing \(r\) and \(\varepsilon\) if necessary,
\(
 D_V\mathfrak F\bigl(X^2,p,V_\pm(X,p)\bigr)
\)
is invertible whenever
\[
 0<|X|<r,\qquad \|p\|<\varepsilon.
\]
\end{lemma}

\begin{proof}
Put $P_pV=\ell_p(V)u_p$, $Q_p=I-P_p$, and
$\mathscr X_0=\ker\ell_0$.  The operator
\[
 S_p:=P_pP_0+Q_pQ_0
\]
is holomorphic and satisfies \(S_0=I\); after
shrinking the parameter neighborhood, it is invertible.  Moreover, it
identifies the two summands of the decomposition at \(p=0\) with their
counterparts at \(p\):
\[
 S_p(\mathbb Cu_0)=\mathbb Cu_p,
 \qquad
 S_p(\ker\ell_0)=\ker\ell_p.
\]  Hence the restriction
\[
J_p:=S_p|_{\mathscr X_0}=Q_p|_{\mathscr X_0}:
 \mathscr X_0\longrightarrow\ker\ell_p
\]
is a family of isomorphisms.  More precisely,
\[
 J_p^{-1}Q_p=S_p^{-1}Q_p:\mathscr X\longrightarrow\mathscr X_0
\]
is holomorphic, so all maps below take values in fixed Banach spaces.
Write
$V=\xi u_p+J_py$ and consider the complementary equation
\[
 \Psi(s,p,\xi,y)
 :=J_p^{-1}Q_p\mathfrak F(s,p,\xi u_p+J_py)=0.
\]
The restriction of $L_p$ to $\ker\ell_p$ is a bounded isomorphism onto
$\ker\ell_p$: injectivity follows from the kernel hypothesis, and
surjectivity follows by subtracting the $u_p$-component of any
preimage.  The analytic implicit function theorem therefore gives a
holomorphic solution $y=Y(s,p,\xi)$ on one product neighborhood.  Its
uniqueness and $L_pu_p=0$ give
\begin{equation}\label{eq:abstract-complementary-vanishing}
 Y(0,p,0)=0,
 \qquad \partial_\xi Y(0,p,0)=0.
\end{equation}

The remaining scalar equation is
\[
 G(s,p,\xi)
 :=\ell_p\mathfrak F\bigl(s,p,\xi u_p+J_pY(s,p,\xi)\bigr)=0.
\]
Since $G(0,p,0)=\partial_\xi G(0,p,0)=0$, analytic divisibility gives
\begin{equation}\label{eq:abstract-reduced-fold}
 G(s,p,\xi)=sG_1(s,p,\xi)+\xi^2G_2(p,\xi),
\end{equation}
where
\[
 G_1(0,p,0)=A(p),
 \qquad G_2(p,0)=\frac12B(p).
\]
Indeed, one may obtain $G_1$ by integrating $\partial_sG$ along the
segment in $s$ and $G_2$ by applying the second order integral remainder
to $G(0,p,\xi)$.  The terms arising through $Y$ disappear from these two
derivatives because $\ell_pL_p=0$ and
\eqref{eq:abstract-complementary-vanishing} holds.

Set $s=X^2$, $\xi=Xv$, and
\[
 H(X,p,v):=G_1(X^2,p,Xv)+v^2G_2(p,Xv).
\]
Then
\[
 H(0,p,v)=A(p)+\frac12B(p)v^2.
\]
After shrinking the parameter ball, $A$ and $B$ stay nonzero and the
chosen square root extends holomorphically.  The two roots
$v=\pm\gamma(p)$ are simple, so the scalar analytic implicit function
theorem gives
\[
 v_\pm(X,p)=\pm\gamma(p)+Xr_\pm(X,p),
 \qquad H(X,p,v_\pm(X,p))=0.
\]
Using \eqref{eq:abstract-complementary-vanishing} in the Taylor expansion
of $Y$ shows that
\[
 Y\bigl(X^2,p,Xv_\pm(X,p)\bigr)=X^2y_\pm(X,p)
\]
for holomorphic $y_\pm$.  Substitution now gives
\eqref{eq:abstract-fold-branches}.

To prove the last assertion, put
\[
 \xi_\pm(X,p):=Xv_\pm(X,p).
\]
Differentiating \eqref{eq:abstract-reduced-fold} with respect to \(\xi\)
gives
\[
 \partial_\xi G\bigl(X^2,p,\xi_\pm(X,p)\bigr)
 =Xc_\pm(X,p),
\]
where \(c_\pm\) is holomorphic and
\[
 c_\pm(0,p)=\pm\gamma(p)B(p)\ne0.
\]
After one further shrink, \(c_\pm\) remains nonzero and the derivative
of the complementary equation with respect to \(y\) remains invertible.
Solving the differentiated complementary equation first then shows that
\[
 D_V\mathfrak F\bigl(X^2,p,V_\pm(X,p)\bigr)
\]
is invertible whenever \(X\ne0\).

It remains only to prove exhaustiveness.  By uniqueness in the
complementary equation, every sufficiently small zero has the preceding
representation for a unique kernel coordinate \(\xi\), and this
coordinate satisfies \eqref{eq:abstract-reduced-fold}.  After shrinking
the neighborhood, there are constants \(c,M>0\) such that
\[
 |G_2(p,\xi)|\ge c,
 \qquad
 |G_1(s,p,\xi)|\le M.
\]
Consequently, with \(s=X^2\),
\eqref{eq:abstract-reduced-fold} gives
\[
 |\xi|\le C|X|,
 \qquad C:=\sqrt{M/c}.
\]
For \(X\ne0\), put \(v=\xi/X\).  Then \(|v|\le C\) and
\(H(X,p,v)=0\).

At the base point,
\[
 H(0,0,v)=A(0)+\frac12B(0)v^2,
\]
whose only zeros are the two simple roots \(\pm\gamma_0\).  Choose
\[
 R>\max\{C,|\gamma_0|\}.
\]
The base polynomial does not vanish on \(|v|=R\), and hence
\[
 m_R:=\min_{|v|=R}|H(0,0,v)|>0.
\]
Joint continuity of \(H\), uniformly on this compact circle, allows us
to shrink the \((X,p)\)-neighborhood so that
\[
 \sup_{|v|=R}
 |H(X,p,v)-H(0,0,v)|<m_R.
\]
Rouch\'e's theorem therefore shows that
\(H(X,p,\cdot)\) has exactly two zeros in \(|v|<R\), counted with
multiplicity.  The scalar implicit function theorem has already
produced the two simple zeros \(v_\pm(X,p)\) in this disk, so they
exhaust all its zeros.  Since every sufficiently small solution
satisfies \(|v|\le C<R\), it belongs to one of the two branches
constructed above.  When \(X=0\),
\eqref{eq:abstract-reduced-fold} and the nonvanishing of \(G_2\) give
\(\xi=0\).  This proves local exhaustiveness.

Finally, suppose that these objects commute with conjugation and that
\(X,p\) are real.  Conjugating the complementary equation shows that
\(\overline{Y(s,p,\xi)}\) is another solution with the same real parameters;
uniqueness in the complementary implicit function theorem therefore
gives
\[
 \overline{Y(s,p,\xi)}=Y(s,p,\xi)
\]
for real \(s,p,\xi\).  Similarly,
\(\overline{v_\pm(X,p)}\) solves the scalar reduced equation and remains
near the same real base root \(\pm\gamma_0\); uniqueness in the scalar
implicit function theorem gives
\[
 \overline{v_\pm(X,p)}=v_\pm(X,p).
\]
Thus the corresponding branches \(V_\pm(X,p)\) are real for real
\((X,p)\).
\end{proof}

\begin{lemma}[Uniform Banach space valued big-\(O\) transfer]
\label{lem:uniform-Banach-Delta-transfer}
Let \(\mathscr Y\) be a complex Banach space.  Let \(\Delta_+\) be a
dented domain at \(1\) of the form
\eqref{eq:base-dented-domain}.  By decreasing its outer radius and
increasing its dent angle, choose a smaller dented domain
\(\Delta\). 
Let
\[
 F_\lambda:\Delta_+\longrightarrow\mathscr Y,
 \qquad \lambda\in\Lambda,
\]
be a family of holomorphic maps, where \(\Lambda\) is an arbitrary
index set.  Fix \(\delta>0\) sufficiently small and put
\[
 \Omega
 :=
 \Delta\cap\{|1-w|>\delta\}.
\]
Suppose that, for some \(\nu\in\mathbb R\),
\begin{equation}\label{eq:uniform-transfer-hypotheses}
 \sup_{\substack{\lambda\in\Lambda\\w\in\Omega}}
 \|F_\lambda(w)\|_{\mathscr Y}<\infty,
 \qquad
 \sup_{\substack{\lambda\in\Lambda,\ w\in\Delta\\
                  0<|1-w|<2\delta}}
 \frac{\|F_\lambda(w)\|_{\mathscr Y}}
      {|1-w|^\nu}<\infty.
\end{equation}
Then
\begin{equation}\label{eq:uniform-Banach-big-O-transfer}
 \sup_{\lambda\in\Lambda}
 \bigl\|[w^k]F_\lambda(w)\bigr\|_{\mathscr Y}
 =
 O(k^{-\nu-1}).
\end{equation}
\end{lemma}

\begin{proof}
This is the uniform in the parameter Banach space valued form of the standard
big-\(O\) transfer theorem; see
\cite[Theorem~VI.3 and Lemma~IX.2]{FlajoletSedgewick2009}.
For the standard indented coefficient contour
\(\Gamma_k\subset\Delta\), Banach space valued Cauchy's formula gives
\[
 [w^k]F_\lambda(w)
 =
 \frac{1}{2\pi i}
 \int_{\Gamma_k}
 F_\lambda(w)\frac{\dd w}{w^{k+1}},
\]
where the integral is understood in the Bochner sense.  Hence
\[
 \bigl\|[w^k]F_\lambda(w)\bigr\|_{\mathscr Y}
 \le
 \frac{1}{2\pi}
 \int_{\Gamma_k}
 \|F_\lambda(w)\|_{\mathscr Y}
 \frac{|\dd w|}{|w|^{k+1}}.
\]
The first bound in
\eqref{eq:uniform-transfer-hypotheses} makes the outer contribution
exponentially small, while the second gives
\(O(k^{-\nu-1})\) on the local part.  Both estimates are uniform in
\(\lambda\), proving
\eqref{eq:uniform-Banach-big-O-transfer}.
\end{proof}

For later use, we also state the exact scalar identities
\begin{align}
 [w^k](1-w)^\lambda
 &=
 \frac{\Gamma(k-\lambda)}
      {\Gamma(-\lambda)\Gamma(k+1)}
 =
 \frac{k^{-\lambda-1}}{\Gamma(-\lambda)}
 \{1+O_\lambda(k^{-1})\},
 \qquad
 \lambda\notin\mathbb Z_{\ge0},
 \label{eq:exact-power-coefficient}\\
 [w^k]\log(1-w)^{-1}
 &=\frac1k,
 \qquad k\ge1.
 \label{eq:exact-log-coefficient}
\end{align}
If \(m\in\mathbb Z_{\ge0}\), then \((1-w)^m\) is a polynomial and its
\(k\)-th coefficient vanishes for \(k>m\).  Moreover, on
\(\Omega\), the quantity \(|1-w|\) is bounded above and away from
zero.  Thus subtracting finitely many displayed power, polynomial, or
logarithmic models with uniformly bounded coefficients preserves the
outer bound in \eqref{eq:uniform-transfer-hypotheses}.

\section{Spectral bridge for continuous kernels}
\label{app:continuous-kernel-HS}
The nonlinear tree equations are naturally solved in \(C(S)\), whereas
the spectral poles, cyclic traces, and regularized determinants are
analyzed on \(L^2(\mu)\).  This appendix relates the two
continuous kernel realizations and states the rank one spectral
calculus used at the Perron singularity.
Let $S$ be a compact metric space, let $\mu$ be a Borel probability
measure with full support, and put
\[
 E=C(S),\qquad H=L^2(\mu),
\]
with canonical inclusion $J:E\hookrightarrow H$.  For
$a\in C(S^2)$, let $A_E$ and $A_H$ be the two realizations of
\[
 (Af)(x)=\int_Sa(x,y)f(y)\,\mu(\dd y).
\]

\begin{lemma}[Continuous kernel realizations and cyclic traces]
\label{lem:continuous-kernel-realizations}
Define $\mathcal A_a:H\to E$ by the same integral formula.  Then
\begin{equation}\label{eq:abstract-kernel-factorization}
 A_H=J\mathcal A_a,
 \qquad A_E=\mathcal A_aJ,
 \qquad
 \|\mathcal A_a\|_{H\to E}\le\|a\|_\infty,
 \qquad
 \|A_H\|_{\mathfrak S_2}=\|a\|_{L^2(\mu^{\otimes2})}.
\end{equation}
In particular, $A_H$ is Hilbert--Schmidt and $A_E$ is compact.  Their
nonzero spectra agree, including algebraic multiplicities, and their
Riesz projections are intertwined by $J$.  In particular, for nonzero
$\zeta$ in the common resolvent set,
\begin{equation}\label{eq:abstract-resolvent-intertwining}
 (\zeta I_H-A_H)^{-1}J=J(\zeta I_E-A_E)^{-1}.
\end{equation}

If $a_1,\ldots,a_m\in C(S^2)$, $m\ge2$, and $A_j$ are their
realizations on $H$, then $A_1\cdots A_m$ is trace class and
\begin{equation}\label{eq:abstract-cyclic-trace}
 \operatorname{Tr}_H(A_1\cdots A_m)
 =\int_{S^m}a_1(x_1,x_2)\cdots a_m(x_m,x_1)
   \prod_{j=1}^m\mu(\dd x_j).
\end{equation}
Finally, if the kernels depend holomorphically in $C(S^2)$, then their
realizations depend holomorphically in $\mathcal K(E)$ and
$\mathfrak S_2(H)$, and so do the cyclic traces.
\end{lemma}

\begin{proof}
Cauchy--Schwarz proves \eqref{eq:abstract-kernel-factorization}; uniform
continuity of $a$ makes $\mathcal A_af$ continuous, and
Arzel\`a--Ascoli makes $A_E$ compact.  The $AB$--$BA$ resolvent identity
applied to $J$ and $\mathcal A_a$ gives equality of the nonzero spectra
and \eqref{eq:abstract-resolvent-intertwining}; contour integration gives
the assertion about Riesz projections.  For $\lambda\ne0$, every
generalized $\lambda$-eigenvector of $A_H$ lies in
$\operatorname{ran}A_H\subset J(E)$, and the two intertwining maps
identify the generalized eigenspaces; hence algebraic multiplicities
also agree.  Formula
\eqref{eq:abstract-cyclic-trace} is immediate for tensor of finite rank
kernels.  Uniform density of $C(S)\otimes C(S)$ in $C(S^2)$ and Schatten
H\"older extend it to continuous kernels.  The final assertion follows
because the maps from $C(S^2)$ into $\mathcal K(E)$ and
$\mathfrak S_2(H)$ are continuous linear, while the trace in
\eqref{eq:abstract-cyclic-trace} is continuous multilinear.
\end{proof}

\begin{lemma}[Analytic continuous-kernel rank-one splitting]
\label{lem:continuous-kernel-rank-one-calculus}
Let $\Omega$ be open in a complex Banach space, let $q_0\in\Omega$, and
let $q\mapsto b(q)\in C(S^2)$ be holomorphic.  Let \(B_E(q)\) and \(B_H(q)\) denote the realizations of
\(b(q)\) on \(E\) and \(H\), respectively. Suppose that \(1\) is an algebraically simple eigenvalue of
\(B_H(q_0)\).
After shrinking to a neighborhood $V$ of $q_0$, there are a holomorphic
simple eigenvalue $\lambda(q)$ and compatible rank one Riesz projections
$\mathsf P_E(q),\mathsf P_H(q)$.  Moreover, $\mathsf P_H(q)$ has a
holomorphic kernel $p(q)\in C(S^2)$.

Put
\[
 C_H=B_H-\lambda\mathsf P_H,
 \qquad C_E=B_E-\lambda\mathsf P_E.
\]
Then $C_H$ and $C_E$ have a common holomorphic kernel
$c(q)\in C(S^2)$, and both $I-C_H$ and $I-C_E$ are invertible on $V$.
More generally, if $q\mapsto d(q)\in C(S^2)$ is holomorphic and
$D_E,D_H$ are its realizations, then
$(I-C_E)^{-1}D_E$ and $(I-C_H)^{-1}D_H$ have one common
holomorphic kernel.  In particular, the complementary resolvent
$C_H(I-C_H)^{-1}$ has such a kernel $r_\perp(q)\in C(S^2)$.  Whenever
$\lambda(q)\ne1$,
\begin{equation}\label{eq:abstract-rank-one-resolvent}
 B_H(I-B_H)^{-1}
 =\frac{\lambda}{1-\lambda}\mathsf P_H+C_H(I-C_H)^{-1}.
\end{equation}
All the maps into continuous kernel spaces above are bounded after passage to a smaller
neighborhood of $q_0$.
\end{lemma}

\begin{proof}
The usual construction by a Riesz contour for an isolated spectral point
\cite[Chapter~VII, Section~1]{Kato1995} gives $\lambda$ and
$\mathsf P_H$ holomorphically.  Choose the contour $\Gamma$ with zero
winding number about the origin.  The resolvent identity
\[
 (\zeta I-B_H)^{-1}
 =\zeta^{-1}I+\zeta^{-2}B_H
   +\zeta^{-2}B_H(\zeta I-B_H)^{-1}B_H
\]
then gives
\[
 \mathsf P_H
 =\frac1{2\pi i}\int_\Gamma
   \zeta^{-2}B_H(\zeta I-B_H)^{-1}B_H\,\dd\zeta.
\]
For continuous kernels $d,e$ and $R\in\mathcal B(H)$, the operator
$D_HRE_H$ has continuous kernel
\[
 (x,y)\longmapsto
 \int_Sd(x,u)\,[R(e(\cdot,y))](u)\,\mu(\dd u),
\]
and this is a continuous trilinear map of $(d,R,e)$ into $C(S^2)$.
The contour formula therefore supplies the asserted holomorphic kernel
$p(q)$.  The preceding lemma identifies its realization on $E$ with
$\mathsf P_E$.

Consequently $C_H$ and $C_E$ have the common kernel
$c=b-\lambda p$.  The point one is outside the spectrum of both
complementary operators at $q_0$, so openness of invertibility and
\eqref{eq:abstract-resolvent-intertwining} give holomorphic compatible
inverses after shrinking $V$.  If $G_E=(I-C_E)^{-1}$ and $d(q)$ is as
in the statement, then
\[
 r_d(q;x,y):=[G_E(q)d(q;\cdot,y)](x)
\]
is a holomorphic continuous kernel.  Intertwining shows that it
represents $(I-C_H)^{-1}D_H$ as well.  Taking $d=c$ gives $r_\perp$
and represents $(I-C_H)^{-1}C_H=C_H(I-C_H)^{-1}$.  Finally,
$B_H=\lambda\mathsf P_H+C_H$ with vanishing cross products, which gives
\eqref{eq:abstract-rank-one-resolvent}.  Local boundedness follows by
one final shrink.
\end{proof}

\section{Criteria and examples for projective quadratic regularity}
\label{sec:QP-examples}

The compact type result in \Cref{thm:local-main} assumes the
projective quadratic regularity condition \Cref{ass:QP}.  This appendix
gives direct criteria for checking that condition from feature,
Fourier, polynomial, and radial representations.  The assumption is
placed on \(\kappa\), while \Cref{lem:QP-correction-kernel} transfers
it to the canonical kernel \(\overline H\) used in
\eqref{eq:Psi-QP-identity}.

Every continuous kernel on the compact space \(S\) induces a
Hilbert--Schmidt operator.  Projective regularity gives the stronger
trace class conclusion.  Indeed, if
\(h=\sum_j a_jf_j\otimes f_j\) satisfies
\eqref{eq:projectiveexpansion}, then its integral operator
\(T_h\) on \(L^2(\mu)\) is trace class and
\[
 \|T_h\|_1
 \le \sum_j|a_j|\|f_j\|_2^2
 \le \sum_j|a_j|\|f_j\|_\infty^2.
\]
Trace class gives only an \(L^2(\mu)\) nuclear decomposition, whereas
the profile estimates require a decomposition by continuous functions
with summable uniform norm costs.  At the same time, \QP{} is weaker
than summability in a continuous orthonormal eigenbasis;
\Cref{ex:QP-strictly-weaker} separates the two conditions.  If a
nonnegative kernel below is strictly positive pointwise,
division by its operator norm preserves
\(\mathscr P_{\mathrm s}(S)\), enforces \textup{(B1)}, and gives
\textup{(B2)} with \(m=1\).

\subsection{The canonical correction kernel}

Before turning to concrete criteria, we state the equivalence between
projective regularity of the connection kernel and that of the
canonical correction kernel used in the component calculation.
\begin{lemma}[Equivalent kernel formulation of \textup{(QP)}]
\label{lem:QP-correction-kernel}
For \(z\in S\), put \(\kappa_z=\kappa(\cdot,z)\in C(S)\), and recall
\(\kappa_2\) and \(\overline H\) from \Cref{lem:profile}\textup{(iv)}.

Then \(\kappa_2\in\mathscr P_{\mathrm s}(S)\) and
\begin{equation}\label{eq:kappa2-projective-bound}
 \|\kappa_2\|_{\pi,\mathrm s}
 \le \int_S\|\kappa_z\|_\infty^2\,\mu(\dd z)
 \le \|\kappa\|_\infty^2.
\end{equation}
Moreover,
\begin{equation}\label{eq:Hbar-projective-identity}
 \overline H=\kappa_2-\kappa-(d-\1)\otimes(d-\1).
\end{equation}
Consequently,
\begin{equation}\label{eq:QP-kappa-equivalence}
 \overline H\in\mathscr P_{\mathrm s}(S)
 \quad\Longleftrightarrow\quad
 \kappa\in\mathscr P_{\mathrm s}(S).
\end{equation}
More quantitatively, with
\(M_\kappa=\|\kappa\|_\infty^2+\|d-\1\|_\infty^2\),
\begin{equation}\label{eq:QP-norm-comparison}
 \|\overline H\|_{\pi,\mathrm s}
 \le \|\kappa\|_{\pi,\mathrm s}+M_\kappa,
 \qquad
 \|\kappa\|_{\pi,\mathrm s}
 \le \|\overline H\|_{\pi,\mathrm s}+M_\kappa.
\end{equation}
\end{lemma}

\begin{proof}
The map \(z\mapsto\kappa_z\) is continuous from the compact space
\(S\) into \(C(S)\).  The quadratic map \(f\mapsto f\otimes f\) is
continuous in the symmetric projective norm; polarization gives
\[
 \|f\otimes f-g\otimes g\|_{\pi,\mathrm s}
 \le2\|f-g\|_\infty(\|f\|_\infty+\|g\|_\infty).
\]
Thus \(z\mapsto\kappa_z\otimes\kappa_z\) is Bochner integrable in the
symmetric projective tensor product.  Its integral has kernel
\(\kappa_2\), and the triangle inequality gives
\eqref{eq:kappa2-projective-bound}.  Expanding \(-(d-\1)\otimes(d-\1)\) gives the definition of
\(\overline H\) in \Cref{lem:profile}\textup{(iv)}, proving
\eqref{eq:Hbar-projective-identity}.  Since the other two terms in
that identity belong to \(\mathscr P_{\mathrm s}(S)\), the equivalence
and the two norm bounds follow by rearrangement and the triangle
inequality.
\end{proof}

\subsection{Feature representations}

The two examples distinguish discrete feature expansions from integral
representations.

\begin{example}[Discrete feature expansions]
\label{ex:QP-separable}
\emph{Finite features.}
Let \(g_1,\ldots,g_m\in C(S)\), put
\[
 g(x)=(g_1(x),\ldots,g_m(x))^{\mathsf T},
\]
and suppose
\[
 \kappa(x,y)=g(x)^{\mathsf T}A g(y)
\]
for a real symmetric \(m\times m\) matrix \(A\).  Choose an orthogonal
matrix \(U\) such that
\(A=U^{\mathsf T}\diag(\lambda_1,\ldots,\lambda_m)U\), and put
\(f_j(x)=(Ug(x))_j\).  Then
\[
 \kappa(x,y)=\sum_{j=1}^m\lambda_j f_j(x)f_j(y),
 \qquad
 \|\kappa\|_{\pi,\mathrm s}
 \le\sum_{j=1}^m|\lambda_j|\|f_j\|_\infty^2<\infty.
\]
Hence every kernel of this form satisfies \QP.  The diagonalization
concerns only the finite matrix \(A\), not the integral operator.

\emph{Nonorthogonal series.}
More generally, suppose that a symmetric kernel has a possibly
nonorthogonal expansion
\[
 \kappa(x,y)=\sum_{j\ge1}u_j(x)v_j(y),
 \qquad
 \sum_{j\ge1}\|u_j\|_\infty\|v_j\|_\infty<\infty.
\]
The series converges absolutely and uniformly.  Symmetry gives
\[
 \kappa=\frac12\sum_j(u_j\otimes v_j+v_j\otimes u_j).
\]
After omitting zero terms, rescale each pair by
\[
 \widetilde u_j
 =\left(\frac{\|v_j\|_\infty}{\|u_j\|_\infty}\right)^{1/2}u_j,
 \qquad
 \widetilde v_j
 =\left(\frac{\|u_j\|_\infty}{\|v_j\|_\infty}\right)^{1/2}v_j.
\]
Both rescaled functions have norm
\(\{\|u_j\|_\infty\|v_j\|_\infty\}^{1/2}\), and polarization yields
\[
 \frac12(u_j\otimes v_j+v_j\otimes u_j)
 =\frac14\left\{
   (\widetilde u_j+\widetilde v_j)^{\otimes2}
   -(\widetilde u_j-\widetilde v_j)^{\otimes2}
  \right\}.
\]
The projective cost is at most
\(2\|u_j\|_\infty\|v_j\|_\infty\), so summation proves
\(\kappa\in\mathscr P_{\mathrm s}(S)\).  Thus neither orthogonality nor
an eigenfunction expansion is required.

\emph{Spectral series.}
A familiar special case is an eigenfunction expansion.  Let \(K\) be
the integral operator associated with \(\kappa\), as in
\eqref{eq:integral-operator}, and suppose that this compact self-adjoint
operator has continuous orthonormal eigenfunctions \((\phi_j)\) with
\[
 \kappa(x,y)=\sum_{j\ge1}\theta_j\phi_j(x)\phi_j(y),
 \qquad
 \sum_{j\ge1}|\theta_j|\|\phi_j\|_\infty^2<\infty.
\]
The displayed series is a projective representation, so \QP{} holds.
We call this sufficient condition \emph{spectral uniform norm
summability}.  It is convenient when the eigenbasis is explicit, but
\Cref{ex:QP-strictly-weaker} shows that it is stronger than \QP{}
itself.
\end{example}

\begin{example}[Integral feature representations]
\label{ex:QP-integral-features}
Let \((\Omega,\nu)\) be a measure space, let \(a:\Omega\to\R\) be
measurable, and let \(\omega\mapsto f_\omega\in C(S)\) be strongly
measurable.  Suppose
\[
 \int_\Omega |a(\omega)|\|f_\omega\|_\infty^2\,\nu(\dd\omega)<\infty
\]
and define
\[
 \kappa(x,y)=\int_\Omega a(\omega)
       f_\omega(x)f_\omega(y)\,\nu(\dd\omega).
\]
The integrand is Bochner integrable in the symmetric projective tensor
product.  Approximation by simple functions gives
\[
 \kappa\in\mathscr P_{\mathrm s}(S),
 \qquad
 \|\kappa\|_{\pi,\mathrm s}
 \le\int_\Omega |a(\omega)|\|f_\omega\|_\infty^2\,\nu(\dd\omega).
\]
When \(\nu\) is a probability measure and \(a=1\), write
\(F=f_\omega\).  Then
\(\kappa(x,y)=\E[F(x)F(y)]\) satisfies \QP{} whenever
\(\E\|F\|_\infty^2<\infty\).  This expectation only represents the
deterministic kernel \(\kappa\) and introduces no additional randomness
into the graph model.  These conclusions concern projective regularity;
use as a random graph connection kernel additionally requires
\(\kappa(x,y)\ge0\) pointwise.
\end{example}

\subsection{Structured kernels}

Fourier and power series expansions give more concrete tests for
stationary and Euclidean kernels.

\begin{example}[Translation invariant kernels on a torus]
\label{ex:QP-torus}
Let \(S=\mathbb T^d=\R^d/\Z^d\), let \(\mu\) be Haar probability, and
let
\[
 \kappa(x,y)=q(x-y)
\]
for a real continuous even function \(q\).  Then \QP{} is
equivalent to absolute summability of the Fourier coefficients.  With
\[
 \widehat q(m)=\int_{\mathbb T^d}q(x)e^{-2\pi i m\cdot x}\,\dd x,
 \qquad m\in\mathbb Z^d,
\]
one has
\begin{equation}\label{eq:torus-QP-Wiener}
 \kappa\in\mathscr P_{\mathrm s}(\mathbb T^d)
 \quad\Longleftrightarrow\quad
 \sum_{m\in\Z^d}|\widehat q(m)|<\infty.
\end{equation}
For sufficiency, let \(\Z^d_+\) contain one representative from each
pair \(\{m,-m\}\), \(m\ne0\).  Absolute summability gives the uniformly
convergent expansion below; its coefficients are real because \(q\) is
real and even:
\[
 q(x-y)=\widehat q(0)+2\sum_{m\in\Z^d_+}\widehat q(m)
 \{c_m(x)c_m(y)+s_m(x)s_m(y)\},
\]
where \(c_m(x)=\cos(2\pi m\cdot x)\) and
\(s_m(x)=\sin(2\pi m\cdot x)\).  This is a representation of the form
\eqref{eq:projectiveexpansion}.

Conversely, for any representation
\(\kappa=\sum_j a_jf_j\otimes f_j\) satisfying
\eqref{eq:projectiveexpansion}, termwise Fourier integration and
Parseval's identity give
\[
 \widehat q(m)=\sum_j a_j|\widehat f_j(m)|^2,
 \qquad
 \sum_m|\widehat q(m)|
 \le\sum_j|a_j|\|f_j\|_2^2
 \le\sum_j|a_j|\|f_j\|_\infty^2.
\]
This proves the converse implication in
\eqref{eq:torus-QP-Wiener}.  In particular, repeated integration by
parts shows that \(q\in C^r(\mathbb T^d)\), with integer \(r>d\), is
sufficient; real analytic periodic kernels are sufficient as well.  If
\(q>0\) pointwise, then
\(|\widehat q(m)|<\widehat q(0)\) for \(m\ne0\), so normalization by
\(\widehat q(0)\) gives \textup{(B1)--(B2)} with \(u=\1\).
\end{example}

\begin{example}[Analytic dot product and Gaussian kernels]
\label{ex:QPdotproduct}
Let \(S\) be a compact subset of \([-R,R]^d\), let \(b\in C(S)\), and
write \(x\cdot y=\sum_{\ell=1}^d x_\ell y_\ell\) for the Euclidean
inner product.  Let \(h\) be the one variable function
\[
 h(z)=\sum_{m\ge0}c_mz^m,
 \qquad
 \sum_{m\ge0}|c_m|(dR^2)^m<\infty.
\]
For \(\kappa(x,y)=b(x)b(y)h(x\cdot y)\), the multinomial formula gives
\[
 \kappa(x,y)=\sum_{m\ge0}\sum_{|\alpha|=m}
 c_m\frac{m!}{\alpha!}
 \{b(x)x^\alpha\}\{b(y)y^\alpha\}.
\]
Its projective norm satisfies
\[
 \|\kappa\|_{\pi,\mathrm s}
 \le \|b\|_\infty^2\sum_{m\ge0}|c_m|R^{2m}
       \sum_{|\alpha|=m}\frac{m!}{\alpha!}
 =\|b\|_\infty^2
   \sum_{m\ge0}|c_m|(dR^2)^m<\infty.
\]
Taking \(b=\1\) gives polynomial kernels and analytic dot product
kernels with the displayed coefficient bound.  For the Gaussian radial basis kernel with \(\gamma>0\), take
\[
 b(x)=e^{-\gamma\|x\|^2},
 \qquad h(z)=e^{2\gamma z},
\]
to obtain \(\kappa(x,y)=e^{-\gamma\|x-y\|^2}\) and the explicit bound
\[
 \|\kappa\|_{\pi,\mathrm s}\le e^{2\gamma dR^2}.
\]
\end{example}

\subsection{Comparison with spectral summability}

The final example shows that the spectral uniform norm summability in
\Cref{ex:QP-separable} is sufficient but not necessary for \QP.

\begin{example}[\QP{} is strictly weaker than spectral
summability]\label{ex:QP-strictly-weaker}
Take \(S=[0,1]\) with Lebesgue measure.  Choose
\(\psi\in C_c((0,1))\) with
\(\int_0^1\psi=0\), \(\int_0^1\psi^2=1\), and put
\(M=\|\psi\|_\infty\).  Let
\(\ell_j=2^{-j-1}\) and
\(I_j=(2^{-j},2^{-j}+\ell_j)\); these intervals are disjoint.  Define
\[
 \phi_j(x)=\ell_j^{-1/2}
   \psi\left(\frac{x-2^{-j}}{\ell_j}\right)
\]
where \(\psi\) is extended by zero outside \((0,1)\).  The functions
\((\phi_j)\) are continuous, orthonormal, have mean zero, and satisfy
\(\|\phi_j\|_\infty=M\ell_j^{-1/2}\).  Set
\[
 \theta_j=\frac{\ell_j}{jM^2},
 \qquad
 q=\sum_{j\ge1}\theta_j\phi_j\otimes\phi_j,
 \qquad
 F_\varepsilon=\sum_{j\ge1}\sqrt{\theta_j}\,
                    \varepsilon_j\phi_j,
\]
where \((\varepsilon_j)\) are independent Rademacher signs.  Disjoint
supports and
\(\theta_j\|\phi_j\|_\infty^2=1/j\) give
\[
 \left\|\sum_{j>N}\theta_j\phi_j\otimes\phi_j\right\|_\infty
 \le\frac1{N+1},
 \qquad
 \sup_\varepsilon\left\|
 \sum_{j>N}\sqrt{\theta_j}\varepsilon_j\phi_j\right\|_\infty
 \le\frac1{\sqrt{N+1}}.
\]
Thus \(q\) is continuous.  Moreover, \(F_\varepsilon\) is a measurable
random element of \(C([0,1])\), obtained as a uniform limit, with
\(\|F_\varepsilon\|_\infty\le1\).  Independence gives
\(q(x,y)=\E_\varepsilon[F_\varepsilon(x)F_\varepsilon(y)]\), so
\Cref{ex:QP-integral-features} implies
\(q\in\mathscr P_{\mathrm s}([0,1])\).

Choose \(c>\max\{\|q\|_\infty,\theta_1\}\) and put
\(\kappa(x,y)=1+c^{-1}q(x,y)\).  This kernel is strictly positive
pointwise.  Its integral operator has eigenvalue \(1\) on \(\1\), eigenvalue
\(\theta_j/c\) on \(\phi_j\), and zero on the orthogonal complement.
Hence it satisfies \textup{(B1)--(B2)} and \QP, whereas
\[
 \sum_{j\ge1}\frac{\theta_j}{c}\|\phi_j\|_\infty^2
 =\frac1c\sum_{j\ge1}\frac1j=\infty.
\]
The sequence \((\theta_j)\) is strictly decreasing, so no change of
eigenbasis removes the divergence.  Thus \QP{} is strictly weaker than
spectral uniform norm summability.
\end{example}

\section{Point process convergence tools}
\label{app:point-process-tools}

This appendix first proves the factorial density identification in
\Cref{prop:AW-density} and then gives the two point process tools used in
\Cref{sec:factorial-process,sec:countable}.  The criterion in
\Cref{lem:factorial-criterion} obtains an unmarked local limit from
factorial moment convergence, uniform falling factorial bounds, and
moment determinacy; its \((q+1)\)-st bound supplies uniform
integrability at degree \(q\).  \Cref{lem:mark-replacement} adds a
deterministic mark when, with probability tending to one, no ground
point in a fixed local window carries a mark separated from the
prescribed value.  For background, see
Kallenberg~\cite[Chapter~4]{Kallenberg2017} and Daley and
Vere-Jones~\cite[Sections~5.4--5.5]{DaleyVereJones2003}.  We give the
proofs because the applications require localized formulations.

\subsection{Factorial densities of the Aldous process}

\begin{proof}[Proof of \Cref{prop:AW-density}]
The density formula \eqref{eq:AW-density} follows from the marked
excursion process of Janson and Spencer
\cite[(3.5), Theorems~3.1 and~4.1, Corollary~8.7, and
Remark~8.8]{JansonSpencer2007}.  To pass from their unmarked formula to
the surplus coordinates, put
\[
 F(x,\Lambda)=\frac{x^3}{6}-\frac{\Lambda x^2}{2}
                  +\frac{\Lambda^2x}{2},
 \qquad
 \Psi(s)=\sum_{r\ge0}w_rs^r.
\]
Their Theorem~3.1 and (3.5) give, conditionally on an atom of size
\(x\), independent surplus marks with distribution
\[
 \Prob\{N=r\mid x\}
   =\frac{w_rx^{3r/2}}{\Psi(x^{3/2})},
 \qquad r\in\Z_{\ge0}.
\]
Their Theorem~4.1 gives the one point intensity, including the factor
\(\Psi(x^{3/2})\).  Corollary~8.7, together with the mixed set
extension in Remark~8.8, yields the unmarked factorial moment density
as a product of successively shifted intensities.
Multiplication by the mark probabilities cancels the factors
\(\Psi(x_i^{3/2})\).  The remaining exponential factors combine by
\begin{equation}\label{eq:AW-exponent-telescoping}
 \sum_{i=1}^q
 F\left(x_i,\Lambda-\sum_{j<i}x_j\right)
 =F\left(\sum_{i=1}^qx_i,\Lambda\right),
\end{equation}
which follows by expanding both sides.  This proves
\eqref{eq:AW-density}.

For determinacy, write again \(S_q=\sum_{i=1}^q x_i\).  The total mass
of the \(q\)-th factorial moment measure on \(B_{a,b,R}^q\) is
\[
\sum_{r_1,\ldots,r_q=0}^R
\int_{[a,b]^q}
\exp\left\{-\frac{S_q^3}{6}+\frac{\Lambda S_q^2}{2}
-\frac{\Lambda^2S_q}{2}\right\}
\prod_{i=1}^q\frac{w_{r_i}}{\sqrt{2\pi}}\,
x_i^{3r_i/2-5/2}\,\dd x_1\cdots\dd x_q.
\]
It is at most
\[
 C^q\exp\{-c(qa)^3+C(qb)^2\}
\]
for constants \(c,C>0\) depending only on \(a,b,R,\Lambda\).  Indeed,
\(qa\le S_q\le qb\) on \([a,b]^q\), which gives the displayed
exponential bound, while the sums over \(r_i\le R\), the remaining
power factors, and the integrations contribute at most \(C^q\).

Fix a finite Borel partition of \(B_{a,b,R}\).  The mixed factorial
moments of the corresponding count vector satisfy this estimate at
total order \(q\), so their multivariate factorial moment series
converges absolutely for every complex value of its arguments.  The
standard relation between factorial moment measures and probability
generating functionals
\cite[Sections~5.4--5.5]{DaleyVereJones2003} shows that the series
determines the joint probability generating function and hence the
joint count distribution.  Count vectors on finite partitions from a
countable generating algebra determine a finite point measure on the
compact window; see, for example, \cite{Kallenberg2017}.  Thus the
factorial moment measures determine the restricted point process law.
\end{proof}

\subsection{A local factorial moment criterion}

\begin{lemma}[Local factorial moment criterion]
\label{lem:factorial-criterion}
Let \(B\) be a compact metric space and let \(\xi_n\) be finite point
processes on \(B\).  Here \((v)_q\) is the falling factorial defined in
\eqref{eq:fallingfactorialdefinition}.  Suppose:
\begin{enumerate}[label=\textup{(\roman*)}]
\item for every fixed \(q\ge1\),
      the factorial moment measure \(\mathfrak a_{\xi_n}^{(q)}\) from
      \eqref{eq:factorial-measure} converges weakly on \(B^q\) to a
      finite Borel measure \(\mathfrak a^{(q)}\);
\item for every fixed \(q\ge1\),
      \(\sup_n\E(\xi_n(B))_q<\infty\);
\item there is a countable algebra \(\mathcal R\) of subsets of \(B\),
      containing \(B\) and generating the Borel \(\sigma\)-field on
      \(B\), such that, for every \(q\ge1\) and
      \(A_1,\ldots,A_q\in\mathcal R\), the product
      \(A_1\times\cdots\times A_q\) is a continuity set for
      \(\mathfrak a^{(q)}\).  In addition, for every finite pairwise
      disjoint family in \(\mathcal R\), the mixed factorial moments
      obtained from \((\mathfrak a^{(q)})_{q\ge1}\) determine a unique
      joint count distribution.
\end{enumerate}
Then there is a unique finite point process \(\xi\) with factorial
moment measures \((\mathfrak a^{(q)})_{q\ge1}\), and
\(\xi_n\Longrightarrow\xi\) in \(\Mp(B)\) equipped with the weak
topology.
\end{lemma}

\begin{proof}
Put \(N_n=\xi_n(B)\).  Condition~\textup{(ii)} with \(q=1\) and
Markov's inequality show that \((N_n)\) is tight.  Since \(B\) is compact, each
set \(\{\eta\in\Mp(B):\eta(B)\le H\}\) is weakly compact; hence the
laws of \(\xi_n\) are tight.  Fix a convergent subsequence, still
indexed by \(n\), and write \(\xi_n\Longrightarrow\xi\).

For \(f\in C(B)\), the continuous mapping theorem gives
\(\int f\,\dd\xi_n\Rightarrow\int f\,\dd\xi\).  Moreover,
\(N_n^2=(N_n)_2+N_n\), so condition~\textup{(ii)} with \(q=1,2\)
makes these integrals uniformly integrable.  Condition~\textup{(i)}
with \(q=1\) therefore yields
\[
 \int_B f\,\dd\mathfrak a_\xi^{(1)}
 =\E\int_B f\,\dd\xi
 =\lim_n\E\int_B f\,\dd\xi_n
 =\int_B f\,\dd\mathfrak a^{(1)}.
\]
Thus \(\mathfrak a_\xi^{(1)}=\mathfrak a^{(1)}\).  For
\(A\in\mathcal R\), condition~\textup{(iii)} gives
\(\E\xi(\partial A)=\mathfrak a_\xi^{(1)}(\partial A)
=\mathfrak a^{(1)}(\partial A)=0\).  Hence
\(\xi(\partial A)=0\) almost surely, and the continuous mapping theorem
gives convergence of every finite count vector on \(\mathcal R\).

Fix \(q\ge1\) and \(C_1,\ldots,C_q\in\mathcal R\), and set
\[
 Y_n=\sum_{z_1,\ldots,z_q\in\xi_n}^{\ne}
      \prod_{i=1}^q\1_{C_i}(z_i),
 \qquad
 Y=\sum_{z_1,\ldots,z_q\in\xi}^{\ne}
      \prod_{i=1}^q\1_{C_i}(z_i),
\]
where the superscript \(\ne\) means that the ordered points are
distinct.  Taking the atoms generated by these sets within
\(\mathcal R\) expresses \(Y_n\) as a polynomial in a finite count
vector.  Hence \(Y_n\Rightarrow Y\).  Since \(0\le Y_n\le(N_n)_q\), for
every integer \(H>q\),
\[
 \sup_n\E\big[Y_n\1_{\{Y_n>(H)_q\}}\big]
 \le \sup_n\E\big[(N_n)_q\1_{\{N_n>H\}}\big]
 \le\frac{\sup_n\E(N_n)_{q+1}}{H-q}\longrightarrow0,
\]
because \((v)_{q+1}=(v)_q(v-q)\).  Hence \((Y_n)\) is uniformly
integrable.  Conditions~\textup{(i)} and~\textup{(iii)} now
give
\[
 \mathfrak a_\xi^{(q)}(C_1\times\cdots\times C_q)
 =\E Y=\lim_n\E Y_n
 =\mathfrak a^{(q)}(C_1\times\cdots\times C_q).
\]
These rectangles generate the Borel \(\sigma\)-field of \(B^q\), and
the measures are finite, so \(\mathfrak a_\xi^{(q)}=\mathfrak
a^{(q)}\) for every \(q\).

For each finite disjoint family in \(\mathcal R\), these identities
give the mixed factorial moments prescribed in~\textup{(iii)}, which
determine its joint count distribution.  Disjointization therefore
determines every finite dimensional count law on \(\mathcal R\).
Since \(\Mp(B)\) is a standard Borel space, the count maps on this
countable generating algebra determine its Borel \(\sigma\)-field and
therefore the point process law; see
\cite[Chapter~4]{Kallenberg2017}.  Every subsequential limit
consequently has the same law, so tightness gives convergence of the
full sequence.  The same argument proves uniqueness.
\end{proof}

We next pass from the ground process to deterministic marks.

\begin{lemma}[Replacement by a deterministic mark]
\label{lem:mark-replacement}
Let \((\mathsf M,d_{\mathsf M})\) be either a compact metric space or
the mark space \((\cP(\N),d_{\TV})\) defined in
\Cref{subsec:topology}.  Let \(p\in\mathsf M\) and let
\[
 \xi_n=\sum_i\delta_{z_{n,i}}\in\Mp(E),
 \qquad
 \widehat\xi_n=\sum_i\delta_{(z_{n,i},m_{n,i})}
\]
be point measures with the same ground points, where
\(m_{n,i}\in\mathsf M\).  Suppose that
\(\xi_n\Longrightarrow\xi\) vaguely and that, for every compact
\(K_E\subset E\) and every \(\varepsilon>0\),
\begin{equation}\label{eq:uniform-mark-replacement-condition}
 \Prob\{\text{some \(i\) has \(z_{n,i}\in K_E\) and
          \(d_{\mathsf M}(m_{n,i},p)>\varepsilon\)}\}\longrightarrow0.
\end{equation}
When \(\mathsf M\) is compact, regard \(\widehat\xi_n\) as an element
of \(\Mp(E\times\mathsf M)\) with the usual vague topology.  When
\(\mathsf M=\cP(\N)\), regard it as an element of
\(\cM_{\rm lm}(E\times\cP(\N))\) with the localized marked topology of
\Cref{subsec:topology}.  In the respective cases, define
\[
 J_p:\Mp(E)\longrightarrow\Mp(E\times\mathsf M)
 \quad\text{or}\quad
 J_p:\Mp(E)\longrightarrow
       \cM_{\rm lm}(E\times\cP(\N))
\]
by
\begin{equation}\label{eq:deterministic-marking-map}
 J_p\left(\sum_i\delta_{z_i}\right)
 =\sum_i\delta_{(z_i,p)}.
\end{equation}
Thus \(J_p\) retains the ground points and replaces all marks by
\(p\).  Then
\[
 \widehat\xi_n\Longrightarrow J_p(\xi)
\]
in the corresponding marked topology.
\end{lemma}

\begin{proof}
If \(f\) belongs to the test class defining the relevant marked
topology, then \(f_p(z):=f(z,p)\) lies in \(C_c(E)\) and
\(\int f\,\dd J_p(\eta)=\int f_p\,\dd\eta\).  Hence \(J_p\) is continuous
in either topology.  The continuous mapping theorem
\cite[Theorem~2.7]{Billingsley1999} yields
\begin{equation}\label{eq:deterministic-marking-CMT}
 J_p(\xi_n)\Longrightarrow J_p(\xi).
\end{equation}

Fix such an \(f\), and let the compact set \(K_E\subset E\) contain the
projection of \(\supp(f)\).  By continuity of \(f\) and compactness of
\(K_E\times\{p\}\), for every \(\delta>0\) there is \(\eta>0\) such
that
\begin{equation}\label{eq:mark-test-uniformity}
 \sup_{z\in K_E,\ d_{\mathsf M}(m,p)\le\eta}|f(z,m)-f(z,p)|<\delta.
\end{equation}
The sequence \((\xi_n(K_E))\) is tight.  Indeed, choose
\(0\le\varphi\in C_c(E)\) with \(\varphi\ge1\) on \(K_E\).  Then
\(\int\varphi\,\dd\xi_n\Longrightarrow\int\varphi\,\dd\xi\) and
\(\xi_n(K_E)\le\int\varphi\,\dd\xi_n\).

Let \(D_n(\eta)\) be the event in
\eqref{eq:uniform-mark-replacement-condition} with
\(\varepsilon=\eta\).  On \(D_n(\eta)^c\),
\eqref{eq:mark-test-uniformity} gives
\[
 \left|\int f\,\dd\widehat\xi_n-
             \int f\,\dd J_p(\xi_n)\right|
 \le\delta\xi_n(K_E).
\]
Given \(\rho,\gamma>0\), choose \(M\ge1\) such that
\(\sup_n\Prob\{\xi_n(K_E)>M\}<\gamma\), and use
\eqref{eq:mark-test-uniformity} with \(\delta=\rho/M\).  Then
\[
 \limsup_n\Prob\left\{
 \left|\int f\,\dd\widehat\xi_n-
             \int f\,\dd J_p(\xi_n)\right|>\rho\right\}
 \le\gamma,
\]
because \(\Prob(D_n(\eta))\to0\).  Letting \(\gamma\downarrow0\)
proves convergence to zero in probability for each test function.

Both marked topologies have a countable convergence determining test
family \((f_j)_{j\ge1}\); see \Cref{subsec:topology}.  We may therefore
use the compatible metric
\[
 d_{\rm mark}(\zeta,\zeta')=
 \sum_{j\ge1}2^{-j}\left(1\wedge
 \left|\int f_j\,\dd\zeta-\int f_j\,\dd\zeta'\right|\right).
\]
Applying the preceding estimate to finitely many terms and bounding
the remaining tail gives
\[
 d_{\rm mark}(\widehat\xi_n,J_p(\xi_n))\longrightarrow0
 \qquad\text{in probability}.
\]
The conclusion follows from
\eqref{eq:deterministic-marking-CMT} and the asymptotic equivalence
theorem \cite[Theorem~3.1]{Billingsley1999}.
\end{proof}

\section*{Use of AI and formal verification}
The research questions and overall proof strategy were formulated by the
authors, with the main approach building on earlier work by one of the
authors~\cite{Sun2026ER}. AI tools assisted with the development of technical details in
the proofs, manuscript drafting, and language editing. The main results (\Cref{thm:Perron-Wright,thm:local-main,thm:countable-realization,thm:countable-unrenormalized}) have
been formalized and checked in Lean 4. The formal proofs and instructions
for checking them are available at
\url{https://github.com/wentaohu20/PerronWrightFormalization}.
The authors have checked all mathematical details and rewritten the
proofs for clarity and readability. They take full responsibility for
the mathematical content and the final text.

\raggedbottom

\end{document}